 \documentclass[final,3p,times,10pt]{elsarticle}

\usepackage[dvipsnames]{xcolor}

\usepackage{graphicx}
\usepackage{epsfig}

\usepackage{amsmath}
\usepackage{amssymb}
\usepackage{amsthm}
\usepackage{amstext}

\usepackage{tikz}
\usetikzlibrary{matrix,calc}

\newtheorem{thm}{Theorem}
\newtheorem{lem}{Lemma}
\newtheorem{cor}{Corollary}
\newtheorem{ex}{Example}
\newtheorem{prop}{Proposition}
\newtheorem{rem}{Remark}

\newtheorem*{problem*}{Problem}
\newtheorem*{assumption*}{Assumption}
\newtheorem*{exs*}{Examples}

\usepackage{subcaption}

\usepackage{bbm}
\usepackage{bm}
\usepackage{pgfplots}
\pgfplotsset{compat=newest}

\usepackage{lineno}
\usepackage{multirow}

\usepackage{lipsum}
\usepackage{mathtools}
\usepackage{cancel}
\usepackage{enumerate}
\usepackage{setspace}

\usepackage{hyperref}

\usepackage{scalerel,stackengine}
\stackMath
\newcommand\reallywidehat[1]{%
	\savestack{\tmpbox}{\stretchto{%
			\scaleto{%
				\scalerel*[\widthof{\ensuremath{#1}}]{\kern-.6pt\bigwedge\kern-.6pt}%
				{\rule[-\textheight/2]{1ex}{\textheight}}
			}{\textheight}%
		}{0.75ex}}%
	\stackon[2pt]{#1}{\tmpbox}%
}

\journal{Inverse Problems}

\allowdisplaybreaks
\begin{document}

\onehalfspacing

\begin{frontmatter}

\title{Inversion of $\alpha$-sine and $\alpha$-cosine transforms on $\mathbb{R}$}

\author{Ly Viet Hoang}
\ead{ly.hoang@uni-ulm.de}
\author{Evgeny Spodarev}
\ead{evgeny.spodarev@uni-ulm.de}
\address{Ulm University}

%
%
%
%
%

\begin{abstract}
We consider the $\alpha$-sine transform of the form $T_\alpha f(y)=\int_0^\infty\vert\sin(xy)\vert^\alpha f(x)dx$ for $\alpha>-1$, where $f$ is an integrable function on $\mathbb{R}_+$. First, the inversion of this transform for $\alpha>1$ is discussed in the context of a more general family of integral transforms on the space of weighted, square-integrable functions on the positive real line. 
In an alternative approach, we show that the $\alpha$-sine transform of a function $f$ admits a series representation for all $\alpha>-1$, which involves the Fourier transform of $f$ and coefficients which can all be explicitly computed with the Gauss hypergeometric theorem. Based on this series representation we construct a system of linear equations whose solution is an approximation of the Fourier transform of $f$ at equidistant points. Sampling theory and Fourier inversion allow us to compute an estimate of $f$ from its $\alpha$-sine transform. The same approach can be extended to a similar $\alpha$-cosine transform on $\mathbb{R}_+$ for $\alpha>-1$, and the two-dimensional spherical $\alpha$-sine and cosine transforms for $\alpha>-1$, $\alpha\neq 0,2,4,\dots$. 
In an extensive numerical analysis, we consider a number of examples, and compare the inversion results of both methods presented. 


\end{abstract}


\end{frontmatter}


\nocite{gennady}
\nocite{glueckrothspodarev}
\nocite{approxinv}
\nocite{clasfour}
\nocite{bracewell,marks}
\nocite{zayed}
\nocite{brown}
\nocite{anastassiou}
\nocite{handbookconvex}
\nocite{stochgeomappl}

\section{Introduction}
\label{intro}

Spherical $\alpha$-sine and $\alpha$-cosine integral transforms and their inversion are of particular interest in stochastic geometry and tomography. For example when analyzing the structure of a fibrous material, the so-called rose of intersections is the spherical cosine transform of the directional distribution measure of the fibers \cite{goodey3, matheron, spodarev2}. In the context of convex geometry, the support function of a zonoid is the $\alpha$-cosine transform of some generating signed measure \cite{weil}.

Spherical $\alpha$-sine and $\alpha$-cosine transforms, and in particular the closely related spherical Radon transform, were extensively studied in the last few decades. Groemer \cite{groemer} and Helgason \cite{hel}  presented many results in the fields of integral geometry and convex analysis.
Further important work was done by Goodey and Weil \cite{goodey1,goodey2}, Mecke \cite{mecke1}, as well as Rubin \cite{invformulas,interseccosine} in light of fiber processes and stochastic geometry.

In this paper, contrary to the spherical transforms above, we consider the integral transforms on the positive real line. We study the solution of the integral equation
\begin{align*}
	g(y)=\int\limits_0^\infty\left\vert\sin\left(xy\right)\right\vert^\alpha f(x)dx,\quad y>0,
\end{align*}
for $\alpha>-1$, where the integral transform on the right-hand side is the aforementioned $\alpha$-sine transform. 

The inversion of this $\alpha$-sine transform is applicable in the context of stationary real harmonizable symmetric $\alpha$-stable random processes. These processes are uniquely determined by their so-called control measure. Furthermore, the codifference function describes their dependence structure. It is a generalization of the covariance function to the $\alpha$-stable case with $0<\alpha<2$, where second moments are infinite. 

Assuming this control measure has a density function $f$, which we refer to as \emph{spectral density}, with respect to the Lebesgue measure on $\mathbb{R}$, it can be shown the $\alpha$-sine transform of $f$ can be obtained from the codifference function. Estimation of the codifference function (and other parameters) as well as the subsequent inversion of the $\alpha$-sine transform yield the spectral density $f$, cf. Section 5.

Another application example is given in the case when	
$f$ is a $2\pi$-periodic functions on $\mathbb{R}$. Then, it suffices to consider the above $\alpha$-sine transform (or $\alpha$-cosine transform when the integral kernel is replaced by the cosine function) on the interval $[-\pi,\pi]$. This coincides with the two-dimensional spherical $\alpha$-sine and $\alpha$-cosine transforms on the unit circle. 

We will first introduce necessary notation and transformations in Section \ref{prelim}. In Section \ref{directapproach} we consider the inversion for $\alpha>1$ in the context of a more general class of integral transformations on the space of weighted $L^2$-functions. We will refer to this as the \emph{direct approach}.
An alternative, approximative approach, applicable to $\alpha>-1$ is given in Section \ref{Ffapprox}. This approach relies on the relation between the $\alpha$-sine transform and the classical Fourier transform, hence the name \emph{Fourier approximation approach}. Our approach can also be extended to $\alpha$-cosine transforms, which will be introduced later. Applications in the context of harmonizable symmetric $\alpha$-stable stochastic processes and two-dimensional spherical $\alpha$-sine and cosine transforms are outlined in Section \ref{appl}.
Lastly, numerical results of each approach are presented and discussed in Section \ref{num}.

\section{Preliminaries}
\label{prelim}

Denote by $\mathbb{R}_+$ the positive reals. Let $C^k(\mathbb{R}_+)$ be the space of $k$-times continuously differentiable functions on $\mathbb{R}_+$ with $C(\mathbb{R}_+)$ being the class of all continuous functions. Denote by $C_b(A)$ the space of bounded continuous functions on the interval $A\subseteq\mathbb{R}_+$.
We define the space of $p$-integrable functions on $\mathbb{R_+}$ with respect to the measure $w$ by $L^p(\mathbb{R}_+,w)$. 
Furthermore, we write $L^p(\mathbb{R}_+)=L^p(\mathbb{R}_+,dx)$ for the space of $p$-integrable functions with respect to the Lebesgue measure on $\mathbb{R}_+$ with
$L^p$-norm $\Vert f\Vert_p=\left(\int_0^\infty\vert f(x)\vert^pdx\right)^{1/p}$ for $f\in L^p(\mathbb{R}_+)$, $p\geq 1$. For $p=\infty$ this is the uniform norm $\Vert f\Vert_\infty=\text{ess}\sup_{x\in\mathbb{R}_+}\vert f(x)\vert$. We say that a sequence $(f_n)_{n\in\mathbb{N}}$ converges to some limit $f$ in $L^p$ if the $L^p$-distance $\Vert f_n-f\Vert_p$ converges to 0 as $n$ tends to infinity. 

Define the \emph{$\alpha$-sine transform} $T_\alpha $ for $\alpha>-1$ by 
\begin{align}
	\label{asinetrafo}
	T_\alpha f(y)=\int\limits_0^\infty\left\vert\sin\left(xy\right)\right\vert^\alpha f(x)dx,
\end{align} 
which is even in its argument, therefore, it suffices to consider $y\in\mathbb{R}_+$. 

We define the weighted function space $L^1_{w,\alpha}(\mathbb{R}_+)$ by 
\begin{align}
	\label{weightedLp}
	L^1_{w,\alpha}(\mathbb{R}_+)=\begin{cases}
		L^1(\mathbb{R}_+),&\alpha\geq0,\\
		L^1(\mathbb{R}_+)\cap L^1\left(\mathbb{R}_+,\max\left\{\frac{1}{x},1\right\}dx\right),&-1<\alpha<0.
	\end{cases}
\end{align}
\begin{lem}
	\label{lemma1}
	For any $f\in L^1_{w,\alpha}(\mathbb{R}_+)$ the function $T_\alpha f$ is well defined almost everywhere on $\mathbb{R}_+$. Additionally, it holds that $T_\alpha f(0)=0$ in case $\alpha>0$. 
\end{lem}
\begin{proof}
	Let $\alpha\geq 0$. Note that by the triangle inequality it holds that $\vert T_\alpha f(y)\vert\leq T_\alpha\vert f\vert (y)\leq\Vert f\Vert_1$ for all $y\in\mathbb{R}_+$. The relation $T\alpha f(0)=0$ is trivially satisfied for $\alpha>0$.
	For $-1<\alpha<0$ the finiteness of $\int\limits_0^K\vert T_\alpha f(y)\vert dy$ would imply that $T_\alpha f$ is finite almost everywhere on the interval $[0,K]$. Again, by the triangle inequality it holds that
	\begin{align*}
		\int\limits_0^K\vert T_\alpha f(y)\vert dy\leq\int\limits_0^KT_\alpha \vert f\vert(y)dy=\int\limits_0^K\left(\int\limits_0^\infty\left\vert\sin(xy)\right\vert^\alpha\vert f(x)\vert dx\right)dy.
	\end{align*}
Using Fubini's theorem we can further compute 
	\begin{align*}
		\int\limits_0^K\left(\int\limits_0^\infty\left\vert\sin(xy)\right\vert^\alpha\vert f(x)\vert dx\right)dy=\int\limits_0^\infty\left(\int\limits_0^K\vert\sin(xy)\vert^\alpha dy\right)\vert f(x)\vert dx=\int\limits_0^\infty\left(\frac{1}{x}\int\limits_0^{Kx}\vert\sin(u)\vert^\alpha du\right)\vert f(x)\vert dx,
	\end{align*}
	where the last equality stems from the substitution $u=xy$. Since $\vert\sin(u)\vert$ is $\pi$-periodic, we can estimate 
	\begin{align*}
		\frac{1}{x}\int\limits_0^{Kx}\vert\sin(u)\vert^\alpha du\leq\frac{1}{x}\left(\left\lfloor\frac{Kx}{\pi}\right\rfloor+1\right)\underbrace{\int\limits_0^\pi\vert\sin(u)\vert^\alpha du}_{\eqqcolon C_\alpha}\leq\frac{1}{x}\left(\frac{Kx}{\pi}+1\right)C_\alpha=\left(\frac{K}{\pi}+\frac{1}{x}\right)C_\alpha,
	\end{align*} 
	where the constant $C_\alpha$ is given by $C_\alpha=\sqrt{\pi}\frac{\Gamma\left(\frac{1+\alpha}{2}\right)}{\Gamma\left(1+\frac{\alpha}{2}\right)}$ for $\alpha>-1$. The above converges to $K/\pi$ as $x\rightarrow\infty$, and demanding $f\in L^1(\mathbb{R}_+)\cap L^1\left(\mathbb{R}_+,\max\left\{\frac{1}{x},1\right\}dx\right)$ ensures that $\int\limits_0^K\vert T_\alpha f(y)\vert dy<\infty$. Hence, 
	by the integrability of $T_\alpha f$ on $[0,K]$, it holds that $\vert T_\alpha f\vert$ is finite almost everywhere on $[0,K]$. Using the subadditivity of the Lebesgue measure, it follows that $T_\alpha f$ is finite almost everywhere on $\mathbb{R}_+$, i.e.
	\begin{align*}
		\mathcal{L}\left(\left\{y\in\mathbb{R}_+:\vert T_\alpha f(y)\vert=\infty\right\}\right)
		&=\mathcal{L}\left(\bigcup_{K\in\mathbb{N}}\left\{y\in[0,K]:\vert T_\alpha f(y)\vert=\infty\right\}\right)\leq \sum\limits_{K\in\mathbb{N}}\underbrace{\mathcal{L}\left(\left\{y\in[0,K]:\vert T_\alpha f(y)\vert=\infty\right\}\right)}_{=0,~K\in\mathbb{N}}=0,
	\end{align*}
	where $\mathcal{L}$ denotes the Lebesgue measure. 
\end{proof}

For $\alpha\geq0$, using the triangle inequality for integrals, one can show that the transform $T_\alpha $ is a bounded linear operator from $L^1_{w,\alpha}$ into the space of bounded continuous functions on $\mathbb{R}_+$. 
In the case $-1<\alpha<0$, one needs to impose more conditions on the function $f$ such that $T_\alpha $ is bounded. Both cases are analyzed in detail in Theorem \ref{boundedoperator}. 

The goal is to invert the transform $T_\alpha $, or in other words to solve the integral equation $g=T_\alpha f$ for the function $f$.
Each approach, presented in Sections \ref{directapproach} and \ref{Ffapprox}, respectively, requires the introduction of different integral operators and special functions, which will be given in the following.

Section \ref{Ffapprox} establishes the close relationship between the $\alpha$-sine transform (\ref{asinetrafo}) and the \emph{classical Fourier transform} on $\mathbb{R}$. The $\alpha$-sine transform is well defined for all functions from the space $L_{w,\alpha}^1(\mathbb{R}_+)$. We can evenly extend functions $f\in L_{w,\alpha}^1(\mathbb{R}_+)$ to the negative half of the real line by setting $f(-x)=f(x)$ for all $x\in\mathbb{R}$. For ease of notation, we denote this by $f\in L_{e,w,\alpha}^1(\mathbb{R}_+)$. Similarly, we denote by $L_e^p(\mathbb{R})$ the space of all even $L^p$-functions. 

Define the \emph{ Fourier transform} and its \emph{inverse transform} on the space of integrable functions $L^1(\mathbb{R})$ by
\begin{align*}
	\mathcal{F}v(y)=\int\limits_\mathbb{R} e^{ixy}v(x)dx~,\qquad
	\mathcal{F}^{-1}w(x)=\frac{1}{2\pi}\int\limits_\mathbb{R}e^{-ixy}w(y)dy~.
\end{align*}
for $v,w\in L^1(\mathbb{R})$. By the Euler formula, the integral kernels $e^{ixy}$ and $e^{-ixy}$ in the definition of the Fourier transform above can be replaced by $\cos(xy)$ for even functions $v,w\in L^1_e(\mathbb{R})$.

Note that on the space of Lebesgue integrable functions $L^1(\mathbb{R})$ the above Fourier transform is bounded, uniformly continuous, and vanishes at infinity by the Riemann-Lebesgue lemma \cite[Prop. 2.2.17.]{clasfour}.
It is well known that the Fourier transform of an integrable function might not be integrable itself. Therefore, only under certain additional conditions on $v$ the Fourier inversion theorem $v=\mathcal{F}^{-1}\mathcal{F}v=\mathcal{F}\mathcal{F}^{-1}v$ is applicable, e.g. if $v$ is a Schwartz function, or if it is integrable and square integrable \cite[Section 2.2.4]{clasfour}. 
On the space $L^1(\mathbb{R})\cap L^2(\mathbb{R})$ the Fourier transform $\mathcal{F}$ is a $L^2$-isometry, i.e. $\Vert \mathcal{F}v\Vert_2=2\pi\Vert v\Vert_2$ by the Plancherel theorem.
Furthermore, let $(v_n)_{n\in\mathbb{N}}$ be a sequence of functions in $L^1(\mathbb{R})\cap L^2(\mathbb{R})$ with $v_n\rightarrow v$ in $L^2$-norm as $n\rightarrow\infty$. Then, the convergence is preserved under the Fourier transform in the sense that $\mathcal{F}v_n\rightarrow\mathcal{F}v$ in the $L^2$-norm as $n\rightarrow\infty$ \cite[Eq. 2.2.16]{clasfour}. 

For Section \ref{directapproach} we consider the \emph{Fourier transform and its inverse transform on the multiplicative group $(\mathbb{R}_+,\cdot)$} by $\mathcal{F}_+: L^2\left(\mathbb{R}_+,\frac{dx}{x}\right)\rightarrow L^2\left(\mathbb{R}_+,\frac{dx}{x}\right)$ with
\begin{align*}
	\mathcal{F}_+v(y)=\int\limits_0^\infty e^{-i\log(x)\log(y)}v(x)\frac{dx}{x}~,\qquad
	\mathcal{F}_+^{-1}w(x)=\frac{1}{2\pi}\int\limits_0^\infty e^{i\log(x)\log(y)}w(y)\frac{dy}{y}
\end{align*}
for $v,w\in L^2\left(\mathbb{R}_+,\frac{dx}{x}\right)$,
as well as the similarity transform $\mathcal{M}:L^2\left(\mathbb{R}_+,x^cdx\right)\rightarrow L^2\left(\mathbb{R}_+,\frac{dx}{x}\right)$ and its inverse by
\begin{align*}
	\mathcal{M}v(y)=y^{(c+1)/2}v(y)~,\qquad
	\mathcal{M}^{-1}w(x)=x^{-(c+1)/2}w(x)
\end{align*}
for $v\in L^2\left(\mathbb{R}_+,x^cdx\right)$, $w\in L^2\left(\mathbb{R}_+,\frac{dx}{x}\right)$.

Additionally, we state the following useful result. For any complex number $z\in\mathbb{C}$ one can expand
\begin{align}
\label{binomseries}
	\left(1+x\right)^z=\sum\limits_{k=0}^\infty\binom{z}{k}x^k~,
\end{align}
where $\binom{z}{k}$ is the \emph{generalized binomial coefficient} defined by
\begin{align*}
	\binom{z}{k}=\frac{\Gamma(z+1)}{\Gamma(z-k+1)\Gamma(k+1)}=\frac{z(z-1)\dots(z-k+1)}{k!}~.
\end{align*}
Here, $\Gamma$ denotes the gamma function.  For $\vert x\vert<1$ the series converges absolutely for any $z\in\mathbb{C}$. If $\vert x\vert=1$ absolute convergence is given if and only if $Re(z)>0$, for $-1<Re(z)\leq 0$ the series converges if $x\neq-1$.

Lastly, the so-called \emph{generalized hypergeometric function} ${_pF_q}$ and the special \emph{Gauss hypergeometric function} ${}_2F_1$, which are  well known in mathematical physics, play important roles in the Fourier approximation approach. 

The generalized hypergeometric function ${}_pF_q$ with $p,q\in\mathbb{N}_0$ is defined by
\begin{align}
	\label{genhypergeom}
	{}_pF_q\left[a_1,\dots,a_p;b_1,\dots,b_q;z\right]=\sum\limits_{k=0}^\infty\frac{(a_1)_k\dots(a_p)_k}{(b_1)_k\dots(b_q)_k}\frac{z^k}{k!}
\end{align}
for any complex numbers $a_1,\dots,a_p,b_1,\dots,b_q\in\mathbb{C}$ and $z\in\mathbb{C}$, where $(\cdot)_n$ is called the \emph{Pochhammer symbol}, or \emph{rising factorial}, with 
\begin{align*}
	(a)_n=\begin{cases}
		1,&n=0,\\
		a(a+1)(a+2)\dots(a+n-1),&n\geq1.
	\end{cases}
\end{align*}
For $p\leq q$ the generalized hypergeometric function converges for all $z\in\mathbb{C}$, and for $p>q+1$ only if $z=0$. In the case $p=q+1$ the series converges if $\vert z\vert<1$, and when $z=1$ provided $Re\left(\sum b_i-\sum a_i\right)>0$ or when $z=-1$ provided $Re\left(\sum b_i-\sum a_i+1\right)>0$ \cite[p. 8]{hypergeometric}.

For the Gauss hypergeometric function ${}_2F_1$ it holds that
\begin{align*}
	{}_2F_1\left[a,b;c;z\right]=\sum\limits_{k=0}^\infty\frac{(a)_k(b)_k}{(c)_k}\frac{z^k}{k!}
\end{align*}
for $a,b,c,z\in\mathbb{C}$, and for $z=1$ provided $Re(c-a-b)>0$ the series converges absolutely with 
\begin{align}
	\label{gausshypgeomthm}
	{}_2F_1\left[a,b;c;1\right]=\sum\limits_{k=0}^\infty\frac{(a)_k(b_k)}{c_k}\frac{1}{k!}=\frac{\Gamma(c)\Gamma(c-a-b)}{\Gamma(c-a)\Gamma(c-b)}~.
\end{align}
This classical result is known as the Gauss hypergeometric theorem \cite[Section 1.3]{hypergeometric}. 
\section{Direct approach for $\alpha>1$}
\label{directapproach}

In \cite{glueckrothspodarev} the existence and uniqueness of a solution to integral equations of the form 
\begin{align*}
	w(y)=\int\limits_{\text{supp}(\gamma)}\beta(t)v\left(\gamma(t) y\right)dt
 \end{align*}
for given measurable functions $\beta,\gamma:\mathbb{R}^d\rightarrow\mathbb{R}$, and a weighted $L^2$-function $v$ on $\mathbb{R}$ is analyzed. The set $\text{supp}(\gamma)=\left\{t\in\mathbb{R}^d:\gamma(t)\neq 0\right\}$ denotes the support of $\gamma$. 
Their solution theory is based on operators on the multiplicative group on $\mathbb{R}^\times=\mathbb{R}\setminus\{0\}$. Since even functions are of interest, it suffices to consider the positive reals $\mathbb{R}_+$ only. 

Define the linear integral operator $\mathcal{G}:L^2\left(\mathbb{R}_+,x^cdx\right)\rightarrow L^2\left(\mathbb{R}_+,x^cdx\right)$ by
\begin{align}
	\label{G+}
	\mathcal{G}v(y)=\int\limits_{\text{supp}(\gamma)}\beta(t)v\left(\gamma(t) y\right)dt~,\quad y>0~,
\end{align}
where the functions $\beta$ and $\gamma$ are chosen such that the constant
\begin{align}
	C=\int\limits_{\text{supp}(\gamma)}\left\vert\beta(t)\right\vert\left\vert\gamma(t)\right\vert^{-\frac{c+1}{2}}dt<\infty
\end{align}
is finite. 
Furthermore, we introduce the function $\mu:\mathbb{R}_+\rightarrow\mathbb{C}$ given by
\begin{align}
	\label{mu+}
	\mu(x)=\int\limits_{\text{supp}(\gamma)}\beta(t)\vert\gamma(t)\vert^{-\frac{c+1}{2}}e^{i\log(x)\log\vert\gamma(t)\vert}dt~.
\end{align}
The function $\mu$ is bounded and its continuity follows from Lebesgue's dominated convergence theorem. 

The following lemma on the injectivity and surjectivity of the operator $\mathcal{G}$ can be derived from (\cite[Corollary 2.4]{glueckrothspodarev}) which states the results in the context of the multiplicative group $(\mathbb{R}^\times,\cdot)$.
\begin{lem}
	Assume that $C<\infty$, and let $\mu:\mathbb{R}\rightarrow\mathbb{C}$ be the bounded, continuous function defined in (\ref{mu+}). Then, the operator $\mathcal{G}:L^2\left(\mathbb{R}_+,x^cdx\right)\rightarrow L^2\left(\mathbb{R}_+,x^cdx\right)$ as defined in (\ref{G+}) is
	\begin{enumerate}[(i)]
		\item injective if and only if $\mu\neq0$ almost everywhere on $\mathbb{R}_+$ (with respect to the Lebesgue measure). 
		\item bijective if and only if $\inf\limits_{x\in\mathbb{R}_+}\left\vert\mu(x)\right\vert>0$.
	\end{enumerate}
\end{lem}

It is now possible to consider the $\alpha$-sine transform in the context of the integral operator $\mathcal{G}$. 
For all $y\in\mathbb{R}_+$, the transform $T_\alpha f$ can be reformulated by substituting $t=xy$ and setting $z=1/y$ to
\begin{align*}
	T_\alpha f(y)=\int\limits_0^\infty \left\vert\sin\left(xy\right)\right\vert^\alpha f(x)dx=z\int\limits_0^\infty \left\vert\sin\left(t\right)\right\vert^\alpha f(tz)dt=z\mathcal{G}_+v(z)~.
\end{align*}
with $\beta(t)=\left\vert\sin\left(t\right)\right\vert^\alpha$, $\gamma(t)=t$ with $supp(\gamma)=\mathbb{R}_+$ and $v=f$.
With $g=T_\alpha f$ this yields the equation
\begin{align}
\label{inteq}
	\mathcal{G}f(z)=z^{-1}g(z^{-1})
\end{align}
for all $z\in\mathbb{R}_+$. Plugging in the functions $\beta$ and $\gamma$ into the definitions of $\mathcal{G}$ and $\mu$, we can state the following theorem similar to \cite[Proposition 2.1, Theorem 2.3]{glueckrothspodarev}:
\begin{thm}
	\label{thmdirectapproach}
	Let $c=1+\delta$ with $\delta>0$ such that $\alpha\geq1+\delta/2$. Then 
	\begin{align}
	\label{intcondC+}
		C=\int\limits_0^\infty \left\vert\sin\left(t\right)\right\vert^\alpha t^{-\frac{c+1}{2}}dt<\infty~,
	\end{align}
	and
	\begin{enumerate}[(i)]
		\item 
		the linear operator $\mathcal{G}:L^2\left(\mathbb{R}_+,x^cdx\right)\rightarrow L^2\left(\mathbb{R}_+,x^cdx\right)$ given by
		\begin{align}
		\label{G+2}
			\mathcal{G}v(y)=\int\limits_0^\infty \left\vert\sin\left(t\right)\right\vert^\alpha v(ty)dt~,\quad y>0~,
		\end{align}
		is bounded on $L^2\left(\mathbb{R}_+,x^cdx\right)$ with operator norm $\Vert\mathcal{G}\Vert\leq C$. 
		\item for all functions $v\in L^2\left(\mathbb{R}_+,\frac{dx}{x}\right)$ the equation 
		$
		\tilde{\mathcal{G}}v=\mu v
		$
		holds, 
		where 
		$
		\tilde{\mathcal{G}}=\mathcal{F_+MGM}^{-1}\mathcal{F}_+^{-1}
		$
		and $\mu:\mathbb{R_+}\rightarrow\mathbb{C}$ with
		\begin{align}
			\label{mu+2}			
			\mu(x)=\int\limits_0^\infty \left\vert\sin\left(t\right)\right\vert^\alpha t^{-\frac{c+1}{2}}e^{i\log(t)\log(x)}dt~,\quad x>0~.
		\end{align}
	\end{enumerate}
\end{thm}
\begin{proof}
	The application of \cite[Proposition 2.1]{glueckrothspodarev} yields (i). For (ii) note that  $\tilde{\mathcal{G}}u=\tilde{\mu}u$ holds if and only if $\tilde{\mathcal{G}}\mathcal{F_+M}u=\tilde{\mu}\mathcal{F_+M}u$ is true for all $u\in L^2\left(\mathbb{R}_+,\frac{dx}{x}\right)$. For $y>0$
	\begin{align*}
		\mathcal{F_+M}u(y)=\int\limits_0^\infty x^{\frac{c+1}{2}}u(x)e^{-i\log(x)\log(y)}\frac{dx}{x}=\int\limits_0^\infty x^{\frac{c-1}{2}}u(x)e^{-i\log(x)\log(y)}dx~,
	\end{align*}
	and hence
	\begin{align*}
		\tilde{\mathcal{G}}\mathcal{F_+M}u(y)&=\mathcal{F_+M_+G_+}u(y)
		=\int\limits_0^\infty x^{\frac{c-1}{2}}\left[~\int\limits_0^\infty \left\vert\sin\left(t\right)\right\vert^\alpha u(tx)dt\right]e^{-i\log(x)\log(y)}dx\\
		&=\int\limits_0^\infty \int\limits_0^\infty s^{\frac{c-1}{2}}t^{-\frac{c-1}{2}}\left\vert\sin\left(t\right)\right\vert^\alpha u(s)e^{-i\log(s/t)\log(y)}\frac{ds}{t}dt\\
		&=\int\limits_0^\infty \left\vert\sin\left(t\right)\right\vert^\alpha t^{-\frac{c+1}{2}}e^{i\log(t)\log(y)}dt\int\limits_0^\infty s^{\frac{c-1}{2}}u(s)e^{-i\log(s)\log(y)}st
		=\mu(y)\mathcal{F_+M}u(y)
	\end{align*}
	using Fubini's theorem and the substitution $x=s/t$.
\end{proof}

We can now state an inversion formula for the operator $\mathcal{G}_+$ solving Equation (\ref{inteq}).
\begin{cor}
	\label{final_cor}
	Let $f\in L^1\left(\mathbb{R}_+,dx\right)\cap L^2\left(\mathbb{R}_+,x^cdx\right)$, where $c=1+\delta$ for some $\delta>0$ with $\alpha\geq1+\delta/2$. 
	Assume that the function $z^{-1}g(z^{-1})$ belongs to the space $L^2\left(\mathbb{R}_+, x^cdx\right)$. Then 
	\begin{align}
		\label{invformula}
	f(x)=\left(\mathcal{M}^{-1}\mathcal{F}_+^{-1}\frac{1}{\mu}\mathcal{F_+M}\left[z^{-1}g\left(z^{-1}\right)\right]\right)(x)
	\end{align}
	for all $x\in\mathbb{R}_+$. 
	The operator $\mathcal{G}$ on $L^2\left(\mathbb{R}_+,x^cdx\right)$ is injective, as $\mu\neq 0$ almost everywhere on $\mathbb{R}_+$, but not surjective since $\inf_{x>0}\vert\mu(x)\vert=0$.
\end{cor}
\begin{proof}
	The inversion formula (\ref{invformula}) is a direct consequence of Theorem \ref{thmdirectapproach}. It is easy to see that $\mu\neq0$ almost everywhere on $\mathbb{R}_+$. To show that $\mathcal{G}$ is not surjectiv, i.e. $\inf\limits_{x>0}\vert\mu(x)\vert=0$, note that $1<\frac{c+1}{2}\leq\alpha$ holds, and one can compute 
	\begin{align*}
		\vert\mu(x)\vert&=\left\vert\int\limits_0^\infty \left\vert\sin\left(t\right)\right\vert^\alpha t^{-\frac{c+1}{2}}e^{i\log(t)\log(x)}dt\right\vert\leq\left\vert\int\limits_0^1 \left\vert\sin\left(t\right)\right\vert^\alpha t^{-\frac{c+1}{2}+i\log(x)}dt\right\vert+\left\vert\int\limits_1^\infty \left\vert\sin\left(t\right)\right\vert^\alpha t^{-\frac{c+1}{2}+i\log(x)}dt\right\vert\\
		&\leq\left\vert\int\limits_0^1 t^{\alpha-\frac{c+1}{2}+i\log(x)}dt\right\vert+\left\vert\int\limits_1^\infty t^{-\frac{c+1}{2}+i\log(x)}dt\right\vert=\frac{1}{\left\vert\alpha-\frac{c+1}{2}+1+i\log(x)\right\vert}+\frac{1}{\left\vert\frac{c+1}{2}-1-i\log(x)\right\vert}\longrightarrow 0
	\end{align*}
as $x\rightarrow\infty$. 
\end{proof}

Corollary \ref{final_cor} gives an inversion formula which computes the solution of the integral equation $g=Tf$ directly. We will later see though that the involved operators $\mathcal{F_+}$, $\mathcal{M}$ and the function $\mu$ are numerically unstable, and inversion results are rather unsatisfying in practice. An even more significant drawback is the restriction to $\alpha>1$. The numerical analysis of the direct approach can be found in Section \ref{numDA}.

\section{Fourier approximation approach}
\label{Ffapprox}

Recall the space $L^1_{w,\alpha}(\mathbb{R}_+)$ defined in the preliminaries by 
\begin{align*}
	L^1_{w,\alpha}(\mathbb{R}_+)=\begin{cases}
		L^1(\mathbb{R}_+),&\alpha\geq0,\\
		L^1(\mathbb{R}_+)\cap L^1\left(\mathbb{R}_+,\max\left\{\frac{1}{x},1\right\}dx\right),&-1<\alpha<0.
	\end{cases}
\end{align*}
\begin{rem}
	In the following, the even extension of $f\in L^1_{w,\alpha}(\mathbb{R}_+)$ to the whole real line is also denoted by $f$, whenever we consider the Fourier transform $\mathcal{F}f$.  
\end{rem}

Consider the transform 
$
	T_2 f(y)=\int_0^\infty\left\vert\sin\left(xy\right)\right\vert^2 f(x)dx.
$
Applying the cosine double angle formula  $\cos(2x)=1-2\sin^2(x)$ yields $$T_2f(y)=\frac{\mathcal{F}f(0)}{4}-\frac{\mathcal{F}f(2y)}{4}.$$ Assuming the constant $\mathcal{F}f(0)$ is known, the inversion of $T_2$ is simply achieved by applying the Fourier inverse transform to $\mathcal{F}f(0)-2T_2f$.
But for $\alpha>-1, \alpha\neq 2$, the cosine double angle formula alone does not help a lot. 

In following section, we first prove a series representation of $T_\alpha $. Section \ref{Fourierapproximation} then establishes the Fourier approximation approach, where the Fourier transform $\mathcal{F}f$ is approximated from the transform $T_\alpha f$. Lastly, Section \ref{BLM} presents an interpolation method from which the function $f$ is computed.

\subsection{Series representation}

\begin{thm}
	\label{prop2}
	\begin{enumerate}[(i)]
		\item Let $\alpha>-1$. The function $\frac{1}{2}\left\vert\sin\left(\frac{x}{2}\right)\right\vert^\alpha$ admits the  Fourier series expansion 
		\begin{align*}
			\frac{1}{2}\left\vert\sin\left(\frac{x}{2}\right)\right\vert^\alpha=\frac{c_0}{2}+\sum\limits_{j=1}^\infty c_j\cos(jx)
		\end{align*}
		with real Fourier coefficients 
		\begin{align}
			\label{cjdef}
			c_j=\begin{cases}
				\frac{(-1)^j}{2^{\alpha}}\frac{\Gamma\left(1+\alpha\right)}{\Gamma\left(\frac{\alpha}{2}-j+1\right)\Gamma\left(\frac{\alpha}{2}+j+1\right)},&\alpha\neq 2k, ~j\in\mathbb{N}_0,\\
				\frac{(-1)^j}{4^k}\binom{2k}{k-j},&\alpha=2k,~j=0,\dots,k,~k\in\mathbb{N}_0,\\
				0,&\alpha=2k,~j>k,~k\in\mathbb{N}_0.
			\end{cases}
		\end{align}
	The Fourier series converges absolutely and uniformly on $\mathbb{R}$ for $\alpha\geq 0$. It converges almost everywhere on $\mathbb{R}$ ~ for $-1<\alpha<0$.
	\item For any integrable function $f\in L^1_{w,\alpha}(\mathbb{R}_+)$, $\alpha>-1$, the integral transform $T_\alpha f(y)=\int_0^\infty\left\vert\sin\left(xy\right)\right\vert^\alpha f(x)dx$ has the series representation 
	\begin{align}
		\label{TfFourier2}
		T_\alpha f\left(y\right)=\frac{c_0}{2}\mathcal{F}f(0)+\sum\limits_{j=1}^\infty c_j\mathcal{F}f(2jy).
	\end{align}
	In the case $\alpha=2k$, $k\in\mathbb{N}_0$, the above infinite series in expression (\ref{TfFourier2}) becomes a finite sum, i.e.
	$$T_{\alpha}f(y)=\frac{c_0}{2}\mathcal{F}f(0)+c_1\mathcal{F}f(2y)+c_2\mathcal{F}f(4y)+\dots+c_{k}\mathcal{F}f\left(2k y\right).$$
	\end{enumerate}
\end{thm}
\begin{proof}
\begin{enumerate}[(i)]
\item Let $\alpha>-1$. First of all, note that
\begin{align*}
	\frac{1}{2}\left\vert\sin\left(\frac{x}{2}\right)\right\vert^\alpha&=\frac{1}{2}\left(\sin^2\left(\frac{x}{2}\right)\right)^{\alpha/2}=\frac{1}{2^{\alpha/2+1}}\left(1-\cos\left(x\right)\right)^{\alpha/2}
	=\frac{1}{2^{\alpha/2+1}}\sum\limits_{k=0}^\infty\binom{\alpha/2}{k}(-1)^k\cos^k(x)
\end{align*}
for all $x,y\in\mathbb{R}$ by the binomial series (\ref{binomseries}). Since $\vert\cos(x)\vert<1$ almost everywhere on $\mathbb{R}$ the series above converges absolutely almost everywhere for any $\alpha>-1$. Splitting the series into an odd and even summands and 
applying the cosine power formulae
\begin{align*}
	\cos^k(2xy)
	&=\begin{cases}
		\frac{2}{2^{2n-1}}\sum\limits_{l=0}^{n-1}\binom{2n-1}{l}\cos\left((2n-1-2l)2xy\right)~,&k=2n-1,\\[3ex]
		\frac{1}{2^{2n}}\binom{2n}{n}+\frac{2}{2^{2n}}\sum\limits_{l=0}^{n-1}\binom{2n}{l}\cos\left((2n-2l)2xy\right)~,&k=2n,
	\end{cases}\hspace{5ex}~n\in\mathbb{N},
\end{align*}
yields
\begin{align*}
	\frac{1}{2}\left\vert\sin\left(\frac{x}{2}\right)\right\vert^\alpha&=\frac{1}{2^{\alpha/2+1}}\left(1-\sum\limits_{n=1}^\infty\binom{\alpha/2}{2n-1}\cos^{2n-1}(x)+\sum\limits_{n=1}^\infty\binom{\alpha/2}{2n}\cos^{2n}(x)\right)\\
	&=\frac{1}{2^{\alpha/2+1}}\left(1-\sum\limits_{n=1}^\infty\binom{\alpha/2}{2n-1}\frac{2}{2^{2n-1}}\sum\limits_{l=0}^{n-1}\binom{2n-1}{l}\cos((2n-1-2l)x)
	\right.\\
	&\left.\hspace{10ex}
	+\sum\limits_{n=1}^\infty\binom{\alpha/2}{2n}\frac{1}{2^{2n}}\binom{2n}{n}+\sum\limits_{n=1}^\infty\binom{\alpha/2}{2n}\frac{2}{2^{2n}}\sum\limits_{l=0}^{n-1}\binom{2n}{l}\cos((2n-2l)x)\right)\\
	&=\frac{1}{2^{\alpha/2+1}}\left(\sum\limits_{n=0}^\infty\binom{\alpha/2}{2n}\binom{2n}{n}\frac{1}{4^{n}}+2\sum\limits_{k=1}^\infty\binom{\alpha/2}{k}\frac{(-1)^k}{2^k}\sum\limits_{l=0}^{\lfloor\frac{k-1}{2}\rfloor}\binom{k}{l}\cos\left((k-2l)x\right)\right)~.
\end{align*}
Substituting $j=k-2l$ and rearranging the series above we get
\begin{align}
	\label{TfFourier}
	\frac{1}{2}\left\vert\sin\left(\frac{x}{2}\right)\right\vert^\alpha=\underbrace{\frac{1}{2^{\alpha/2+1}}\sum\limits_{n=0}^\infty\binom{\alpha/2}{2n}\binom{2n}{n}\frac{1}{4^n}}_{=\colon  c}+\sum\limits_{j=1}^\infty\underbrace{\left(\frac{(-1)^j}{2^{\alpha/2}}\sum\limits_{l=0}^\infty\binom{\alpha/2}{j+2l}\binom{j+2l}{l}\frac{1}{2^{j+2l}}\right)}_{=:c_j}\cos(jx)~.
\end{align}
The coefficients $c$ and $c_j$, $j\in\mathbb{N}$, can be expressed in terms of the Gauss hypergeometric function aforementioned in Section \ref{prelim}.

First, we compute the constant $c$ in equation (\ref{TfFourier}). Note that
\begin{align*}
	\binom{\alpha/2}{2n}\binom{2n}{n}\frac{1}{4^n}&=\frac{\frac{\alpha}{2}\left(\frac{\alpha}{2}-1\right)~\dots~\left(\frac{\alpha}{2}-2n+2\right)+\left(\frac{\alpha}{2}-2n+1\right)}{(2n)!}\frac{(2n)!}{n!n!}\frac{1}{4^n}\\
	&=\frac{-\frac{\alpha}{4}\left(-\frac{\alpha}{4}+\frac{1}{2}\right)\left(-\frac{\alpha}{4}+1\right)\left(-\frac{\alpha}{4}+\frac{3}{2}\right)~\dots~\left(-\frac{\alpha}{4}+n-1\right)\left(-\frac{\alpha}{4}+\frac{1}{2}+n-1\right)}{n!}\frac{(-2)^{2n}}{4^n}\frac{1}{n!}\\
	&=\frac{(-\frac{\alpha}{4})_n\left(-\frac{\alpha}{4}+\frac{1}{2}\right)_n}{(1)_n}\frac{1}{n!}.
\end{align*}
For all $\alpha>-1$ it holds that $1-\left(-\frac{\alpha}{4}\right)-\left(-\frac{\alpha}{4}+\frac{1}{2}\right)=\frac{1}{2}+\frac{\alpha}{2}>0$, and the Gauss hypergeometric theorem (\ref{gausshypgeomthm}) yields
\begin{align*}
	\sum\limits_{n=0}^\infty\binom{\alpha/2}{2n}\binom{2n}{n}\frac{1}{4^n}=\sum\limits_{n=0}^\infty\frac{\left(-\frac{\alpha}{4}\right)_n\left(-\frac{\alpha}{4}+\frac{1}{2}\right)_n}{(1)_n}\frac{1}{n!}={}_2F_1\left[-\frac{\alpha}{4},-\frac{\alpha}{4}+\frac{1}{2};1;1\right]=\frac{\Gamma\left(\frac{1}{2}+\frac{\alpha}{2}\right)}{\Gamma\left(1+\frac{\alpha}{4}\right)\Gamma\left(\frac{1}{2}+\frac{\alpha}{4}\right)}.
\end{align*}
Hence, the constant $c$ is given by
\begin{align}
\label{c}
	c=\frac{1}{2^{\alpha/2+1}}\sum\limits_{n=0}^\infty\binom{\alpha/2}{2n}\binom{2n}{n}\frac{1}{4^n}
=\frac{1}{2^{\alpha/2+1}}\frac{\Gamma\left(\frac{1}{2}+\frac{\alpha}{2}\right)}{\Gamma\left(1+\frac{\alpha}{4}\right)\Gamma\left(\frac{1}{2}+\frac{\alpha}{4}\right)}.
\end{align}

For all coefficients $c_j$, $j\in\mathbb{N}$, recall from Equation (\ref{TfFourier})
\begin{align*}
	c_j=\frac{(-1)^j}{2^{\alpha/2}}\sum\limits_{n=0}^\infty\binom{\alpha/2}{j+2n}\binom{j+2n}{n}\frac{1}{2^{j+2n}}.
\end{align*}
We compute 
\begin{align*}
	\binom{\alpha/2}{j+2l}\binom{j+2l}{l}\frac{1}{2^{j+2l}}&=\frac{1}{2^j}\frac{\frac{\alpha}{2}\left(\frac{\alpha}{2}-1\right)~\dots~\left(\frac{\alpha}{2}-j-2l+1\right)}{(j+2l)!}\frac{(j+2l)!}{l!(j+l)!}\frac{1}{4^l}\\
	&=\frac{\frac{\alpha}{2}\left(\frac{\alpha}{2}-1\right)~\dots~\left(\frac{\alpha}{2}-j+1\right)}{2^jj!}\frac{\left(\frac{\alpha}{2}-j\right)\left(\frac{\alpha}{2}-j-1\right)~\dots~\left(\frac{\alpha}{2}-j-2l+1\right)}{(j+1)_l}\frac{1}{4^l}\frac{1}{l!}\\
	&=\frac{\Gamma\left(\frac{\alpha}{2}+1\right)}{2^j\Gamma\left(\frac{\alpha}{2}-j+1\right)j!}\frac{\left(-\frac{\alpha}{4}+\frac{j}{2}\right)_l\left(-\frac{\alpha}{4}+\frac{j}{2}+\frac{1}{2}\right)_l}{(j+1)_l}\frac{(-2)^{2l}}{4^l}\frac{1}{l!}\\
	&=\frac{\Gamma\left(\frac{\alpha}{2}+1\right)}{2^j\Gamma\left(\frac{\alpha}{2}-j+1\right)j!}\frac{\left(-\frac{\alpha}{4}+\frac{j}{2}\right)_l\left(-\frac{\alpha}{4}+\frac{j}{2}+\frac{1}{2}\right)_l}{(j+1)_l}\frac{1}{l!}.
\end{align*}

Note that $j+1-\left(-\frac{\alpha}{4}+\frac{j}{2}\right)-\left(-\frac{\alpha}{4}+\frac{j}{2}+\frac{1}{2}\right)=\frac{1}{2}+\frac{\alpha}{2}>0$ for all $\alpha>-1$. Applying the Gauss hypergeometric theorem (\ref{gausshypgeomthm}) and Legendre's duplication formula $\Gamma\left(\frac{x}{2}\right)\Gamma\left(\frac{x+1}{2}\right)=\sqrt{\pi}/2^{x-1}\Gamma(x)$, it follows that
\begin{align*}
	\sum\limits_{l=0}^\infty\binom{\alpha/2}{j+2l}\binom{j+2l}{l}\frac{1}{2^{j+2l}}&=\frac{\Gamma\left(\frac{\alpha}{2}+1\right)}{2^j\Gamma\left(\frac{\alpha}{2}-j+1\right)j!}\sum\limits_{l=0}^\infty\frac{\left(-\frac{\alpha}{4}+\frac{j}{2}\right)_l\left(-\frac{\alpha}{4}+\frac{j}{2}+\frac{1}{2}\right)_l}{(j+1)_l}\frac{1}{l!}\\
	&=\frac{\Gamma\left(\frac{\alpha}{2}+1\right)}{2^j\Gamma\left(\frac{\alpha}{2}-j+1\right)j!}{}_2F_1\left[-\frac{\alpha}{4}+\frac{j}{2},-\frac{\alpha}{4}+\frac{j}{2}+\frac{1}{2};j+1;1\right]\\
	&=\frac{\Gamma\left(\frac{\alpha}{2}+1\right)}{2^j\Gamma\left(\frac{\alpha}{2}-j+1\right)j!}\frac{\Gamma\left(j+1\right)\Gamma\left(\frac{1}{2}+\frac{\alpha}{2}\right)}{\Gamma\left(1+\frac{\alpha}{4}+\frac{j}{2}\right)\Gamma\left(\frac{1}{2}+\frac{\alpha}{4}+\frac{j}{2}\right)}\\
	&=\frac{\frac{\sqrt{\pi}}{2^\alpha}\Gamma\left(1+\alpha\right)}{2^j\Gamma\left(\frac{\alpha}{2}-j+1\right)\frac{\sqrt{\pi}}{2^{\alpha/2+j}}\Gamma\left(\frac{\alpha}{2}+j+1\right)}	=\frac{1}{2^{\alpha/2}}\frac{\Gamma\left(1+\alpha\right)}{\Gamma\left(\frac{\alpha}{2}-j+1\right)\Gamma\left(\frac{\alpha}{2}+j+1\right)}~.
\end{align*}
Thus, the coefficients $c_j$ are given by
\begin{align}
	c_j=\frac{(-1)^j}{2^{\alpha/2}}\sum\limits_{n=0}^\infty\binom{\alpha/2}{j+2n}\binom{j+2n}{n}\frac{1}{2^{j+2n}}=\frac{(-1)^j}{2^{\alpha}}\frac{\Gamma\left(1+\alpha\right)}{\Gamma\left(\frac{\alpha}{2}-j+1\right)\Gamma\left(\frac{\alpha}{2}+j+1\right)}
\end{align}
for any $j\in\mathbb{N}$. Note, that setting $j=0$ in the above yields $c=c_0/2$, where $c$ was given in (\ref{c}). For $\alpha=2k$, $k\in\mathbb{N}_0$ the binomial series expansion is only a finite sum and the coefficients $c_j$ simplify to $c_j=\frac{(-1)^j}{4^k}\binom{2k}{k-j}$ for $j=0,\dots,k$ and $c_j=0$ for all $j>k$. 

To summarize the Fourier series expansion of  $\frac{1}{2}\left\vert\sin\left(\frac{x}{2}\right)\right\vert^\alpha$ is given by 
\begin{align}
	\label{Fourierexp}
	\frac{1}{2}\left\vert\sin\left(\frac{x}{2}\right)\right\vert^\alpha=\frac{c_0}{2}+\sum\limits_{j=1}^\infty c_j\cos(jx). 
\end{align}

Absolute and uniform convergence of the above Fourier series for $\alpha\geq0$ follow from \cite[Theorem 2.5]{folland} since $\frac{1}{2}\vert\sin(x/2)\vert^\alpha$ is $\pi$-periodic, continuous and piecewise smooth for $\alpha\geq0$. For $-1<\alpha<0$ the function $\frac{1}{2}\vert\sin(x/2)\vert^\alpha$ is only $\pi$-periodic and piecewise smooth (with discontinuities at $x=2k\pi$). The Fourier series converges at every point where the function is continuous \cite[Theorem 2.1]{folland}.

\item For $\alpha>-1$ consider the even extension of $f\in L^1_{w,\alpha}(\mathbb{R}_+)$ onto the whole $\mathbb{R}$. For ease of notation, we write $f\in L^1_{e,w,\alpha}(\mathbb
R)$. 
We make use of the Fourier expansion (\ref{Fourierexp}) of the integral kernel of $T_\alpha$. Note that integration and summation can be interchanged by 
Lebesgue's dominated convergence theorem. Then, 
\begin{align*}
	T_\alpha f(y)&=\int_0^\infty\left\vert\sin\left(xy\right)\right\vert^\alpha f(x)dx=\int_\mathbb{R}	\frac{1}{2}\left\vert\sin\left(\frac{2xy}{2}\right)\right\vert^\alpha f(x)dx\\
	&=\frac{c_0}{2}\int\limits_\mathbb{R}f(x)dx+\sum_{j=1}^\infty c_j\int\limits_\mathbb{R}\cos(2jxy)f(x)dx\\
	&=\frac{c_0}{2}\mathcal{F}f(0)+\sum_{j=1}^\infty c_j\mathcal{F}f(2jy).
\end{align*}
The case $\alpha=2k$, $k\in\mathbb{N}_0$ follows immediately from the definition (\ref{cjdef}) of the coefficients $c_j$ .
\end{enumerate}
\end{proof}

\begin{cor}
	\label{corollary_ck}
	\begin{enumerate}[(i)]
		\item \label{seriesconv}The series $\sum_{j=1}^\infty c_j$ converges absolutely for all $\alpha\geq0$. In particular, for $\alpha>0$ it holds that $\sum_{j=1}^\infty c_j=-\frac{c_0}{2}$,
		and in the case $\alpha\in(0,2]$, the absolute limit is given by $\sum_{j=1}^\infty\vert c_j\vert=c_0/2$.
		For $-1<\alpha<0$, the series $\sum_{j=1}^\infty c_j$ diverges. 
		\item For $\alpha\geq 0$, the convergence in the series representation of $T_\alpha f$, $f\in L^1(\mathbb{R}_+)$, in Equation (\ref{TfFourier2}) is uniform.
		Consequently, $T_\alpha f$ is continuous and bounded for all $f\in L^1(\mathbb{R}_+)$, $\alpha\geq 0$. For $-1<\alpha<0$,
		convergence holds in the $L^2$-sense if $f\in L^1_{w,\alpha}(\mathbb{R}_+)\cap L^2(\mathbb{R}_+)$.

	\end{enumerate}
\end{cor}

\begin{proof}
	\begin{enumerate}[(i)]
		\item 
		The function $\frac{1}{2}\left\vert\sin\left(x/2\right)\right\vert^\alpha$ is $2\pi$-periodic, continuous and piecewise smooth for $\alpha\geq0$. Its Fourier series converges absolutely and uniformly on $\mathbb{R}$, in particular $\sum_{j=1}^\infty\vert c_j\vert<\infty$ \cite[Theorem 2.5]{folland}. Note that $T_\alpha f(0)$ is well-defined, and
		\begin{align*}
		T_\alpha f(0)=0=\frac{c_0}{2}\mathcal{F}f(0)+\sum\limits_{j=1}^\infty c_j\mathcal{F}f(0)=\mathcal{F}f(0)\left(\frac{c_0}{2}+\sum\limits_{j=1}^\infty c_j\right)
		\end{align*}
		for $\alpha>0$, hence $\sum_{j=1}^\infty c_j=-c_0/2$. 
		For $\alpha\in(0,2]$, it holds that $c_j\leq0$ for all $j\geq 1$ by their definition, and consequently $\sum\limits_{j=1}^\infty\vert c_j\vert=-\sum\limits_{j=1}^\infty c_j=c_0/2$. 
		
		To show the divergence of $\sum_{j=1}^\infty c_j$ in the case $-1<\alpha<0$, note that $$\Gamma\left(\frac{\alpha}{2}+j+1\right)=\Gamma\left(\frac{\alpha}{2}+2\right)\left(\frac{\alpha}{2}+2\right)_{j-1},$$ 
		and 
		$$\Gamma\left(\frac{\alpha}{2}-j+1\right)=\frac{\Gamma\left(\frac{\alpha}{2}\right)}{\left(\frac{\alpha}{2}-j+1\right)\left(\frac{\alpha}{2}-j+2\right)\dots\left(\frac{\alpha}{2}-1\right)}=\frac{\Gamma\left(\frac{\alpha}{2}\right)}{(-1)^{j-1}\left(1-\frac{\alpha}{2}\right)_{j-1}},$$
		where $(\cdot)_{n}$ is the Pochhammer symbol. Then,
		\begin{align}
			\label{cj_eq}
			c_j=\frac{(-1)^j}{2^\alpha}\frac{\Gamma(1+\alpha)}{\Gamma\left(\frac{\alpha}{2}+2\right)\left(\frac{\alpha}{2}+2\right)_{j-1}}\frac{(-1)^{j-1}\left(1-\frac{\alpha}{2}\right)_{j-1}}{\Gamma\left(\frac{\alpha}{2}\right)}=-\frac{\Gamma(1+\alpha)}{2^\alpha\Gamma\left(\frac{\alpha}{2}\right)\Gamma\left(\frac{\alpha}{2}+2\right)}\frac{\left(1-\frac{\alpha}{2}\right)_{j-1}\left(1\right)_{j-1}}{\left(\frac{\alpha}{2}+2\right)_{j-1}}\frac{1}{(j-1)!},
		\end{align}
		which yields 
		$$\sum_{j=1}^\infty c_j=-\frac{\Gamma(1+\alpha)}{2^\alpha\Gamma\left(\frac{\alpha}{2}\right)\Gamma\left(\frac{\alpha}{2}+2\right)}{}_2F_1\left[1-\frac{\alpha}{2},1;\frac{\alpha}{2}+2;1\right].$$
		The series diverges as $\frac{\alpha}{2}+2-\left(1-\frac{\alpha}{2}\right)-1=\alpha\in(-1,0)$, and since it holds that 
		\begin{align*}
			\frac{\left(1-\frac{\alpha}{2}\right)_n\left(1\right)_n}{\left(\frac{\alpha}{2}+2\right)_n}\frac{1}{n!}
				=\frac{\Gamma\left(\frac{\alpha}{2}+2\right)}{\Gamma\left(1-\frac{\alpha}{2}\right)}n^{-\alpha-1}\left(1+O\left(n^{-1}\right)\right),
		\end{align*}
		with $n=j-1$, where $-\alpha-1\in(-1,0)$, cf. \cite[p. 57]{bateman}.
		\item Define 
		$
		T_\alpha ^{(n)}f=\frac{c_0}{2}\mathcal{F}f(0)+\sum_{j=1}^n c_j\mathcal{F}f(2j~\cdot).
		$
		By (i) the Fourier coefficients satisfy $\vert c_j\vert\rightarrow0$ as $j\rightarrow\infty$. Furthermore, $L^1_{w,\alpha}(\mathbb{R}_+)\subset L^1(\mathbb{R}_+)$ and $\vert\mathcal{F}f(y)\vert\leq 2\Vert f\Vert_1$ for all $y\in\mathbb{R}_+$, where $\Vert\cdot\Vert_1$ denotes the $L^1$-norm on $\mathbb{R}_+$.
		
		Let $\alpha\geq 0$. Then, 
		\begin{align*}
			\lim\limits_{n\rightarrow\infty}\left\Vert T_\alpha f-T_\alpha ^{(n)}f\right\Vert_\infty&=\lim\limits_{n\rightarrow\infty}\sup\limits_{y\in\mathbb{R}_+}\left\vert T_\alpha f(y)-T_\alpha ^{(n)}f(y)\right\vert=\lim\limits_{n\rightarrow\infty}\sup\limits_{y\in\mathbb{R}}\left\vert\sum\limits_{j=n+1}^\infty c_j\mathcal{F}f(2jy)\right\vert\\
			&\leq \lim\limits_{n\rightarrow\infty}\sup\limits_{y\in\mathbb{R}}\sum\limits_{j=n+1}^\infty \vert c_j\vert\underbrace{\left\vert\mathcal{F}f(2jy)\right\vert}_{\leq 2\Vert f\Vert_1}\leq2\Vert f\Vert_1\lim\limits_{n\rightarrow\infty}\sum\limits_{j=n+1}^\infty\vert c_j\vert=0. 
		\end{align*}
		Since the Fourier transform $\mathcal{F}f$ is bounded and continuous for all integrable functions, continuity and boundedness of $T_\alpha f$ follow by the uniform convergence of the $T_\alpha ^{(n)}f$. 
		
		Considering the case $-1<\alpha<0$, note that on the space $L^1_{w,\alpha}(\mathbb{R}_+)\cap L^2(\mathbb{R}_+)$ the Fourier transform $\mathcal{F}$ is an $L^2$-isometry \cite[Section 2.2.4]{clasfour}, and by the Plancherel theorem the equality $\Vert\mathcal{F}f\Vert_2=2\pi\Vert f\Vert_2$ holds. Then,
		\begin{align*}
			\lim\limits_{n\rightarrow\infty}\left\Vert T_\alpha f-T_\alpha ^{(n)}f\right\Vert_2^2
			&=\lim\limits_{n\rightarrow\infty}\int_0^\infty\left\vert\sum\limits_{j=n+1}^\infty c_j\mathcal{F}f(2jy)\right\vert^2 dy\\
			&\leq \lim\limits_{n\rightarrow\infty}\int_0^\infty\left(\sum_{j=n+1}^\infty\vert c_j\vert^2\vert\mathcal{F}f(2jy)\vert^2+2\sum_{n+1<j<k}^\infty\vert c_j\vert\vert c_k\vert\vert\mathcal{F}f(2jy)\vert\vert\mathcal{F}f(2jk)\vert\right) dy\\
			&\leq \lim\limits_{n\rightarrow\infty}\sum_{j=n+1}^\infty\vert c_j\vert^2\underbrace{\int_0^\infty\vert\mathcal{F}f(2jy)\vert^2dy}_{\leq2\pi\Vert f\Vert_2^2}+2\sum_{n+1<j<k}^\infty\vert c_j\vert\vert c_k\vert\underbrace{\int_0^\infty\vert\mathcal{F}f(2jy)\vert\vert\mathcal{F}f(2jk)\vert dy}_{\leq2\pi\Vert f\Vert_2^2 }\\
			&=0,
		\end{align*}
		where the last inequality is due to the Cauchy-Schwartz inequality.

	\end{enumerate}
\end{proof}
For the boundedness of the integral operator $T_\alpha$, in particular for the case $-1<\alpha<0$, we introduce the Sobolev space $$W^{1,1}(\mathbb{R}_+)=\left\{f\in L^1(\mathbb{R}_+): \exists f'\in L^1(\mathbb{R}_+)\right\}$$ of integrable functions $f$ with integrable first derivative $f'$, and define the norm $\Vert\cdot\Vert_D$ on the space $L^1_{w,\alpha}(\mathbb{R}_+)\cap W^{1,1}(\mathbb{R}_+)$ by 
$$\Vert f\Vert_D=\int\limits_{\mathbb{R}_+}\vert f(x)\vert\max\left\{\frac{1}{x},1\right\}dx+\Vert f'\Vert_1.$$
Moreover, consider the space $L^1([0,1],dx)\cap C_b\left((1,\infty)\right)$ with norm  $\Vert\cdot\Vert_\ast$ defined by 
$$\Vert g\Vert_\ast=\int\limits_0^1\vert g(y)\vert dy+\sup\limits_{y>1} \vert g(y)\vert.$$

\begin{thm}
	\label{boundedoperator}
	\begin{enumerate}[(i)]
		\item For $\alpha\geq 0$ the integral operator $T_\alpha:\left(L^1(\mathbb{R}_+),\Vert\cdot\Vert_1\right)\rightarrow \left(C_{b}(\mathbb{R}_+),\Vert\cdot\Vert_\infty\right)$ is a linear bounded (continuous) operator. In particular, for $\alpha\in(0,2]$, the operator norm is bounded by $\Vert T_\alpha\Vert\leq 2c_0$. 
		\item For $-1<\alpha<0$ the integral operator $T_\alpha:\left(L^1_{w,\alpha}(\mathbb{R_+})\cap W^{1,1}(\mathbb{R}_+),\Vert\cdot\Vert_D\right)\rightarrow\left(L^1\left([0,1],dx\right)\cap C_b\left((1,\infty)\right),\Vert\cdot\Vert_\ast\right)$ is a linear bounded (continuous) operator. The operator norm is bounded by 
		\begin{align*}
			\Vert T_\alpha\Vert\leq C_\alpha\left(\frac{1}{\pi}+1\right)+c_0\left(1-\frac{\alpha}{\alpha+2}{}_3F_2\left[1-\frac{\alpha}{2},1,1;\frac{\alpha}{2}+2,2;1\right]\right),
		\end{align*}
		with $C_\alpha=\sqrt{\pi}\frac{\Gamma\left(\frac{1+\alpha}{2}\right)}{\Gamma\left(1+\frac{\alpha}{2}\right)}$ and hypergeometric function ${}_3F_2$ as defined in Equation (\ref{genhypergeom}).
	\end{enumerate}
\end{thm}
\begin{proof}
	Linearity of $T_\alpha$ follows directly from the linearity of integrals. For the boundedness of $T_\alpha$ we consider the cases $\alpha\geq 0$ and $-1<\alpha<0$ separately.
	\begin{enumerate}[(i)]
		\item For $\alpha\geq 0$, use the inequality $\vert\mathcal{F}f(y)\vert\leq 2\Vert f\Vert_1$, $y\in\mathbb{R}_+$, to get
		\begin{align*}
			\Vert T_\alpha f\Vert_\infty=\left\Vert \frac{c_0}{2}\mathcal{F}f(0)+\sum\limits_{j=1}^\infty c_j\mathcal{F}f(2j\cdot)\right\Vert_\infty\leq2\Vert f\Vert_1\left(\frac{c_0}{2}+\sum\limits_{j=1}^\infty\vert c_j\vert\right)<\infty
		\end{align*}
		and 
		\begin{align*}
			\Vert T_\alpha \Vert=\sup\limits_{f\in L^1(\mathbb{R}_+)\setminus\{0\}}\frac{\Vert T_\alpha f\Vert_\infty}{\Vert f\Vert_1}\leq
			2\left(\frac{c_0}{2}+\sum\limits_{j=1}^\infty\vert c_j\vert\right)<\infty,
		\end{align*}
		where in the above $0$ denotes the constant zero function. Following Corollary \ref{corollary_ck} (ii) the upper bound of the operator norm of $T_\alpha$ for $\alpha\in(0,2]$ is then given by $$\Vert T_\alpha \Vert \leq 2\left(\frac{c_0}{2}+\sum\limits_{j=1}^\infty\vert c_j\vert\right)=2c_0.$$
		\item Note that for $f\in C^n(\mathbb{R}_+)$ with integrable derivatives $f^{(k)}\in L^1(\mathbb{R}_+)$ for all $k\leq n$ the Fourier transform $\mathcal{F}f$ behaves as $O\left(\frac{1}{y^n}\right)$, $y\rightarrow\infty$. In particular, for $f\in W^{1,1}(\mathbb{R}_+)$ it follows that $\vert\mathcal{F}f(y)\vert=\frac{\vert\mathcal{F}f'(y)\vert}{\vert y\vert }$. Furthermore, $\vert c_j\vert=c_j$ since the coefficients $c_j$, $j\geq 0$, are positive for $-1<\alpha<0$ by their definition in Equation (\ref{cjdef}), and from Equation (\ref{cj_eq}) in Corollary \ref{corollary_ck} (ii) it holds that
		\begin{align}
			\label{cjhypergeom}
			c_j=-\frac{\Gamma(1+\alpha)}{2^\alpha\Gamma\left(\frac{\alpha}{2}\right)\Gamma\left(\frac{\alpha}{2}+2\right)}\frac{\left(1-\frac{\alpha}{2}\right)_{j-1}}{\left(\frac{\alpha}{2}+2\right)_{j-1}}=-\frac{\Gamma(1+\alpha)}{2^\alpha\Gamma\left(1+\frac{\alpha}{2}\right)^2}\frac{\alpha}{\alpha+2}\frac{\left(1-\frac{\alpha}{2}\right)_{j-1}}{\left(\frac{\alpha}{2}+2\right)_{j-1}}=-c_0\frac{\alpha}{\alpha+2}\frac{\left(1-\frac{\alpha}{2}\right)_{j-1}}{\left(\frac{\alpha}{2}+2\right)_{j-1}}.
		\end{align}
		for $j\geq 1$. With $j=\frac{(2)_{j-1}}{(1)_{j-1}}$ it follows that
		\begin{align*}
			\sum\limits_{j=1}^\infty \frac{\vert c_j\vert}{j}=\sum\limits_{j=1}^\infty \frac{c_j}{j}=-c_0\frac{\alpha}{\alpha+2}\sum_{j=1}^\infty\frac{\left(1-\frac{\alpha}{2}\right)_{j-1}\left(1\right)_{j-1}\left(1\right)_{j-1}}{\left(\frac{\alpha}{2}+2\right)_{j-1}\left(2\right)_{j-1}}\frac{1}{(j-1)!}=-c_0\frac{\alpha}{\alpha+2}\cdot{}_3F_2\left[1-\frac{\alpha}{2},1,1;\frac{\alpha}{2}+2,2;1\right].
		\end{align*}
		The above hypergeometric function series converges since $\frac{\alpha}{2}+2+2-\left(1-\frac{\alpha}{2}\right)-1-1=1+\alpha>0$ for all $-1<\alpha<0$.
		
		Analogously to the proof of Lemma \ref{lemma1} we can compute 
		\begin{align*}
			\int_0^1\vert T_\alpha f(y)\vert\leq \int_0^\infty C_\alpha\left(\frac{1}{\pi}+\frac{1}{x}\right)\vert f(x)\vert dx=C_\alpha\Bigg(\frac{1}{\pi}\underbrace{\Vert f\Vert_1}_{\leq\Vert f \Vert_D}+\underbrace{\int_0^\infty\vert f(x)\vert\frac{dx}{x}}_{\substack{\leq \int_0^\infty\vert f(x)\vert\max\left\{\frac{1}{x},1\right\}dx\\[1ex]\hspace{-11ex}\leq\Vert f\Vert_D}}\Bigg)\leq C_\alpha\left(\frac{1}{\pi}+1\right)\Vert f\Vert_D,
		\end{align*}
	where $C_\alpha=\int_0^\pi\vert\sin(u)\vert^\alpha du=\sqrt{\pi}\frac{\Gamma\left(\frac{1+\alpha}{2}\right)}{\Gamma\left(1+\frac{\alpha}{2}\right)}$. Moreover, 
	\begin{align*}
		\sup\limits_{y>1}\left\vert T_\alpha f(y)\right\vert&\leq \frac{c_0}{2}\underbrace{\mathcal{F}f(0)}_{=2\Vert f\Vert_1}+\sup\limits_{y>1}\sum\limits_{j=1}^\infty c_j\left\vert\mathcal{F}f(2jy)\right\vert=c_0\Vert f\Vert_1+\sup\limits_{y>1}\sum\limits_{j=1}^\infty c_j\frac{1}{2jy}\underbrace{\vert\mathcal{F}f'(2jy)\vert}_{\leq 2\Vert f'\Vert_1}\\
		&\leq c_0\Vert f\Vert_D+\sup\limits_{y>1}\frac{1}{y}\Vert f'\Vert_1\sum\limits_{j=1}^\infty\frac{c_j}{j}\\
		&\leq c_0\Vert f\Vert_D-\Vert f\Vert_Dc_0\frac{\alpha}{\alpha+2}\cdot{}_3F_2\left[1-\frac{\alpha}{2},1,1;\frac{\alpha}{2}+2,2;1\right]\\[1ex]
		&=\Vert f\Vert_D c_0\left(1-\frac{\alpha}{\alpha+2}\cdot{}_3F_2\left[1-\frac{\alpha}{2},1,1;\frac{\alpha}{2}+2,2;1\right]\right).
	\end{align*}
	Ultimately, the operator norm is bounded by 
	\begin{align*}
		\Vert T_\alpha\Vert=\sup\limits_{f\in \left(L^1_{w,\alpha}(\mathbb{R_+})\cap W^{1,1}(\mathbb{R}_+)\right)\setminus\{0\}}\frac{\Vert T_\alpha f\Vert_\ast}{\Vert f\Vert_D}\leq C_\alpha\left(\frac{1}{\pi}+1\right)+c_0\left(1-\frac{\alpha}{\alpha+2}\cdot{}_3F_2\left[1-\frac{\alpha}{2},1,1;\frac{\alpha}{2}+2,2;1\right]\right).
	\end{align*}
	\end{enumerate}
\end{proof}

Similar to Theorem \ref{prop2}, it is also possible to give a series representation of the \emph{$\alpha$-cosine transform on $\mathbb{R}_+$}, which we define by
\begin{align}
\label{costrafo}
K_\alpha f(y)=\int_0^\infty\left\vert\cos\left(xy\right)\right\vert^\alpha f(x)dx, \quad y\geq0,~f\in L^1_{w,\alpha}(\mathbb{R}_+).
\end{align}
\begin{cor}
	\label{coscorollary}
	Let $\alpha>-1$. Any function $f\in L^1_{w,\alpha}(\mathbb{R}_+)$ is mapped by $K_\alpha$ onto
	\begin{align*}
	K_\alpha f(y)=\frac{\tilde{c}_0}{2}\mathcal{F}f(0)+\sum\limits_{j=1}^\infty\tilde{c}_j\mathcal{F}f(2jy),\quad y\geq 0
	\end{align*}
	with coefficients $\tilde{c}_j$, $j=0,1,\dots$, given by
	\begin{align*}
	\tilde{c}_j=(-1)^jc_j=\begin{cases}
			\frac{1}{2^{\alpha}}\frac{\Gamma\left(1+\alpha\right)}{\Gamma\left(\frac{\alpha}{2}-j+1\right)\Gamma\left(\frac{\alpha}{2}+j+1\right)},&\alpha\neq 2k, ~j\in\mathbb{N}_0,\\
	\frac{1}{4^k}\binom{2k}{k-j},&\alpha=2k,~j=0,\dots,k,~k\in\mathbb{N}_0,\\
	0,&\alpha=2k,~j>k,~k\in\mathbb{N}_0.
	\end{cases}
	\end{align*}
\end{cor}
\begin{proof}
	The result follows immediately from the equivalent formulation of the cosine double angle formula given by $\cos(2x)=2\cos^2(x)-1$, which yields
	\begin{align*}
		\frac{1}{2}\left\vert\cos\left(\frac{x}{2}\right)\right\vert^\alpha=\frac{1}{2}\left(\cos^2\left(\frac{x}{2}\right)\right)^{\alpha/2}=\frac{1}{2}\left(\frac{1+\cos(x)}{2}\right)^{\alpha/2}=\frac{1}{2^{\alpha/2+1}}\sum\limits_{k=0}^\infty\binom{\alpha/2}{k}\cos^k(x).
	\end{align*}
\end{proof}

\begin{rem}
			Counterparts of Corollary \ref{corollary_ck} and Theorem \ref{boundedoperator} for $K_\alpha$ follow immediately from $\vert\tilde{c}_j\vert=\vert c_j\vert$.
			The only exception is the convergence of the series $\sum\limits_{j=1}^\infty\tilde{c}_j$ for $-1<\alpha<0$. This can be easily seen by the Leibniz criterion. More precisely, by Equation (\ref{cjhypergeom}) it holds that 
			\begin{align*}
				\tilde{c}_j=(-1)^jc_j=(-1)^{j-1}c_0\frac{\alpha}{\alpha+2}\frac{\left(1-\frac{\alpha}{2}\right)_{j-1}}{\left(\frac{\alpha}{2}+2\right)_{j-1}},
			\end{align*}
			hence
			\begin{align*}
			\sum\limits_{j=1}^\infty\tilde{c}_j=c_0\frac{\alpha}{\alpha+2}\sum\limits_{j=1}^\infty\frac{(1)_{j-1}\left(1-\frac{\alpha}{2}\right)_{j-1}}{\left(\frac{\alpha}{2}+2\right)_{j-1}}\frac{(-1)^{j-1}}{(j-1)!}=c_0\frac{\alpha}{\alpha+2}\cdot{}_2F_1\left[1,1-\frac{\alpha}{2};2+\frac{\alpha}{2};-1\right].
			\end{align*}
			The hypergeometric series above converges since $2+\frac{\alpha}{2}-1-\left(1-\frac{\alpha}{2}\right)+1=\alpha+1>0$ \ for $\alpha\in(-1,0)$.
\end{rem}

\subsection{Approximating the Fourier transform}
\label{Fourierapproximation}

Denote by $\hat{f}=\mathcal{F}f$ the Fourier transform of $f\in L_{e,w,\alpha}^1(\mathbb{R})$. Let $\hat{f}_R=\hat{f}\mathbbm{1}\{ \vert y\vert\leq R\}$, $R>0$ be its restriction to $[-R,R]$.
By the Fourier transform's symmetry it suffices to approximate $\hat{f}_R$ at equidistant points $\left\{n\frac{R}{N}\right\}$, $n=1,\dots,N$, $N\in\mathbb{N}$, from the series representation of $T_\alpha f$ in Theorem \ref{prop2}. Define vectors $\bm{\xi}, \bm{\eta}\in\mathrm{R}^N$ with elements
\begin{align*}
	\xi_n=\hat{f}_R\left(n\frac{R}{N}\right),
\end{align*}
and
\begin{align*}
	\eta_n=T_\alpha f\left(n\frac{R}{2N}\right)-\frac{c_0}{2}\mathcal{F}f(0)
\end{align*}
for $n=1,\dots,N$, as well as the matrix 
$\bm{C}\in\mathbb{R}^{N\times N}$ with
\begin{align}
	\label{cij}
	\bm{C}_{i,j}=\begin{cases}
		c_k,& i<j\text{ with }j=ki\text{ for some }k\in\mathbb{N}.\\
		0, &\text{else},
	\end{cases}
\end{align}
where the matrix elements $c_k$ are the Fourier coefficients from Theorem \ref{prop2}.
\begin{prop}
	\label{propmatrix}
	The vector $\bm{\xi}$ is uniquely determined by 
 \begin{align}
 	\label{syseqsol}
 	\bm{\xi}=\bm{C}^{-1}\cdot\bm{\eta}~,
 \end{align}
i.e. it is the unique solution of the system of linear equations
$
\bm{\eta}=\bm{C}\cdot\bm{\xi}
$.
\end{prop}
\begin{proof}
	First note that $\hat{f}_R$ is continuous on the compact interval $[-R,R]$, as well as integrable and square-integrable on the real line. Furthermore, $\hat{f}_R$ converges in the $L^2$-norm to $\hat{f}$, hence
		$\mathcal{F}^{-1}\hat{f}_R\rightarrow\mathcal{F}^{-1}\hat{f}=f$
	in the $L^2$-norm as $R\rightarrow\infty$. By the Riemann-Lebesgue lemma, the Fourier transform $\hat{f}$ is bounded, uniformly continuous, and vanishes at infinity \cite{clasfour}, hence we approximate $\hat{f}$ by $\hat{f}_R$ in Equation (\ref{TfFourier2}) with $R$ chosen large enough. Then, 
	\begin{align*}
		\eta_n=T_\alpha f\left(n\frac{R}{2N}\right)-\frac{c_0}{2}\mathcal{F}f(0)=\sum\limits_{j=1}^\infty c_j\hat{f}\left(2jn\frac{R}{2N}\right)\approx\sum\limits_{j=1}^{\lfloor N/n\rfloor}c_j\hat{f}_R\left(jn\frac{R}{N}\right)=\sum\limits_{j=1}^{\lfloor N/n\rfloor}c_j\xi_{jn}
	\end{align*}
	for $n=1,\dots,N$ forms a system of linear equations, which in matrix form is given by 
	\begin{align*}
		\bm{\eta}=\bm{C}\cdot\bm{\xi}~,
	\end{align*}
	where $\bm{C}=\left(\bm{C}_{i,j}\right)_{i,j=1,\dots,N}$ with $\bm{C}_{i,j}$ as in Equation (\ref{cij})
	, i.e.
	\begin{align*}
		\bm{C}=
		\begin{tikzpicture}[baseline=-0.5ex]
			\matrix (m) [matrix of math nodes,
			left delimiter=(,
			right delimiter=),
			inner sep=2.5pt, 
			ampersand replacement=\&]{
				c_1	\&c_2	\&c_3	\&c_4	\&c_5	\&c_6	\&c_7\&c_8\&c_9\&\cdots\&c_{N-2}\&c_{N-1}\&c_N\\
				0	\&c_1	\&0		\&c_2		\&0		\&c_3	\&0 \&c_4\&0\&\cdots\&c_{N/2-1}\&0\&c_{N/2}\\
				0	\&0		\&c_1	\&0		\&0		\&c_2	\&0\&0\&c_3\&\cdots\&\cdots\&\cdots\&\cdots\\
				0	\&0		\&0	\&c_1		\&0		\&0	\&0\&c_2\&0	\&\cdots\&\cdots\&\cdots\&\cdots\\[5ex]
				\&\&\&\&\&\&\&\&\&\&0\&c_1\&0\\
				0\&\&\&\&\&\&\&\&\&\&\&0\&c_1\\
			} ;
			\draw[loosely dotted, thick] (m-4-1)-- (m-6-1);
			\draw[loosely dotted, thick] (m-6-1)-- (m-6-12);
			\draw[loosely dotted, thick] (m-4-4)-- (m-5-12);
		\end{tikzpicture}
		~.
	\end{align*}
In the illustration of the matrix $\bm{C}$ above it is assumed, without loss of generality, that $N$ is an even integer. 
By the construction of the system of linear equations, $\bm{C}$ is an upper triangular matrix with non-zero entries $c_1$ on its main diagonal. In particular, this means that its inverse $\bm{C}^{-1}$ exists, and $\bm{\xi}$ is uniquely determined by
\begin{align*}
	\bm{\xi}=\bm{C}^{-1}\cdot\bm{\eta}~.
\end{align*}
\end{proof}

\subsection{Band-limited interpolation}
\label{BLM}

We aim to reconstruct the Fourier transform $\hat{f}$ with a suitable interpolation. In the previous subsection we introduced the vector $\bm{\xi}$ with coordinates $\xi_n=\hat{f}(nR/N)$, $n=1,\dots,N$. Set $\hat{f}(-nR/N)=\hat{f}(nR/N)$ and $\hat{f}(0)=\mathcal{F}f(0)$, which we assume to be known for now. Note that this is in general not true. A simple workaround solving this is discussed in Section \ref{num}. Moreover, in the application examples of Section \ref{appl}, the value $\mathcal{F}f(0)$ is further specified.

Define the band-limited interpolation of the Fourier transform $\hat{f}$ by
\begin{align}
	\label{bli_interp}
	\widehat{f^{(N)}_R}=\sum\limits_{n=-N}^N\hat{f}\left(n\frac{R}{N}\right)\text{sinc}\left(\frac{N}{R}y-n\right),
\end{align}
where $\text{sinc}(x)=\sin\left(\pi x\right)/(\pi x)$ for $x\neq 0$ and $\text{sinc}(0)=1$. Furthermore, define the estimate 
\begin{align}
	\label{bli_inv_est}
	f^{(N)}_R(x)=\frac{R}{2\pi N}rect\left(\frac{xR}{2\pi N}\right)\sum\limits_{n=-N}^N\hat{f}\left(n\frac{R}{N}\right)e^{-ixnR/N},
\end{align}
where 
\begin{align*}
	rect(t)=\begin{cases}
		1,& \vert t\vert<1/2,\\
		0,&\vert t\vert>1/2,\\
		1/2,& t=1/2.
	\end{cases}
\end{align*}

	\begin{thm}
		\label{lastcor}
		Let $f\in L^1_{e,w,\alpha}(\mathbb{R})\cap L^2_e(\mathbb{R})$ and $B=N/(2R)$. Additionally, assume that $f\in C_e(\mathbb{R})$ with integrable first derivative $f'\in L_e^1(\mathbb{R})$.
		\begin{enumerate}[(i)]
			\item The band-limited interpolation $\widehat{f^{(N)}_R}$ converges to $\hat{f}$ in the $L^2$-norm, i.e. it holds that 
			\begin{align*}
				\lim\limits_{B\rightarrow\infty}\lim\limits_{N\rightarrow\infty}\left\Vert \hat{f}-\widehat{f^{(N)}_R}\right\Vert_2=0.
			\end{align*}
		\item The estimate $f^{(N)}_R(x)$ converges to $f$ in the $L^2$-norm. It holds that 
	\begin{align*}
		\lim\limits_{B\rightarrow\infty}\lim\limits_{N\rightarrow\infty}\left\Vert f-f_R^{(N)}\right\Vert_2=0.
	\end{align*}
		\end{enumerate}
\end{thm}
\begin{proof}
\begin{enumerate}[(i)]
	\item Define the cardinal series of $\hat{f}$ by
	\begin{align*}
		S_B\hat{f}(y)=\sum\limits_{n=-\infty}^\infty\hat{f}\left(\frac{n}{2B}\right)\text{sinc}\left(2By-n\right).
	\end{align*}
	The \emph{Whittaker-Shannon-Kotel'nikov sampling theorem} \cite{bracewell,marks,zayed,anastassiou}, widely known in the field of sampling theory, states that any bandlimited function $u$, i.e. $u\in L^p(\mathbb{R})$, $1\leq p<\infty$, with compactly supported Fourier transform $\hat{u}$ on $[-B,B]$, can be completely reconstructed in from its cardinal series $S_B u$ in $L^p$ \cite[Thm. 6]{anastassiou}. The optimal sampling rate $2B$ is called \emph{Nyquist-rate}. In our case $u=\hat{f}$.
	
	Denote by
	\begin{align}
		\label{fhattrunccardinal}
		S^{(N)}_B\hat{f}(y)=\sum\limits_{n=-N}^N\hat{f}\left(\frac{n}{2B}\right)\text{sinc}\left(2By-n\right),
	\end{align}
	the truncated cardinal series of $\hat{f}$.
	Suppose that $\hat{f}$ is band-limited, i.e. the function $f\in L^1_{e,w,\alpha}(\mathbb{R})\cap L^2_{e}(\mathbb{R})$ has compact support $[-B,B]$ for some $B>0$.
	Then, $\hat{f}\in L^2_{e}(\mathbb{R})$ is completely determined by its cardinal series $S_B\hat{f}$, and 
	\begin{align}
		\label{L2trunc}
		S_B^{(N)}\hat{f}\rightarrow S_B\hat{f}=\hat{f},
	\end{align}
	in the $L^2$-sense as $N\rightarrow\infty$. 
	
	The assumption that $f$ has compact support is certainly too strong, and the Shannon sampling theorem applied to non-bandlimited functions will result in interpolation errors known as \emph{aliasing}. To ensure $L^2$-convergence of the bandlimited interpolation when $\hat{f}$ is not band-limited, we refer to the results in \cite{rahmanvertesi}. For $p>1$, define 
	\begin{align*}
		\mathcal{F}^p(\delta)=\left\{u:\mathbb{R}\rightarrow\mathbb{C}\text{ measurable, }u(x)=O\left(\frac{1}{\left(1+\vert x\vert\right)^{1/p+\delta}}\right)\right\},
	\end{align*}
	and $\mathcal{F}^p=\cup_{\delta>0}\mathcal{F}^p(\delta)$. Denote by $\mathcal{R}$ the set of all measurable functions $g:\mathbb{R}\rightarrow\mathbb{C}$ that are Riemann-integrable on every finite interval. Then, 
	\begin{align*}
		\left\Vert u-S_Bu\right\Vert_p\rightarrow 0
	\end{align*}
	as $B\rightarrow\infty$ for all $g\in\mathcal{F}^p\cap\mathcal{R}$. We need $\left\Vert\hat{f}-S_B\hat{f}\right\Vert_{L^2}\rightarrow 0$ as $B\rightarrow\infty$, i.e. $\hat{f}\in\mathcal{F}^2(\delta)$ needs to hold for some $\delta>0$. 
	
	Recall, that for any function $f\in C^n(\mathbb{R})$ with integrable derivatives $f^{(k)}\in L^1(\mathbb{R})$ for all $k\leq n$, its Fourier transform $\hat{f}$ decays as
		$O\left(\frac{1}{\left(1+\vert y\vert\right)^n}\right)$, $\vert y\vert\rightarrow\infty.$
	Setting $n=1$ yields the desired result if additionally to $f\in L^1_{e,w,\alpha}(\mathbb{R})\cap L^2_e(\mathbb{R})$ it is assumed that $f\in C^1_e(\mathbb{R})$ with integrable first derivative $f'\in L^1_e(\mathbb{R})$. Then, 
	$\hat{f}(y)=O\left(1/(1+\vert y\vert)^{1/2+\delta}\right)$ 
	with $\delta=1/2$. Furthermore, $\hat{f}$ is uniformly continuous on $\mathbb{R}$ and bounded, in particular it is integrable over all finite intervals. Hence, 
	\begin{align*}
		\left\Vert\hat{f}-S_B\hat{f}\right\Vert_2\rightarrow 0
	\end{align*}
	as $B\rightarrow\infty$, i.e. the cardinal series $S_B\hat{f}$ converges to $\hat{f}$ in the $L^2$-norm. The above and the $L^2$-convergence in (\ref{L2trunc}) imply that 
	\begin{align*}
		\lim\limits_{B\rightarrow\infty}\lim\limits_{N\rightarrow\infty}\left\Vert \hat{f}-S_B^{(N)}\hat{f}\right\Vert_2\leq
		\lim\limits_{B\rightarrow\infty}\lim\limits_{N\rightarrow\infty}\left(\left\Vert \hat{f}-S_B\hat{f}\right\Vert_2+\left\Vert S_B\hat{f}-S_B^{(N)}\hat{f}\right\Vert_2\right)=0.
	\end{align*}
	Setting $2B=N/R$ in the truncated cardinal series (\ref{fhattrunccardinal}) yields the interpolant
	\begin{align*}
		\widehat{f^{(N)}_R}=\sum\limits_{n=-N}^N\hat{f}\left(n\frac{R}{N}\right)\text{sinc}\left(\frac{N}{R}y-n\right).
	\end{align*}
\item Note that 
\begin{align*}
	\left\{\mathcal{F}^{-1}\text{sinc}\left(2B\cdot-n\right)\right\}(x)=\frac{1}{4\pi B}rect\left(\frac{x}{4\pi B}\right) e^{-ixn/(2B)},
\end{align*}
and applying the inverse Fourier transform $\mathcal{F}^{-1}$ to the band-limited interpolation $\widehat{f_R^{(N)}}$ yields
\begin{align*}
	f^{(N)}_R(x)=\frac{R}{2\pi N}rect\left(\frac{xR}{2\pi N}\right)\sum\limits_{n=-N}^N\hat{f}\left(n\frac{R}{N}\right)e^{-ixnR/N}.
\end{align*}
As the $L^2$-convergence of the interpolant to $\hat{f}$ is preserved under the inverse Fourier transform, we ultimately get 
\begin{align*}
	\lim\limits_{B\rightarrow\infty}\lim\limits_{N\rightarrow\infty}\left\Vert f-f_R^{(N)}\right\Vert_2=0.
\end{align*}
\end{enumerate}
\end{proof}

\begin{rem}
	\label{linintremark}
		\begin{enumerate}[(i)]
			\item Under the assumptions of Theorem \ref{lastcor} the cardinal series $S_B\hat{f}$ converges uniformly to $\hat{f}$ as $B\rightarrow\infty$ with error estimate 
			\begin{align*}
				\left\Vert\hat{f}-S_B\hat{f}\right\Vert_\infty
				\leq\frac{1}{\pi}\int\limits_{\vert x\vert>B}\vert f(x)\vert dx.
			\end{align*}
			This is a direct consequence of \cite[Theorem 1]{brown}. 
			\item Alternatively, it is also possible to use piecewise linear interpolation. Denote by $\hat{f}_{N,R}$ the piecewise linear interpolation polynomial constructed from the interpolation points $\bm{\xi}=\left(\hat{f}(kR/N)\right)_{k=-N}^N$ of Proposition \ref{propmatrix}. 
			It is well known that the Fourier transform $\hat{f}$ is $n$-times continuously differentiable if the function $f$ is piecewise continuous and $x^kf(x)$ is integrable for all $k\leq n$ \cite{kolmfomin}. Restricted to compact intervals, boundedness follows. 
			Hence, if $f\in L^1_{e,w,\alpha}$ is additionally assumed to satisfy $xf(x), x^2f(x)\in L^1(\mathbb{R})$, then $\hat{f}$ is twice continuously differentiable. It follows that linear interpolation $\hat{f}_{N,R}$ converges to $\hat{f}_R$ uniformly as $N\rightarrow \infty$, and the error estimate
%
			\begin{align*}
				\left\Vert \hat{f}_R-\hat{f}_{N,R}\right\Vert_\infty\leq C\left(\frac{R}{N}\right)^{2}\left\Vert (\hat{f}_R)^{''}\right\Vert_\infty
			\end{align*}
			holds, where $C$ is some positive constant \cite[Ch. 8.3]{quarteroni} and $(\hat{f}_R)''$ denotes the second derivative of $\hat{f}_R$. 
			This implies convergence in the $L^2$-norm over $[0,R]$ which is preserved by the Fourier transform, thus
			\begin{align*}
				\mathcal{F}^{-1}\hat{f}_{N,R}\rightarrow\mathcal{F}^{-1}\hat{f}_R
			\end{align*}
			in the $L^2$-norm as $N\rightarrow\infty$. 
			Furthermore, $\mathcal{F}^{-1}\hat{f}_{N,R}\rightarrow\mathcal{F}^{-1}\hat{f}_R$  in the $L^2$-norm as $N\rightarrow\infty$. 
		\end{enumerate}
\end{rem}

\subsection{Smoothing}
\label{mollifier}
	In applications, there is a possibility that the function $T_\alpha f$ is contaminated by noise, which directly affects the solution of the linear equation $\bm{\eta}=\bm{C}\cdot\bm{\xi}$ from Proposition \ref{propmatrix}. 
	To cope with noisy data of $T_\alpha f$, we fix a non-negative, even \emph{mollifier function} $e_\gamma(x)$, $\gamma>0$, that integrates to $1$, and converges to the Dirac-Delta function as $\gamma$ tends to $0$. Additionally, determine the respective \emph{reconstruction kernel} $\psi_\gamma$ by the relationship
$
	\psi_\gamma=\mathcal{F}e_\gamma~.
	$
	Suitable mollifiers and their respective reconstruction kernels are given in Example \ref{ex2} of Section \ref{noisy}. A smoothed estimate of $f$ can then be computed by the following Proposition. 
\begin{prop}
Under the assumptions of Theorem \ref{lastcor}, a smoothed estimate of $f$ is given by 
\begin{align}
	f^{(N)}_{R,\gamma}=\mathcal{F}^{-1}\left(\widehat{f^{(N)}_R}\cdot\psi_\gamma\right)\rightarrow f
\end{align}
as $\gamma\rightarrow 0$, $N\rightarrow\infty$, $R\rightarrow\infty$, where $\widehat{f^{(N)}_R}=S_B$ is the interpolation given in Equation (\ref{bli_interp}). The convergence to the function $f$ is understood in the $L^2$-sense. Uniform convergence is guaranteed if $f\in C_b(\mathbb{R})$.
\end{prop}
	\begin{proof}
	For a given mollifier $e_\gamma$ we consider the convolution 
	\begin{align}
		\label{fgamma}
		f_\gamma(x)=\left(f\ast e_\gamma\right)(x)=\mathcal{F}^{-1}\left(\mathcal{F}f\cdot\mathcal{F}e_\gamma\right)(x)=\mathcal{F}^{-1}\left(\hat{f}\cdot\psi_\gamma\right)(x),
	\end{align}
	which converges to the function $f$ in the $L^2$-norm as $y\rightarrow\infty$ since $f\in L^2_e(\mathbb{R})$. Uniform convergence is guaranteed if $f\in C_b(\mathbb{R})$ \cite[Theorem 1.2.19]{clasfour}. 
	Applying (\ref{fgamma}) to $\widehat{f^{(N)}_R}$, we then compute the smoothed estimate 
	\begin{align*}
		f^{(N)}_{R,\gamma}=\mathcal{F}^{-1}\left(\widehat{f^{(N)}_R}\cdot\psi_\gamma\right),
	\end{align*}
where $L^2$-convergence is preserved by the inverse Fourier transform.
	\end{proof}

\section{Applications}
\label{appl}

In the following Section, we briefly present two applications for the inversion of $\alpha$-sine and cosine transforms. First, harmonizable symmetric $\alpha$-stable processes as well as their connection to $\alpha$-sine transforms are outlined. The theory of complex stable measures and stochastic integrals can be quite technical, and a detailed discussion would go beyond the scope of this work. We refer to \cite{gennady} for a complete introduction to these processes. The second application deals with the inversion of the two-dimensional spherical $\alpha$-cosine transform.

\subsection{Stationary real harmonizable symmetric $\alpha$-stable processes}

Consider a probability space $(\Omega,\mathcal{F},P)$. Define the spaces $L^0(\Omega)$ and $L^0_c(\Omega)$ of real-valued and complex valued random variables on this probability space, respectively. Then, every element $X\in L^0_c(\Omega)$ is of the form $X=X_1+iX_2$ with $X_1,X_2\in L^0(\Omega)$.
A real-valued \emph{symmetric $\alpha$-stable random variable $X\sim S\alpha S(\sigma)$} on this probability space is defined by the characteristic function 
\begin{align*}
	\mathbb{E}\left[e^{isX}\right]=\exp\left\{-\sigma^\alpha\vert s\vert^\alpha\right\},
\end{align*}
where $\sigma>0$ is called the \emph{scale parameter} of $X$ and $\alpha\in(0,2]$ its \emph{index of stability}. In the multivariate case, a real-valued symmetric $\alpha$-stable random vector $\bm{X}=(X_1,\dots,X_n)$ is defined by its joint characteristic function
\begin{align*}
	\mathbb{E}\left[\exp\left\{i(\bm{s},\bm{X})\right\}\right]=\exp\left\{-\int\limits_{S^{n-1}}\left\vert(\bm\theta,\bm{s})\right\vert^\alpha\Gamma(d\bm{\theta})\right\},
\end{align*}
where $(\bm{x},\bm{y})$ denotes the scalar product of two vectors $\bm{x},\bm{y}\in\mathbb{R}^n$, and $S^{n-1}$ is the unit sphere in $\mathbb
R^n$. The measure $\Gamma$ is called the \emph{spectral measure} of $\bm{X}$. It is unique, finite and symmetric in the case $0<\alpha<2$ \cite[Thm. 2.4.3]{gennady}.

Let us define the notion of complex random measure. 
Let $(E,\mathcal{E})$ be a measurable space, and let $(S^1,\mathcal{B}(S^1))$ be the measurable space on the unit circle $S^1$ equipped with the Borel $\sigma$-algebra $\mathcal{B}(S^1)$. Let $k$ be a measure on the product space $(E\times S^1,\mathcal{E}\times\mathcal{B}(S^1))$, and let
$$\mathcal{E}_0=\left\{A\in\mathcal{E}:k(A\times S^1)<\infty\right\}.$$
A \emph{complex-valued $S\alpha S$ random measure} on $(E,\mathcal{E})$ is an independently scattered, $\sigma$-additive, complex-valued set function 
$$M:\mathcal{E}_0\rightarrow L^0_c(\Omega)$$ 
such that the real and imaginary part of $M(A)$, i.e. the vector $(M^{(1)}(A),M^{(2)}(A))=\left(Re(M(A)),Im(M(A))\right)$, is jointly $S\alpha S$ with spectral measure $k(A\times\cdot)$ for every $A\in\mathcal{E}_0$ \cite[Def. 6.1.2]{gennady}. We refer to $k$ as the \emph{circular control measure} of $M$, and denote by $m(A)=k(A\times S^1)$ the \emph{control measure} of $M$. Furthermore, $M$ is isotropic if and only if its circular control measure is of the form $$k=m\gamma,$$ where $\gamma$ is the uniform probability measure on $S^1$.

A stochastic integral with respect to a complex $S\alpha S$ random measures $M$ is defined by
\begin{align*}
	I(u)=\int\limits_Eu(x)M(dx),
\end{align*}
where $f:E\rightarrow\mathbb{C}$ is a measurable, complex-valued function from the space $L^\alpha(E,m)$. The integral $I(u)$ is well-defined by \cite[Prop. 6.2.2]{gennady}. 

Let $(E,\mathcal{E})=(\mathbb{R},\mathcal{B})$ and $u(t,x)=e^{itx}$. Then, the stochastic process $\left\{X(t):t\in\mathbb{R}\right\}$ defined by 
\begin{align*}
	X(t)=Re\left(I\left(u(t,\cdot)\right)\right)=Re\left(\int\limits_\mathbb{R}e^{itx}M(dx)\right),
\end{align*}
where $M$ is a complex $S\alpha S$ random measure on $(\mathbb{R},\mathcal{B})$ with finite circular control measure $k$ (equivalently, with finite control measure $m$), is called a \emph{real harmonizable $S\alpha S$ process}. 

By \cite[Thm. 6.5.1]{gennady}, this process is stationary if and only if $M$ is isotropic, i.e. its spectral measure is of the form $k=m\gamma$. 
For all $n\in\mathbb{N}$ and $t_1,\dots,t_n\in\mathbb{R}$, the characteristic function of the finite-dimensional distributions of the process $X$ is given by 
\begin{align}
\label{charfn}
\mathbb{E}\left[\exp\left\{i\sum\limits_{i=1}^ns_iX_{t_i}\right\}\right]=\exp\left\{-\lambda_\alpha\int\limits_\mathbb{R}\left\vert\sum\limits_{j,k=1}^ns_js_k\cos\left(\left(t_k-t_j\right)x\right)\right\vert^{\alpha/2}m(dx)\right\}
\end{align}
with constant
$
	\lambda_\alpha=\frac{1}{2\pi}\int\limits_0^{2\pi}\vert\cos\left(x\right)\vert^\alpha dx
$
, see \cite[Proposition 6.6.3]{gennady}.

The \emph{codifference function} $\tau(t)$ of a $S\alpha S$ stochastic process is defined as the codifference of the random variables $X_0$ and $X_t$, i.e.
\begin{align}
\label{codiff}
	\tau(t)=\Vert X_0\Vert_\alpha^\alpha+\Vert X_t\Vert_\alpha^\alpha-\Vert X_t-X_0\Vert_\alpha^\alpha,
\end{align}
where $\Vert\cdot\Vert_\alpha$ is the scale parameter of $X_0, X_t$ and $X_t-X_0$, respectively. 

\subsection{Problem setting and inversion}
Let $X=\left\{X_t:t\in\mathbb{R}\right\}$ be a stationary real harmonizable $S\alpha S$ process with finite circular control measure $k$, and suppose its control measure $m$ has an even, non-negative density function $f\in L^1_e(\mathbb{R})\cap L^2_e(\mathbb{R})$ with respect to the Lebesgue measure on $\mathbb{R}$. We refer to the function $f$ as the \emph{spectral density}. It follows that $m(dx)=f(x)dx$ in Equation (\ref{charfn}), and the random variables $X_0, X_t$ and $X_t-X_0$ in (\ref{codiff}) are $S\alpha S$ random variables with the following scale parameters.

First, it holds that 
$X_t\sim S\alpha S(\sigma)$ with 
\begin{align*}
	\sigma=\left(\lambda_\alpha\int\limits_\mathbb{R}f(x)dx\right)^{1/\alpha}=\left(\lambda_\alpha m(\mathbb{R})\right)^{1/\alpha}~.
\end{align*}
In particular, this yields $m(\mathbb{R})=\sigma^\alpha/\lambda_\alpha$. By stationarity of $X$, the random variable $X_0$ has the same scale parameter $\sigma$. 

Similarly, for $X_t-X_0$ Equation (\ref{charfn}) yields 
\begin{align*}
	\mathbb{E}\left[is \left(X_t-X_0\right)\right]
	&=\exp\left\{-\lambda_\alpha\int\limits_\mathbb{R}\Big\vert s ^2\Big(2-cos(tx)-{cos(-tx)}\Big)\Big\vert^{\alpha/2}f(x)dx \right\}\\
	&=\exp\left\{-\vert s \vert^\alpha \lambda_\alpha\left(\int\limits_\mathbb{R}\Big\vert 2-2{cos(tx)}\Big\vert^{\alpha/2}f(x)dx\right) \right\}\\
	&=\exp\left\{-\vert s \vert^\alpha2^\alpha \lambda_\alpha\left(\int\limits_\mathbb{R}\left\vert \sin\left(\frac{tx}{2}\right)\right\vert^{\alpha}f(x)dx\right)\right\}~.
\end{align*}
It follows that $X_t-X_0\sim S\alpha S(\tilde\sigma)$ with
\begin{align*}
	\tilde\sigma=2\left(\lambda_\alpha\int\limits_\mathbb{R}\left\vert \sin\left(\frac{tx}{2}\right)\right\vert^{\alpha}f(x)dx\right)^{1/\alpha}~.
\end{align*}
Ultimately, the codifference function in Equation (\ref{codiff}) expands to 
\begin{align*}
	\tau(t)&=2\sigma^\alpha-\tilde\sigma^\alpha=2\sigma^\alpha-2^\alpha \lambda_\alpha\int\limits_\mathbb{R}\left\vert\sin\left(\frac{tx}{2}\right)\right\vert^{\alpha}f(x)dx.
\end{align*}
 Suppose that the spectral density $f$ is an even, integrable function. Then rearranging the above yields
 \begin{align}
 \label{gfromtau}
  T_\alpha f\left(t\right)=\int\limits_0^\infty\left\vert\sin\left(tx\right)\right\vert^{\alpha}f(x)dx
  =\frac{2\sigma^\alpha-\tau(2t)}{2^{\alpha+1} \lambda_\alpha}
  =:g(t).
\end{align}

Assume that the scale parameter $\sigma$, the index of stability $\alpha\in(0,2)$ and the codifference function $\tau$ are known. Then, $\mathcal{F}f(0)=m(\mathbb{R})=\sigma^\alpha/\lambda_\alpha$.
Furthermore, set $\bm{\eta}=(\eta_1,\dots,\eta_N)$ with $\eta_n=g\left(n\frac{R}{2N}\right)-\frac{c_0\sigma^\alpha}{2\lambda_\alpha}$, $n=1,\dots, N$, and solve the system of linear equations $\bm{\eta}=\bm{C}\bm{\xi}$, where $\bm{\xi}=(\xi_1,\dots,\xi_N)$ with $\xi_n=\hat{f}\left(n\frac{R}{N}\right)$, $n=1,\dots, N$. Apply Theorem \ref{lastcor} to compute the estimate of the spectral density $f$.

\begin{rem}
	In statistical inference, the scale parameter $\sigma$, the index of stability $\alpha\in(0,2)$ and the codifference function $\tau$ need to be estimated from a realization of the process $X$. 
\end{rem}

\subsection{Two-dimensional spherical $\alpha$-cosine transform}

As mentioned in the introduction before, spherical $\alpha$-cosine transforms are of particular interest in stochastic and convex geometry. For example, when analyzing fiber processes, the so-called rose of intersections is closely related to the $\alpha$-cosine transform of the underlying directional distribution of the fibers. Also, in convex geometry, the support function of a zonoid is the spherical $\alpha$-cosine transform of some generating measure $\rho$. 

Let $f$ be an even, integrable function on $S^n$(i.e. the unit sphere in $\mathbb{R}^{n+1}$) with respect to the area surface measure on $S^n$. We write $f\in L^1_e(S^n)$.
Define the \emph{$(n+1)$-dimensional spherical $\alpha$-cosine transform} of $f$ by $K_{\alpha,S^n}:L^1_e(S^n)\rightarrow C_e(S^n)$ with
\begin{align}
\label{spherecos}
	K_{\alpha,S^n} f(\bm{\eta})=\int\limits_{S^n}\left\vert(\bm{\theta},\bm{\eta})\right\vert^\alpha f(\bm{\theta})d\bm{\theta},\quad \bm{\eta}\in S^n,
\end{align}
where $(\cdot,\cdot)$ denotes the inner product of the vectors $\bm{\theta},\bm{\eta}\in S^n$. The spherical $\alpha$-cosine transform is well-defined for $\alpha>-1$, see \cite{invformulas}.

Every point $\bm{\theta}=(\cos(x),\sin(x))$ on the unit circle $S^1$ corresponds one-to-one to an angle $x\in(-\pi,\pi]$, and any function $f\in L^1(S^1)$ can be parameterized as a $2\pi$-periodic function on $\mathbb{R}$, which is integrable over $[-\pi,\pi]$. We simply denote this by $f\in L^1\left(\left[-\pi,\pi\right]\right)$. It is $\pi$-periodic if $f$ is even on $S^1$. Furthermore, the inner product of $\bm{\theta},\bm{\eta}\in S^1$ is equivalent to the cosine of the angle between those two vectors.
Define the convolution of $2\pi$-periodic functions $u,v\in L^2([-\pi,\pi])$ by 
\begin{align*}
	(u\ast v)(t)=\int\limits_{-\pi}^\pi u(t-x)v(x)dx=\int\limits_{-\pi}^\pi u(x)v(t-x)dx,\quad t\in[-\pi,\pi].
\end{align*}

Any $2\pi$-periodic function $u\in L^2([-\pi,\pi])$ has the Fourier series expansion (in exponential form)
\begin{align*}
	u(x)=\sum\limits_{n=-\infty}^\infty\widehat{u}(n)e^{inx}
\end{align*}
with Fourier coefficients 
\begin{align*}
	\widehat{u}(n)=\frac{1}{2\pi}\int\limits_{-\pi}^\pi e^{-inx}u(x)dx.
\end{align*}
The convolution theorem of Fourier coefficients states that $\widehat{(u\ast v)}(n)=2\pi\widehat{u}(n)\widehat{v}(n)$ for all $u,v\in L^2([-\pi,\pi])$. Furthermore, the Fourier coefficient of the translation of $u$ by a constant $h\in\mathbb{R}$ is given by $\widehat{u(\cdot-h)}(n)=\widehat{u}(n)e^{-inh}$.

Then, the two-dimensional equivalent of (\ref{spherecos}) on $S^1$ is defined by 
\begin{align}
\label{spherecos2}
	K_{\alpha,S^1} f(y)=\int\limits_{-\pi}^\pi\vert\cos(y-x)\vert^\alpha f(x)dx=\left(\vert\cos(\cdot)\vert^\alpha\ast f\right)(y),\quad y\in[-\pi,\pi].
\end{align}

\begin{cor}
	\label{spherecoscor}
	Let $\alpha>-1$, $\alpha\neq0,2,4,\dots$. Let $f\in L^2([-\pi,\pi])$ be a $\pi$-periodic probability density function. Then, $f$ can be completely reconstructed from its Fourier coefficients by 
	\begin{align}
		\label{invspherecos}
		f(x)=\frac{1}{2\pi}\left(1+\sum\limits_{n\in\mathbb{Z}\setminus\{0\}}\frac{\reallywidehat{\left(K_{\alpha,S} f\right)}(2n)}{\tilde{c}_{\vert n\vert}}e^{i2nx}\right)
	\end{align}
	for all $x\in[-\pi,\pi]$, where the coefficients $\tilde{c}_k$, $k=0,1,\dots$, are given in Corollary \ref{coscorollary}.
\end{cor}
\begin{proof}
	
	Since $\left\{\tilde{c}_n\right\}$ are the coefficients of the Fourier series (in sine-cosine form) of the function $\frac{1}{2}\left\vert\cos\left(\frac{x}{2}\right)\right\vert^\alpha$, one can easily verify that $\reallywidehat{\vert\cos(\cdot)\vert^\alpha}(n)=\tilde{c}_{\vert n/2\vert}$ for even $n\in\mathbb{Z}$ and $0$ otherwise. Thus, by the convolution theorem and Equation (\ref{spherecos2}) we get the Fourier coefficients $\reallywidehat{K_{\alpha,S^1}}(n)=2\pi c_{\vert n/2\vert }\widehat{f}(n)$ for even $n\in\mathbb{Z}$ and $0$ otherwise, as well as the series expansion
	$$
		K_{\alpha,S^1}f(y)=2\pi\sum\limits_{n=-\infty}^\infty \tilde{c}_{\vert n\vert}\widehat{f}(2n)e^{i2ny}.
	$$
	
	Note that $\widehat{f}(0)=1/(2\pi)$, and for any even $n\in\mathbb{N}$ we can compute $\widehat{f}(n)=\reallywidehat{K_{\alpha,S^1}f}(n)/(2\pi c_{n/2})$. Moreover, since $f$ is $\pi$-periodic, there exists a shift $h\in[0,\pi]$ such that $f^\ast(x)=f(x+h)$ is an even, $\pi$-periodic function. It holds that 
	\begin{align*}
		\widehat{f^\ast}(n)=\frac{1}{2\pi}\int\limits_{-\pi}^\pi e^{-inx}f^\ast(x)dx=\frac{1}{\pi}\int\limits_0^\pi\cos(nx)f^\ast(x)dx=0
	\end{align*}
	since $\cos(n\cdot)$ is odd about $\pi/2$ on $[0,\pi]$ for all odd $n\in\mathbb{N}$. By the shift property we get $\widehat{f}(n)=\widehat{f^\ast}(n)e^{-inh}=0$ for $n$ odd. 
	
	To summarize the Fourier coefficients of $f$ are given by 
	\begin{align}
		\label{ffouriercoeff}
		\widehat{f}(n)
		=\begin{cases}
			1/(2\pi),&n=0,\\
			\reallywidehat{K_{\alpha,S^1}f}(n)/(2\pi\tilde{c}_{n/2}),&n=2k,~k\in\mathbb{N},\\
			0&n=2k+1,~k\in\mathbb{N}.
		\end{cases}
	\end{align}
	Lastly, note that $\widehat{f}(-n)=\overline{\widehat{f}(n)}$, where $\overline{z}$ denotes the complex conjugate for $z\in\mathbb{C}$. Applying this to the Fourier series representation $f(x)=\sum_{n=-\infty}^\infty \hat{f}(n)e^{inx}$ yields the desired result. 
\end{proof}

\begin{rem}
	In the case $\alpha=0,2,4,\dots$, the coefficients $\tilde{c}_j=0$ for all $j\geq\alpha/2$ (see Corollary \ref{coscorollary}). 
	Thus, only a finite number of Fourier coefficients of $f$ can be computed from the Fourier coefficients $\reallywidehat{K_{\alpha,S} f}(n)$, making it impossible to fully reconstruct $f$.
\end{rem}

\section{Numerical results}
\label{num}

We first consider the $\alpha$-sine transform on $\mathbb{R}_+$. The inversion results of the direct approach are presented in Section \ref{numDA} for the case $\alpha>1$. Additionally, various numerical difficulties that emerge when dealing with the direct approach are analyzed. 
The Fourier approximation approach is considered in Section \ref{numFAA}. It proves to be more accurate as well as more efficient compared to the direct approach. Most importantly, it is applicable for all $\alpha>-1$. Inversion results are first shown for $\alpha\geq 0$, the case $-1<\alpha<0$ is considered separately. Moreover, in the context of an application to harmonizable $S\alpha S$ processes, Gaussian noise is added to the transform $T_\alpha f$ to test the smoothing procedure of Section \ref{mollifier}.
Lastly, Section \ref{numspherecos} deals with the two dimensional $\alpha$-cosine transform. We consider a number of $\pi$-periodic probability density functions on $[-\pi,\pi]$ and demonstrate inversion formula (\ref{invspherecos}) developed in Corollary \ref{spherecoscor} for $\alpha>-1$.

Consider the following functions in $L^1_{w,\alpha}(\mathbb{R}_+)\cap L^2(\mathbb{R}_+)$ as well as their $2$-sine transform $T_2f$:
\begin{ex} 
	\label{ex}
	\begin{enumerate}[(a)]
		\item $
			f_1(x)=e^{-x^2},\quad T_2f_1(y)=\frac{\sqrt{\pi}}{4}\left(1-e^{-y^2}\right).
		$
		\item $
			f_2(x)=x^2e^{-\vert x\vert},\quad T_2f_2(y)=\frac{8y^2\left(3+6y^2+8y^4\right)}{\left(1+4y^2\right)^3}.
		$
		\item $
			f_3(x)=\frac{1}{(1+x^2)^2},\quad T_2f_3(y)=\frac{\pi}{8}\left(1-e^{-2\vert y\vert}(1+2\vert y\vert)\right).
		$
	\end{enumerate}

\end{ex}
For $\alpha\neq 0,2,4,\dots$, the transform $T_\alpha f$ is in general not explicitly given, though its graph remains fairly similar, see Figure \ref{fig:Tf_plots}.
\begin{figure}[h]
	\centering
	\begin{subfigure}{0.32\textwidth}
		\includegraphics[width=\textwidth]{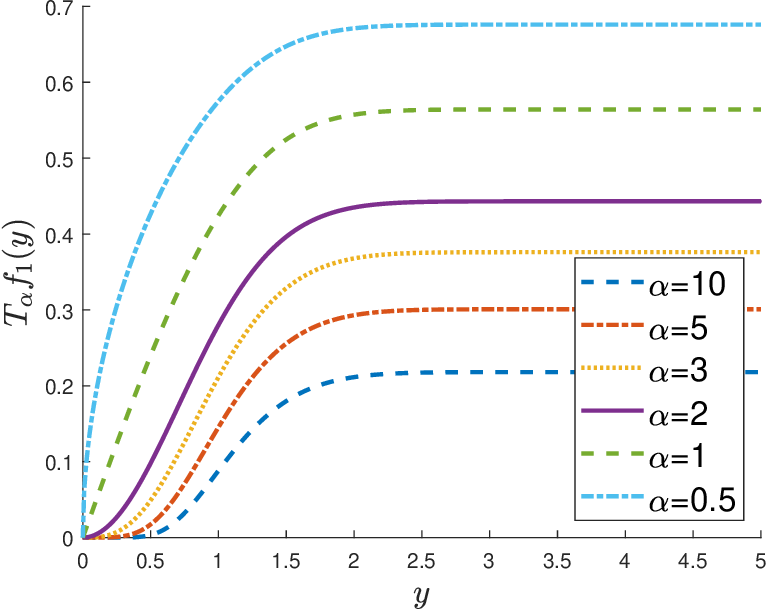}
		\caption{}
	\end{subfigure}
	\begin{subfigure}{0.32\textwidth}
		\includegraphics[width=\textwidth]{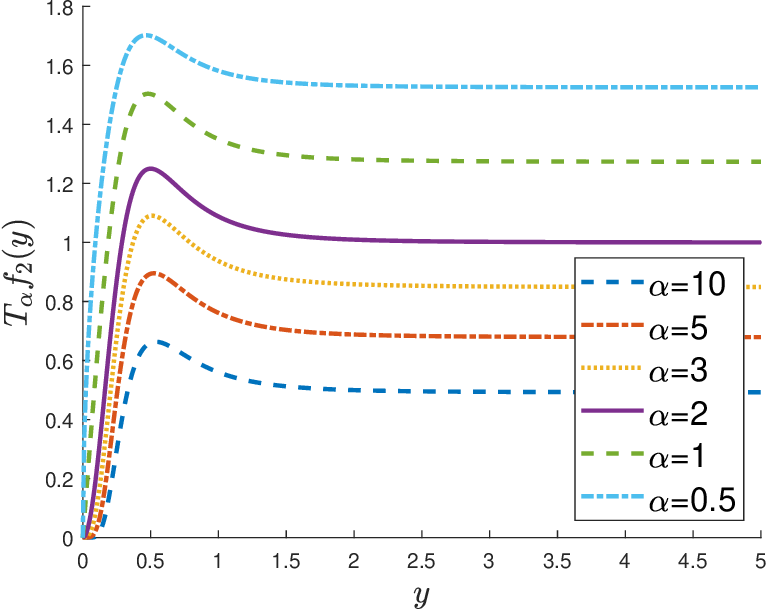}
		\caption{}
	\end{subfigure}
	\begin{subfigure}{0.32\textwidth}
		\includegraphics[width=\textwidth]{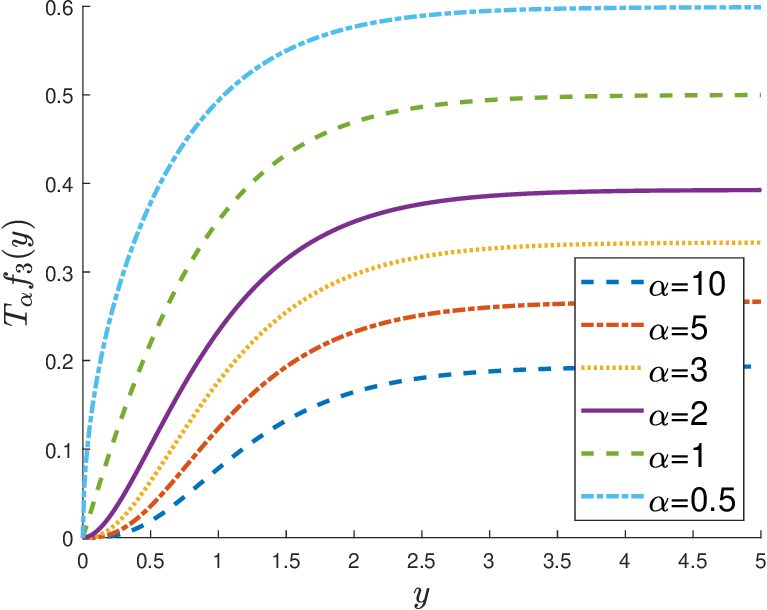}
		\caption{}
	\end{subfigure}
	\caption{\textbf{$\alpha$-sine transform on $\mathbb{R}$.} Plots of the function $Tf$ for all three examples $f_1, f_2$ and $f_3$ with $\alpha\in\{0.5,1,2,3,5,10\}$.}
	\label{fig:Tf_plots}
\end{figure}

\subsection{Direct approach}
\label{numDA}

In application, one has to keep the following two issues with the direct approach in mind. The operator $\mathcal{G}^{-1}$ defined in Corollary \ref{final_cor} only fulfills injectivity, but it is not surjective. Secondly, $\mathcal{G}^{-1}$ is in general not a continuous operator. Therefore a straightforward application of $\mathcal{G}^{-1}$ does not automatically yield $f$.
The reason for the above-mentioned arises from the fact that $\inf_{x>0}\mu(x)=0$, where $\mu$ was defined 
in Equation (\ref{mu+2}). Hence, the function $1/\mu$ is not bounded and multiplication by this function does not define a linear bounded operator on $L^2\left(\mathbb{R}_+,\frac{dx}{x}\right)$. To deal with this we approximate $1/\mu$ by
\begin{align*}
\frac{1}{\mu}\mathbbm{1}_{\{\vert\mu\vert>\varepsilon\}}
\end{align*}
for some small $\varepsilon>0$, and estimate the function $f$ by 
\begin{align}
\label{finvDA}
\tilde{f}(x)=\left(\mathcal{M}^{-1}\mathcal{F}_+^{-1}\frac{1}{\mu}\mathbbm{1}_{\{\vert\mu\vert>\varepsilon\}}\mathcal{F_+M}\left[z^{-1}g\left(z^{-1}\right)\right]\right)(x)~.
\end{align}
To simplify the numerical implementation, note that for any function $u\in L^2\left(\mathbb{R}_+,x^cdx\right)$
\begin{align*}
\mathcal{F_+M}u(x)=\int\limits_{\mathbb{R}_+}\mathcal{M}u(y)e^{-i\log(x)\log(y)}\frac{dy}{y}=\int\limits_{\mathbb{R}_+}y^{\frac{c+1}{2}}u(y)e^{-i\log(x)\log(y)}\frac{dy}{y}=\int\limits_{\mathbb{R}_+}y^{\frac{c-1}{2}-i\log(x)}u(y)dy~.
\end{align*}
Thus, we define $\mathcal{H}:L^2\left(\mathbb{R}_+,x^cdx\right)\rightarrow L^2\left(\mathbb{R}_+,\frac{dx}{x}\right)$ by
\begin{align*}
\mathcal{H}g(x)=\mathcal{F_+M}\left[y^{-1}g\left(y^{-1}\right)\right](x)=\int\limits_{\mathbb{R}_+}y^{\frac{c-1}{2}-i\log(x)}y^{-1}g\left(y^{-1}\right)dy=\int\limits_{\mathbb{R}_+}y^{\frac{c-3}{2}-i\log(x)}g\left(y^{-1}\right)dy~.
\end{align*}
Similarly, we introduce $\mathcal{H}_2: L^2\left(\mathbb{R}_+,x^cdx\right)\rightarrow,  L^2\left(\mathbb{R}_+,\frac{dx}{x}\right)$ by
\begin{align*}
\mathcal{H}_2w(z)=\mathcal{M}^{-1}\mathcal{F_+}^{-1}w(z)=z^{-\frac{c+1}{2}}\mathcal{F_+}^{-1}w(z)=z^{-\frac{c+1}{2}}\frac{1}{2\pi}\int\limits_{\mathbb{R}_+}w(x)e^{i\log(x)\log(z)}\frac{dx}{x}=\frac{z^{-\frac{c+1}{2}}}{2\pi}\int\limits_{\mathbb{R}_+}w(x)x^{i\log(z)-1}dx~.
\end{align*}
Then $\tilde{f}$ is computed by 
\begin{align}
\label{fdirectapproach}
\tilde{f}(x)=\left(\mathcal{H}_2\frac{1}{\mu}\mathbbm{1}_{\{\vert\mu\vert>\varepsilon\}}\mathcal{H}g\right)(x)~.
\end{align}
Recall that the constant $c=1+\delta$, where $\delta>0$ such that $\alpha\geq1+\delta/2$, thus 
$1<c\leq2\alpha-1$. Furthermore, the constant $c$ directly affects the decay of the integrand in the definition of $\mu$, i.e. the larger $c$ the faster its decay. We therefore set $c$ as large as possible, that is $c=2\alpha-1$, in the case $0<\alpha<2$. In the case $\alpha>2$, values of $c>3$, and thus the exponent $(c-3)/2>0$ in the operator $\mathcal{H}$ above lead to difficulties during numerical integration.
It proves to be more convenient to set $c=3\in(1,2\alpha-1]$ here.

The functions $f_1$ and $f_2$ of Example \ref{ex} are contained in the space $L^2(\mathbb{R}_+,x^cdx)$ for all $\alpha>1$, hence the inversion formula from Corollary \ref{final_cor} is applicable for all $\alpha>1$. On the other hand, function $f_3$ fulfills this condition only for $\alpha\in(1,2]$.
The results for all three example functions in the case $\alpha=2$ are displayed in Figure \ref{fig:fa2DA}. The inversion was performed with $\varepsilon=0.025$.
\begin{figure}[h]
	\centering
	\begin{subfigure}{0.32\textwidth}
		\includegraphics[width=\textwidth]{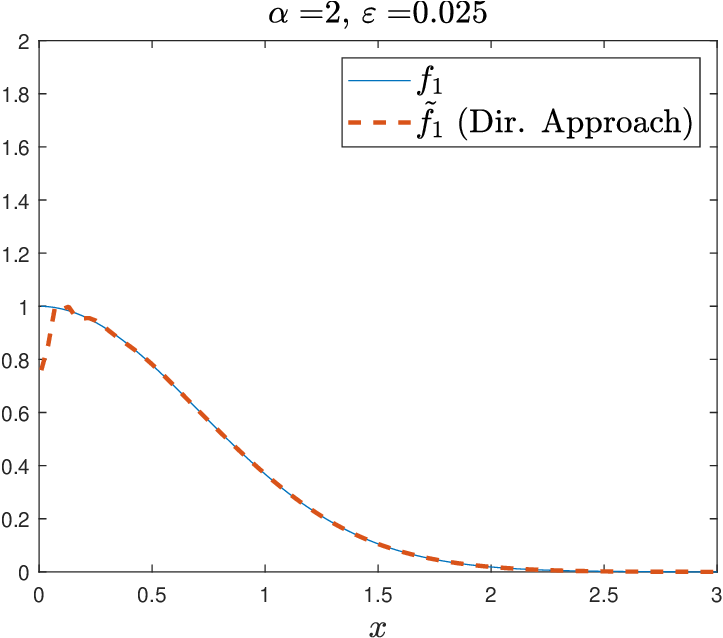}
		\caption{}
	\end{subfigure}
	\begin{subfigure}{0.32\textwidth}
		\includegraphics[width=\textwidth]{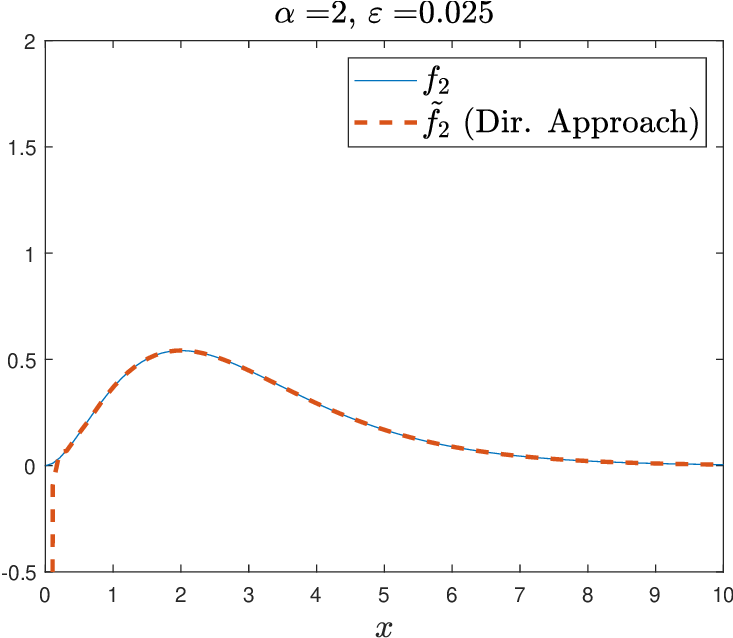}
		\caption{}
	\end{subfigure}
	\begin{subfigure}{0.32\textwidth}
		\includegraphics[width=\textwidth]{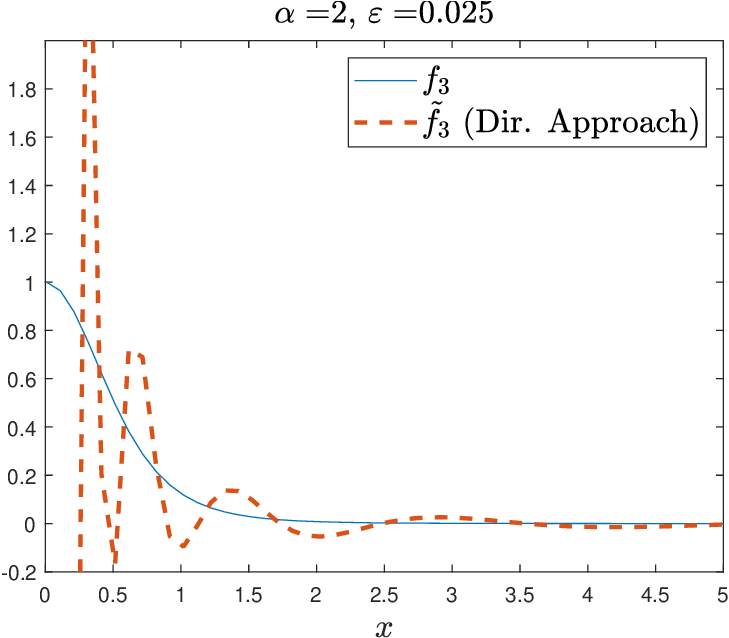}
		\caption{}
	\end{subfigure}
	\caption{\textbf{Direct approach, $\alpha=2$}. Inversion performed with $\varepsilon=0.025$. Solid blue line shows the actual functions $f_1,f_2,f_3$. The dashed red line shows the result of the inversion of the direct approach.}
	\label{fig:fa2DA}
\end{figure}

The choice of $\varepsilon$ is crucial for the precision of the inversion. The closer $\varepsilon$ is to $0$, the better the results of the inversion, but computation takes a significantly larger amount of time. 
Allowing for a larger $\varepsilon$ reduces computation time, but large deviations from the expected result can be seen in Figure \ref{fig:fa2DAbad}, where the inversion was performed for all three example functions with $\varepsilon=0.1$. Moreover, the inversion for the functions $f_1$ and $f_2$ is significantly faster than in the case of $f_3$. 

\begin{figure}[h]
	\centering
	\begin{subfigure}{0.32\textwidth}
		\includegraphics[width=\textwidth]{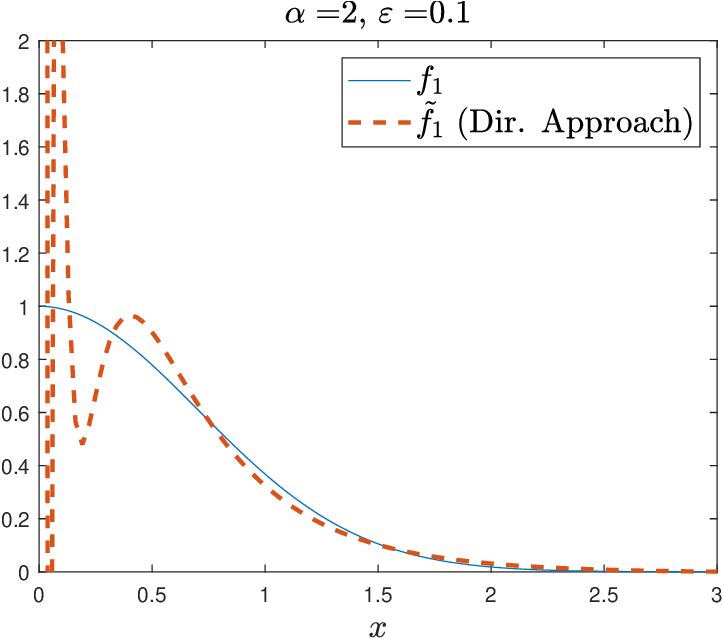}
		\caption{}
	\end{subfigure}
	\begin{subfigure}{0.32\textwidth}
		\includegraphics[width=\textwidth]{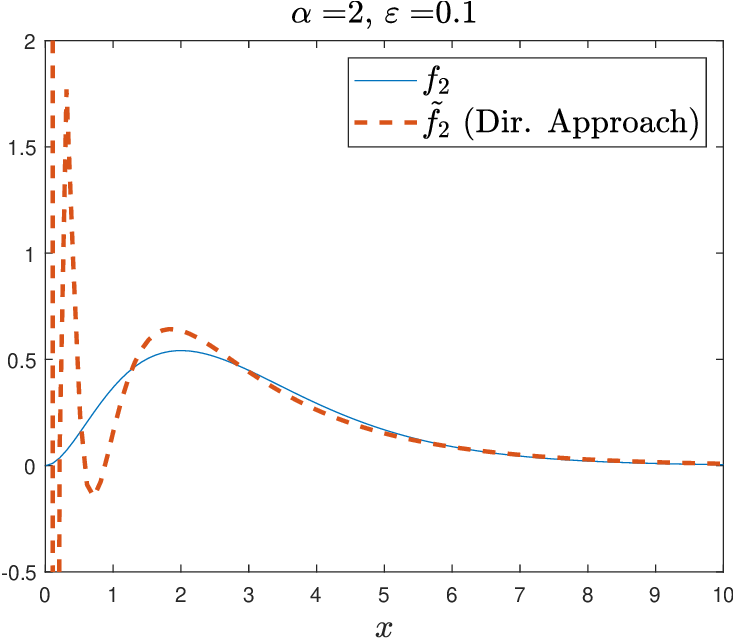}
		\caption{}
	\end{subfigure}
	\begin{subfigure}{0.32\textwidth}
		\includegraphics[width=\textwidth]{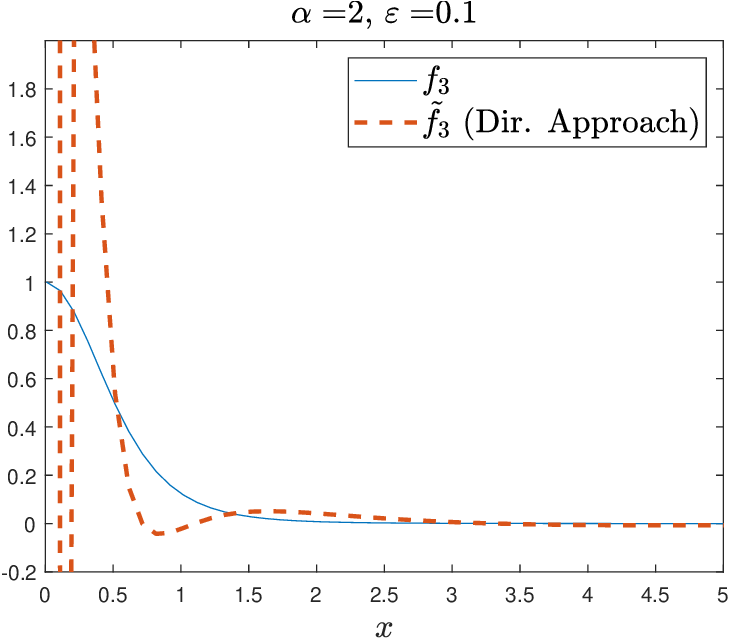}
		\caption{}
	\end{subfigure}
	\caption{\textbf{Direct approach, $\alpha=2$}. Inversion performed with $\varepsilon=0.1$ for all three examples. Solid blue line shows the actual functions $f_1,f_2,f_3$. The dashed red line shows the result of the inversion of the direct approach.}
	\label{fig:fa2DAbad}
\end{figure}

In the following, only example $f_2$ is considered. The inversion is computed for the cases $\alpha\in\{10,3,1.5\}$. The results are illustrated in Figure \ref{fig:f2DA}. Numerical computation takes a significantly larger amount of time for smaller $\alpha$, but at the same time a larger choice $\varepsilon$ suffices.
\begin{figure}
	\centering
	\begin{subfigure}{0.32\textwidth}
		\includegraphics[width=\textwidth]{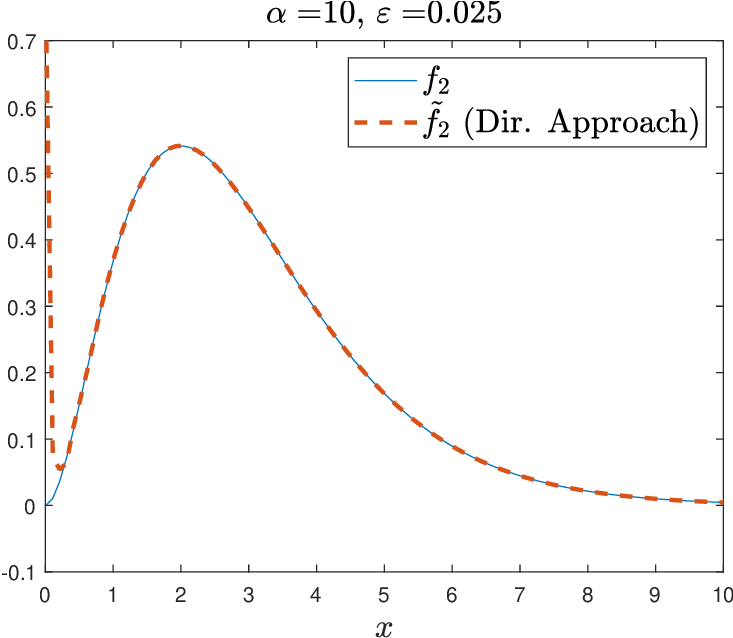}
	\end{subfigure}
	\begin{subfigure}{0.32\textwidth}
		\includegraphics[width=\textwidth]{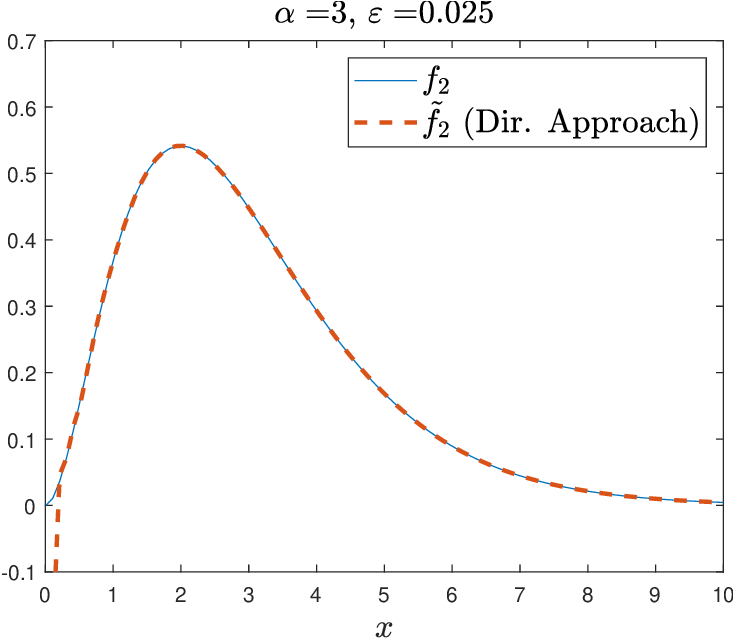}
	\end{subfigure}
	\begin{subfigure}{0.32\textwidth}
		\includegraphics[width=\textwidth]{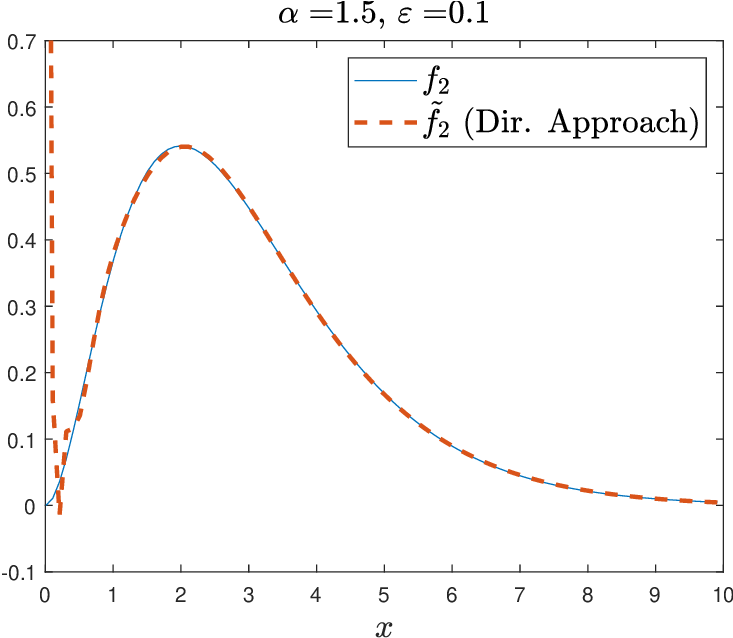}
	\end{subfigure}
	\caption{\textbf{Direct approach, $\alpha\in\{10,3, 1.5\}$}. Inversion performed for $f_2$ . Solid blue line shows the actual function $f_2$. The dashed red line shows the result of the inversion of the direct approach.}
	\label{fig:f2DA}
\end{figure}

The inversion hinges on the behavior of the function $\mu$ which directly depends on $\alpha$. Figure \ref{fig:muplots} depicts the function $\vert\mu\vert$ on the left-hand side as well as the real part of  $1/\mu(x)\mathbbm{1}_{\{\vert\mu(x)\vert>0.1\}}$ on the right-hand side for various values of $\alpha$ with $c=2\alpha-1$ (or $c=3$ for $\alpha\geq2$). The larger $\alpha$ the faster and less fluctuating the decay of $\vert\mu\vert$. Thus, the function $1/\mu(x)\mathbbm{1}_{\{\vert\mu(x)\vert>0.1\}}$ cuts off sooner to 0 which directly contributes to computation of the integral operator $\mathcal{F}_+^{-1}$ in Equation (\ref{finvDA}). Additionally, Figure \ref{fig:muplots} emphasizes how choosing $\varepsilon$ too small when dealing with smaller $\alpha$ is counterproductive, for example consider the case $\alpha=1.25$. The corresponding dotted line on the left-hand side of Figure \ref{fig:muplots} is fluctuating, and $\vert\mu\vert$ fails to fall below $\varepsilon=0.1$. Thus, $1/\mu(x)\mathbbm{1}_{\{\vert\mu(x)\vert>0.1\}}$ on the right-hand side of the plot is also fluctuating a lot and does not cut off to 0 as in the other cases. This will lead to difficulties during numerical integration and prolonged computation times. 

\begin{figure}
	\centering
	\begin{subfigure}{0.32\textwidth}
		\includegraphics[width=\textwidth]{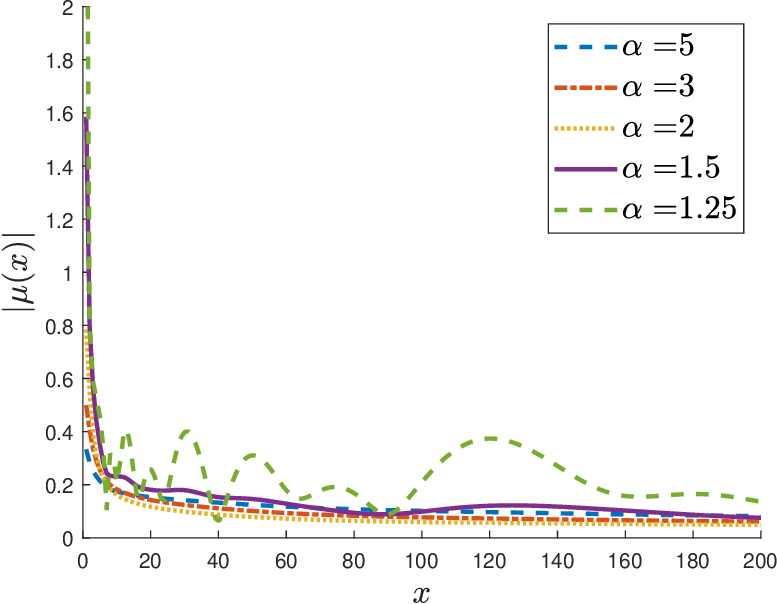}
	\end{subfigure}
	\hspace{3ex}
	\begin{subfigure}{0.32\textwidth}
		\includegraphics[width=\textwidth]{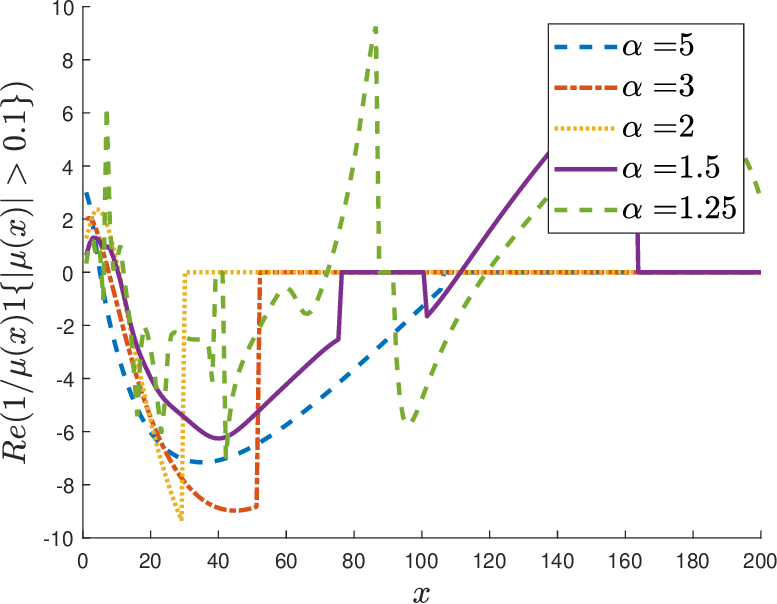}
	\end{subfigure}
	\caption{
		The functions $\vert\mu(x)\vert$ (left) and $1/\mu(x)\mathbbm{1}_{\{\vert\mu(x)\vert>0.1\}}$ (right) for different values of $\alpha$ with $c=2\alpha-1$ ($c=3$ for $\alpha\geq 2$).}
	\label{fig:muplots}
\end{figure}

On another note, for $\alpha\neq 2$ the function $g=T_\alpha f$ is not given in a closed formula which makes applications of integral transforms to it numerically challenging, as the inversion formula (\ref{finvDA}) essentially involves threefold integration, i.e. by $T_\alpha $ itself and by $\mathcal{F}_+, \mathcal{F}_+^{-1}$, respectively. Therefore, the function $g$ was sampled discretely on the interval $[0,20]$ with a step size of $10^{-6}$ between each sample point. Linear interpolation and constant extrapolation with the last sample point was used as an estimate for $g$. The upper bound of the interval was chosen such that $g$ is approximately constant beyond that point.

\subsection{Fourier approximation approach}
\label{numFAA}

We consider the functions given in Example \ref{ex}. 
Function values of the Fourier transform of the wanted function $f$ are the solution of the system of linear equations $\bm{\eta}=\bm{C\cdot\xi}$ described in Section \ref{Fourierapproximation}. Recall, the vectors $\bm{\xi}\in\mathrm{R}^N$ containing function values of the Fourier transform,
$
\xi_n=\hat{f}_R\left(n\frac{R}{N}\right),
$
and $\bm{\eta}\in\mathrm{R}^N$ given by 
$
\eta_n=T_\alpha f\left(n\frac{R}{2N}\right)-\frac{c_0}{2}\mathcal{F}f(0)
$
for $n=1,\dots,N$, as well as the matrix $\bm{C}$ defined in Proposition \ref{propmatrix}. 

The integer $R$ should be chosen as large as possible. Under the assumption that the Fourier transform $\mathcal{F}f$ is negligibly outside of the interval $[-R,R]$, we choose $R$ large enough such that 
$T_\alpha f(y)$ is approximately constant for all $y>R$. 
Furthermore, note that $T_\alpha f(y)\approx\frac{c_0}{2}\mathcal{F}f(0)$ for large $y>R$. Hence, we chose $\mathcal{F}f(0)\approx2T_\alpha f(y)/c_0$ for some large $y>R$. 
Lastly, the larger $N$ the better in general, as this will make for finer sampling. 

Figure \ref{fig:Ffsyssol} shows the solution $\bm{\xi}$ of the system of linear equations for all three examples and the case $\alpha=1.5$ with $N=100$ and $R=10$.
\begin{figure}[h]
	\centering
	\begin{subfigure}{0.32\textwidth}
		\includegraphics[width=\textwidth]{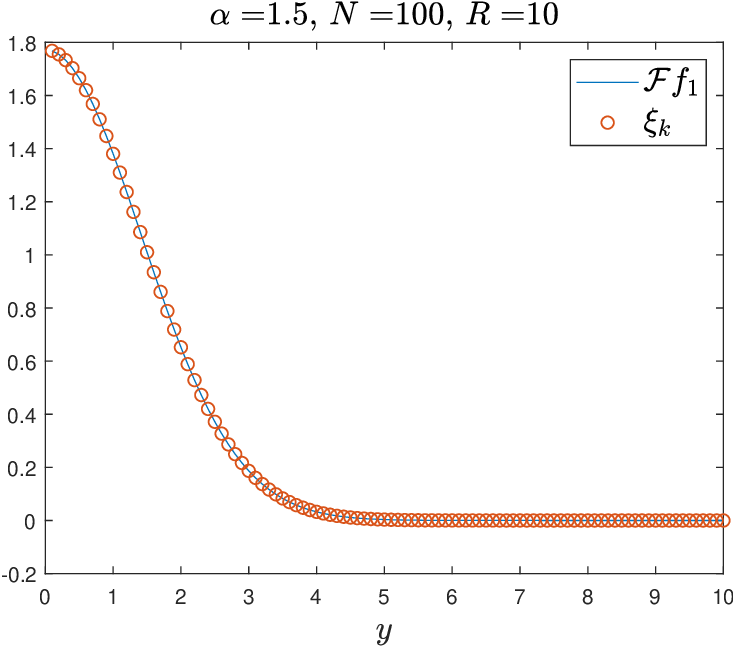}
		\caption{}
	\end{subfigure}
	\begin{subfigure}{0.32\textwidth}
		\includegraphics[width=\textwidth]{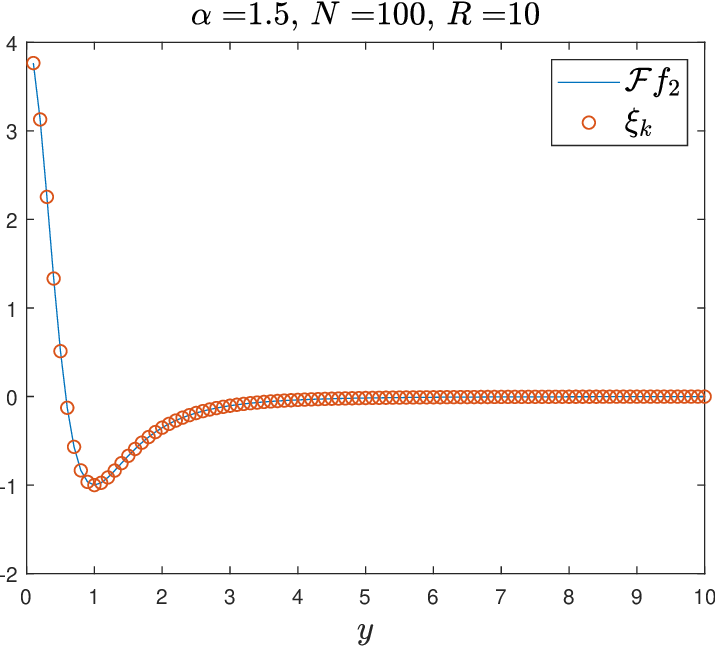}
		\caption{}
	\end{subfigure}
	\begin{subfigure}{0.32\textwidth}
		\includegraphics[width=\textwidth]{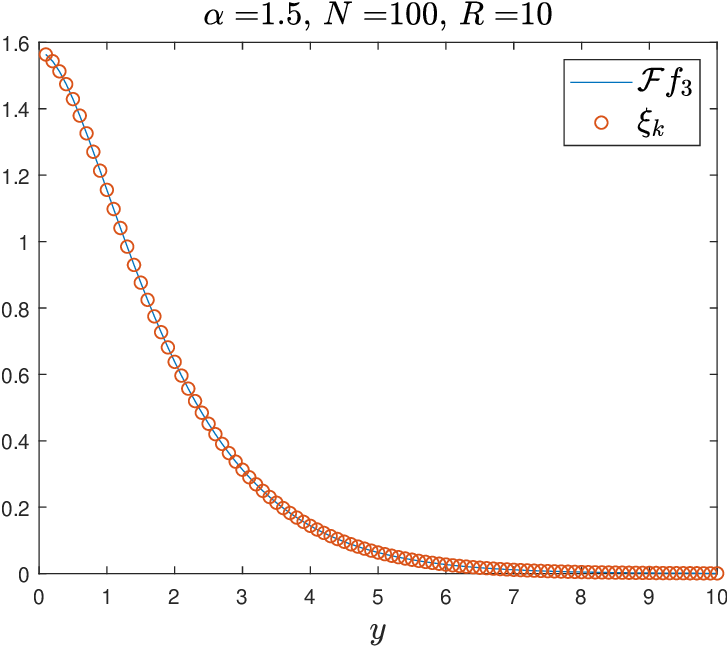}
		\caption{}
	\end{subfigure}
	\caption{\textbf{System of linear equations, $\alpha=1.5, N=100, R=10$.} Solid blue line shows the Fourier transform of the examples $f_1, f_2$ and $f_3$.  The red circles are the solution of the system of linear equations $\bm{\xi}$.}
	\label{fig:Ffsyssol}
\end{figure}
\begin{figure}[h]
	\centering
	\begin{subfigure}{0.32\textwidth}
		\includegraphics[width=\textwidth]{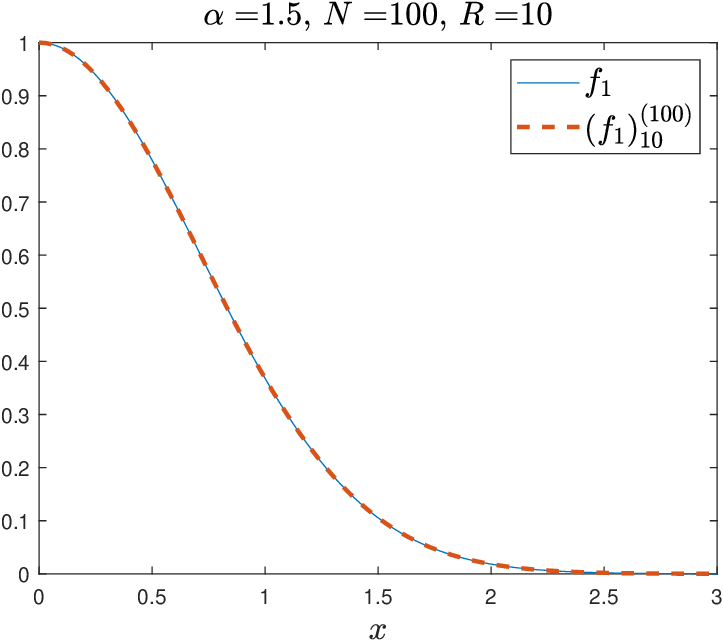}
		\caption{}
	\end{subfigure}
	\begin{subfigure}{0.32\textwidth}
		\includegraphics[width=\textwidth]{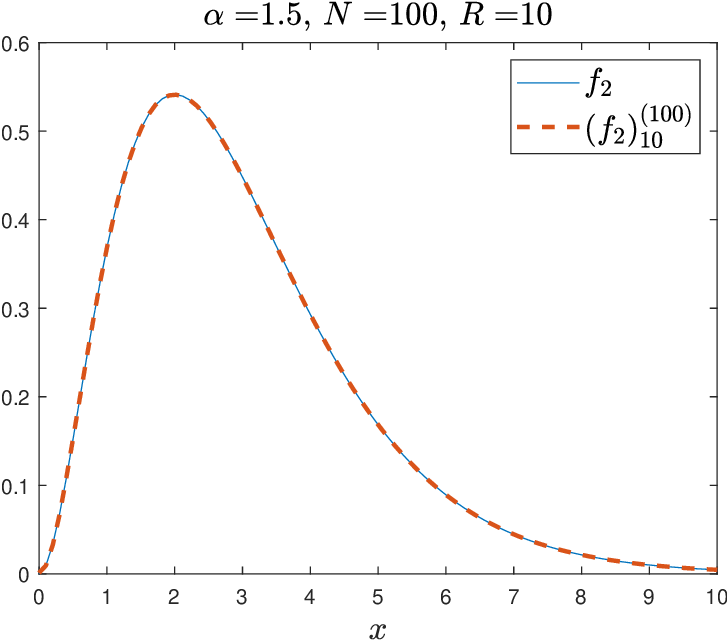}
		\caption{}
	\end{subfigure}
	\begin{subfigure}{0.32\textwidth}
		\includegraphics[width=\textwidth]{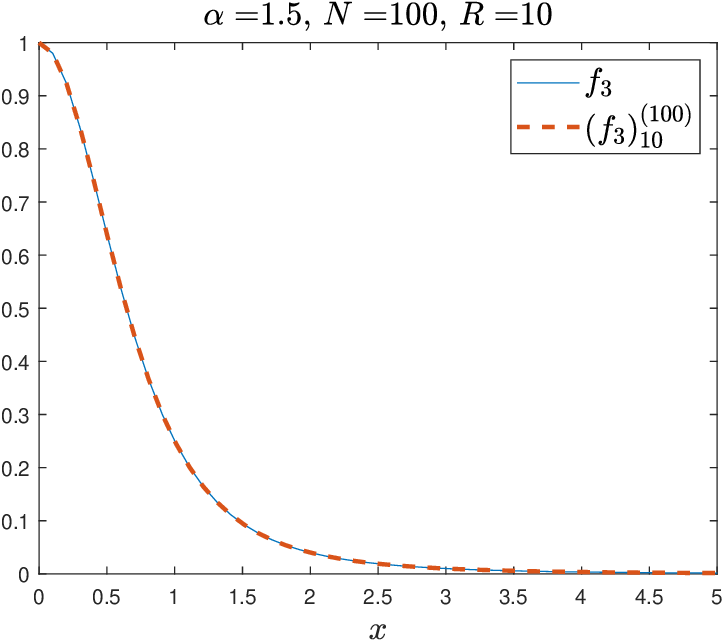}
		\caption{}
	\end{subfigure}
	\caption{\textbf{Band-limited interpolation, Inversion, $\alpha=1.5, N=100, R=10$.} Solid blue line shows the functions $f_1, f_2$ and $f_3$.  The dashed red line shows the inversion results.}
	\label{fig:fa15BLI}
\end{figure}

On the other hand, when fixing $\alpha$ and $N$, choosing $R$ too small is directly reflected in the results for $\bm{\xi}$, which is to be expected.
For example choosing $R=1$ would imply the assumption that $\mathcal{F}f$ vanishes outside of the interval $[-1,1]$, which cannot be true as the function $T_\alpha f(y)$ is still increasing for $y>1$ in all three examples $f_1,f_2,f_3$, see Figure \ref{fig:Tf_plots} (a).  

Contrary to that, larger $R$ also show good results for $\bm{\xi}$. But since the number of sample points is fixed with $N$, the larger $R$ the coarser the grid on which the Fourier transform $\mathcal{F}f$ is determined, which will effect the precision of the interpolation to follow. 

The sampling theory of Shannon as discussed before gives another useful interpolation method for the Fourier transform of the spectral density. By symmetry of the Fourier transform and choosing $\hat{f}(0)=\mathcal{F}f(0)\approx2T_\alpha f(y)/c_0$ for some large $y>R$, the function values $\hat{f}_R\left(kR/N\right)$, $k=-N,\dots,N$ are approximated by the solution of the system of equations $\bm{\xi}$. Using the truncated cardinal series defined in (\ref{fhattrunccardinal}) we interpolate $\hat{f}_R$. 

By comparison of $\hat{f}(n/(2B))=\hat{f}(nR/N)$, the Nyquist rate $2B$ corresponds to the ratio $N/R$. Since we generally cannot assume $f$ to be compactly supported, the convergence results previously stated, suggest that the ratio $N/R$ should tend to infinity. In practice, we therefore choose $R$ just like in the case of the linear interpolation large enough such that $T_\alpha f(y)$ is approximately constant for $y>R$. The integer $N$ is then set as large as possible. 
We compute the estimate for the spectral density $f$ by
\begin{align*}
f_R^{(N)}(x)=
\frac{R}{2\pi N}rect\left(\frac{xR}{2\pi N}\right)\sum\limits_{n=-N}^N\hat{f}\left(n\frac{R}{N}\right)e^{-ixnR/N},
\end{align*}
see Theorem \ref{lastcor}.
Figure \ref{fig:fa15BLI} shows the result of the inversion with $N=100$, $R=10$ (based on the results depicted in Figure \ref{fig:Ffsyssol}) for all three examples in the case $\alpha=1.5$.

We performed the inversion with various values of $\alpha>0$ for all three functions $f_1,f_2$ and $f_3$. 
With a suitable choice of $N, R$, different values of $\alpha$ seem to have no effect on the computation of $\bm{\xi}$. 
Figure \ref{fig:Ffdiffalpha} highlights the result for $f_1$ with $N=100, R=10$ and $\alpha=10$. 

 \begin{figure}
 	\centering
 	\begin{subfigure}{0.32\textwidth}
 		\includegraphics[width=\textwidth]{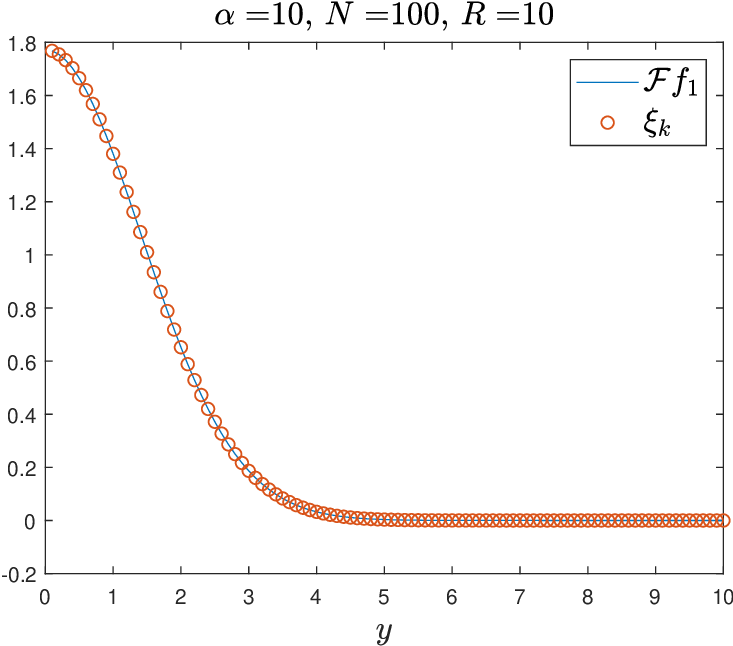}
 	\end{subfigure}
 	\begin{subfigure}{0.32\textwidth}
 		\includegraphics[width=\textwidth]{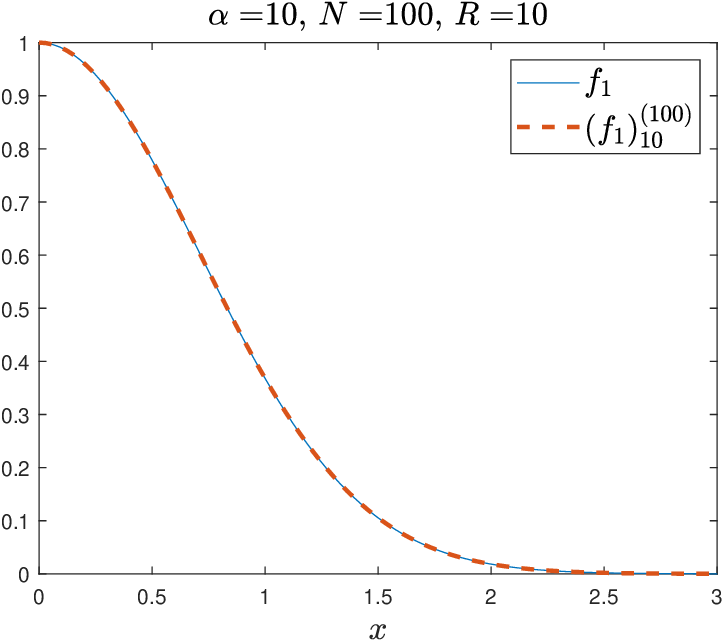}
 	\end{subfigure}
 	\caption{\textbf{Band-limited interpolation, Inversion, $\alpha=10, N=100, R=10$.} Approximation of the Fourier transform $\mathcal{F}f_1$ (left), and the inversion results (right).}
 	\label{fig:Ffdiffalpha}
 \end{figure}

\begin{rem}
	In Remark \ref{linintremark} the possibility of linear interpolation was mentioned. Our numerical experiments show that there seemed to be no visible difference in the results of either interpolation method. Linear interpolations performs slightly faster but is on the other hand also marginally less accurate. See Table \ref{tab:comp} for a comparison between all methods for example $f_2$.
\end{rem}

\subsubsection{Smoothing}
\label{noisy}
Section \ref{mollifier} touched upon the topic of noise contamination of the function $T_\alpha f$. To simulate this scenario, we artificially added noise to the function $T_\alpha f$. 
For all three functions $f_1,f_2$ and $f_3$, their transform $T_\alpha f$ is sampled on the interval $[0,20]$ at $400$ equidistant points. We then added Gaussian noise with a standard deviation of $0.1$ to each sample point, see Figure \ref{fig:Tfnoise}.
\begin{figure}[h]
	\centering
	\begin{subfigure}{0.32\textwidth}
		\includegraphics[width=\textwidth]{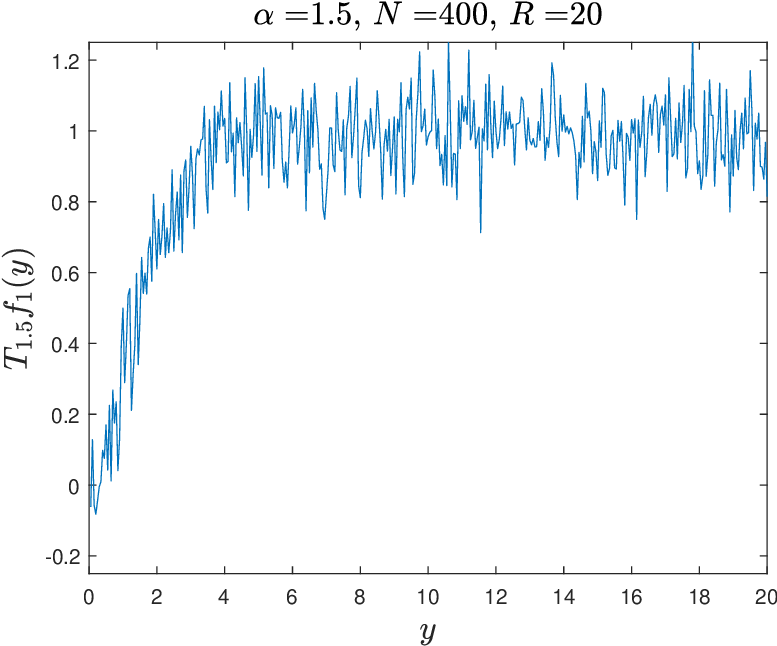}
		\caption{}
	\end{subfigure}
	\begin{subfigure}{0.32\textwidth}
		\includegraphics[width=\textwidth]{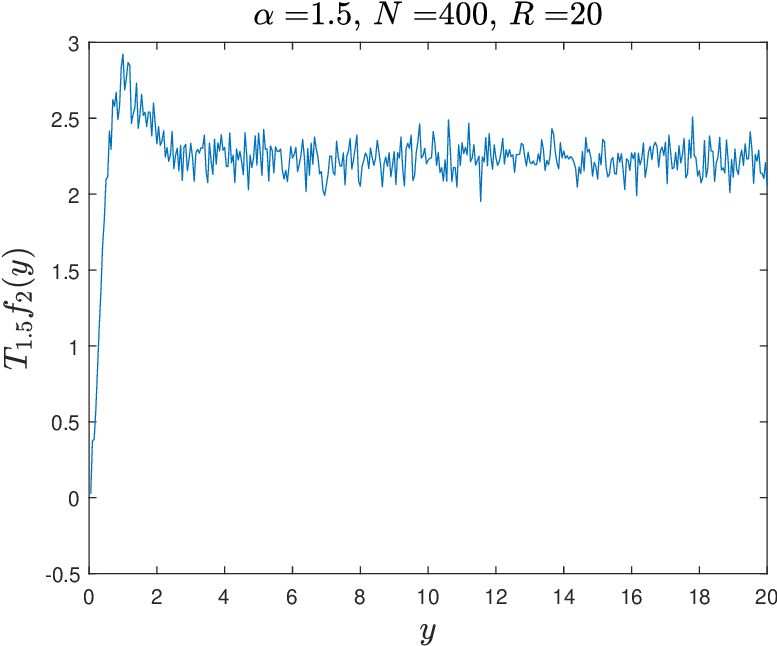}
		\caption{}
	\end{subfigure}
	\begin{subfigure}{0.32\textwidth}
		\includegraphics[width=\textwidth]{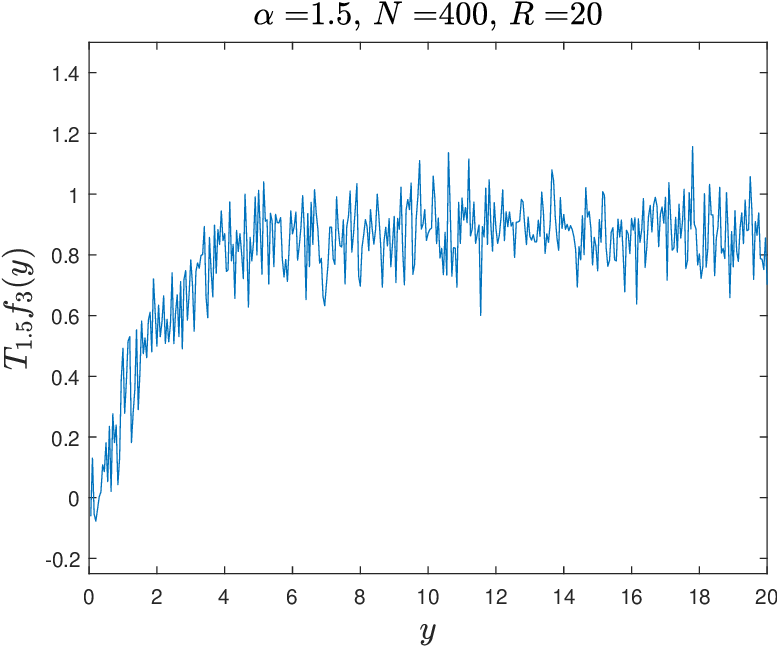}
		\caption{}
	\end{subfigure}
	\caption{\textbf{Noisy data.} Gaussian noise added to $T_{1.5} f$, sampled on the interval $[0,20]$ at $400$ equidistant points, for all three spectral densities $f_1$, $f_2$ and $f_3$.}
	\label{fig:Tfnoise}
\end{figure}

Recall from Section \ref{mollifier}, that we compute the smoothed solution of the function $f$ by $f^{(N)}_{R,\gamma}=\mathcal{F}^{-1}\left(f^{(N)}_R\cdot\psi_\gamma\right)$, where $\psi_\gamma$ is the reconstruction kernel with smoothing parameter $\gamma>0$ satisfying $\psi_\gamma=\mathcal{F}e_\gamma$ for a given mollifier function $e_\gamma$, and $f^{(N)}_{R,\gamma}$ is the approximation of the Fourier transform $\mathcal{F}f$ as computed in the previous Sections. 

There are many options for mollifiers available. We consider the following two examples of mollifiers $e^{(1)}, e^{(2)}$ and their corresponding reconstruction kernels $\psi^{(1)}, \psi^{(2)}$.
\begin{ex}\label{ex2}~\\[-4ex]
	\begin{enumerate}[(i)]
		\item $
		e^{(1)}(x)=
		\begin{cases}
		1-\vert x\vert&,-1\leq x\leq 1,\\
		0&, \text{ else},
		\end{cases}
		\qquad\text{with}\qquad
		\psi^{(1)}(y)=2\left(1-\cos(y)\right)/y^2,~y\in\mathbb{R}.
		$
		\item $
		e^{(2)}(x)=e^{-\pi x^2},~x\in\mathbb{R},
		\qquad\text{with}\qquad\psi^{(2)}(y)=e^{-y^2/(4\pi)},~y\in\mathbb{R}.
		$
	\end{enumerate}
\end{ex}

The solution $\bm{\xi}$ of the linear equation $\bm\eta=\bm{C\xi}$ is interpolated linearly as well as with the truncated cardinal series, see Figure \ref{fig:Ffa15noise}. Whichever interpolation method is employed before does not play a significant role. 
\begin{figure}[h]
	\centering
	\begin{subfigure}{0.32\textwidth}
		\includegraphics[width=\textwidth]{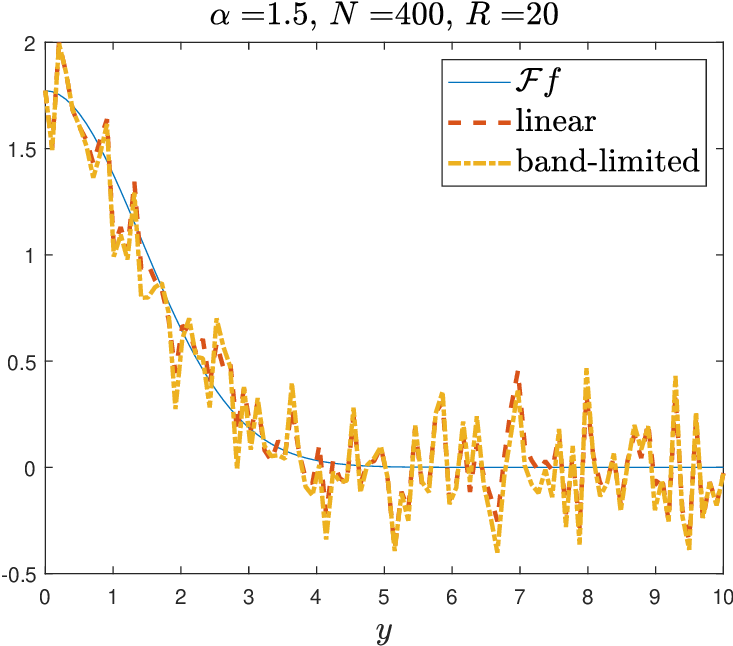}
		\caption{}
	\end{subfigure}
	\begin{subfigure}{0.32\textwidth}
		\includegraphics[width=\textwidth]{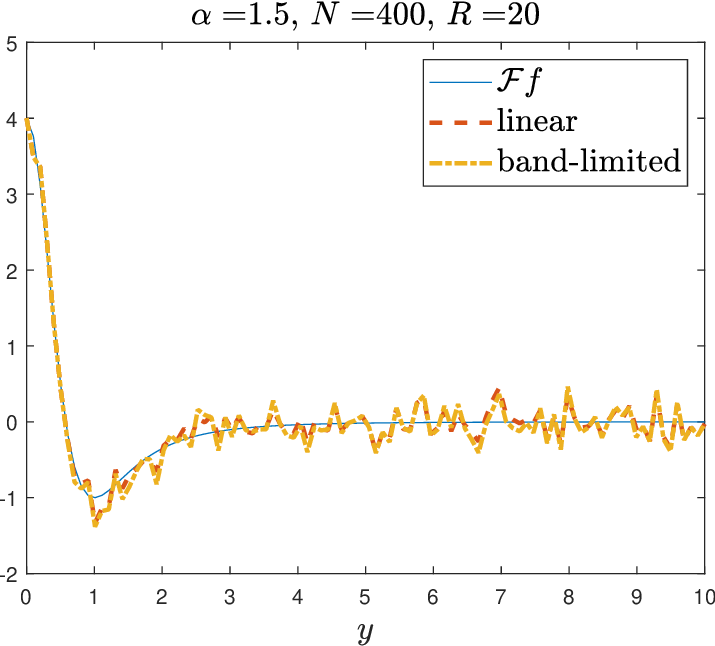}
		\caption{}
	\end{subfigure}
	\begin{subfigure}{0.32\textwidth}
		\includegraphics[width=\textwidth]{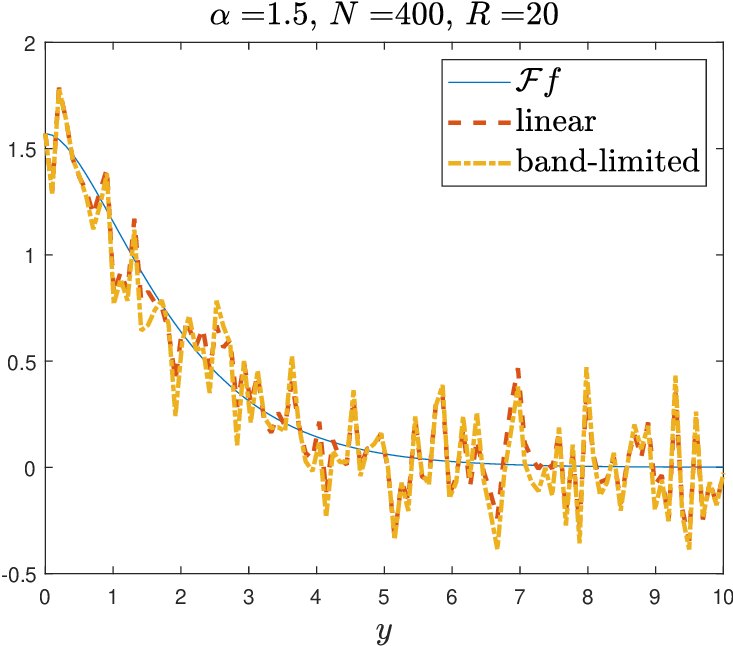}
		\caption{}
	\end{subfigure}
	\caption{\textbf{Noisy data.} Linear interpolation and band-limited interpolation in the case $\alpha=1.5$ for all three example functions.}
	\label{fig:Ffa15noise}
\end{figure}
The smoothed solution is computed for all three examples with $\gamma=0.5$. Smoothing was performed with reconstruction kernel $\psi^{(1)}$. The results are displayed in Figure \ref{fig:fa15moll_LI}.
\begin{figure}
	\centering
	\begin{subfigure}{0.32\textwidth}
		\includegraphics[width=\textwidth]{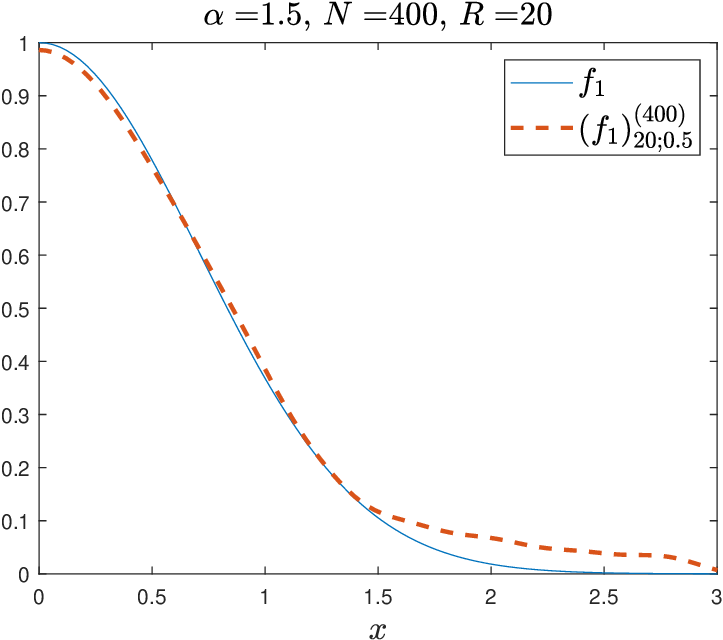}
		\caption{}
	\end{subfigure}
	\begin{subfigure}{0.32\textwidth}
		\includegraphics[width=\textwidth]{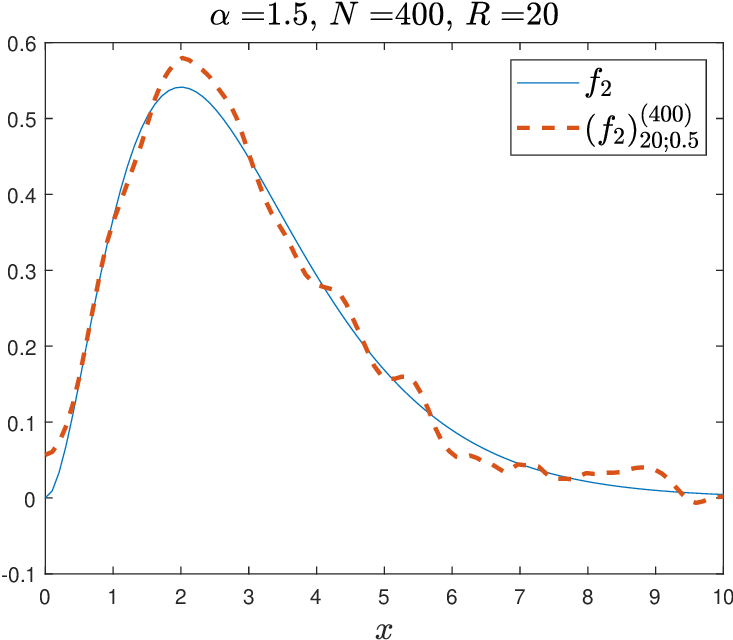}
		\caption{}
	\end{subfigure}
	\begin{subfigure}{0.32\textwidth}
		\includegraphics[width=\textwidth]{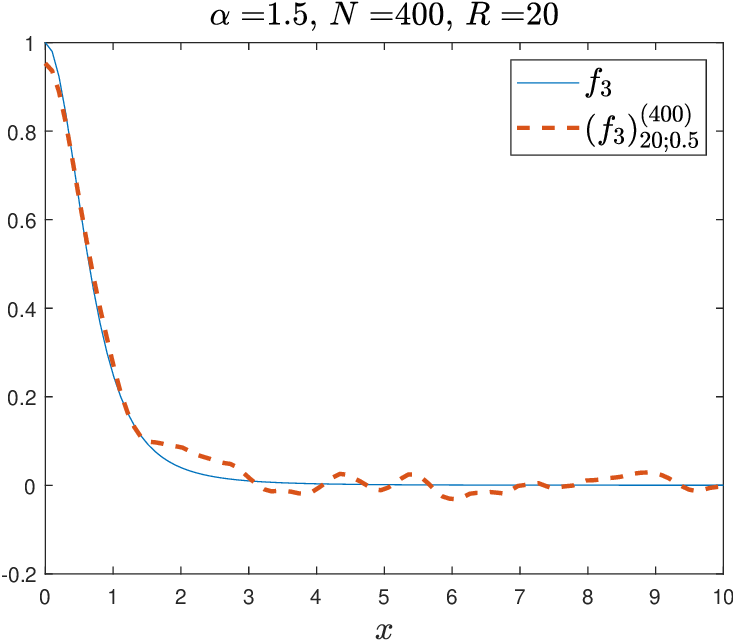}
		\caption{}
	\end{subfigure}
	\caption{\textbf{Smoothed, band-limited interpolation, Inversion.} Smoothed solution for the case $\alpha=1.5$. Inversion performed with $N=400, R=20$ and reconstruction kernel $\psi^{(1)}_{0.5}$ with $\gamma=0.5$. }
	\label{fig:fa15moll_LI}
\end{figure}

\subsubsection{The case $-1<\alpha<0$}

In the following we consider negative values $-1<\alpha<0$. Figure \ref{fig:Tfneg} shows the respective transforms for the functions of Example \ref{ex}. The transform $T_\alpha f$ is not defined at $0$, and as the parameter $\alpha$ approaches $-1$ the irregularities in the graph of $T_\alpha f$ become more sizable, which is explained by the increasing difficulty regarding the integrability of the integral kernel $\vert\sin(t/2)\vert^\alpha$.
\begin{figure}[h]
	\centering
	\begin{subfigure}{0.32\textwidth}
		\includegraphics[width=\textwidth]{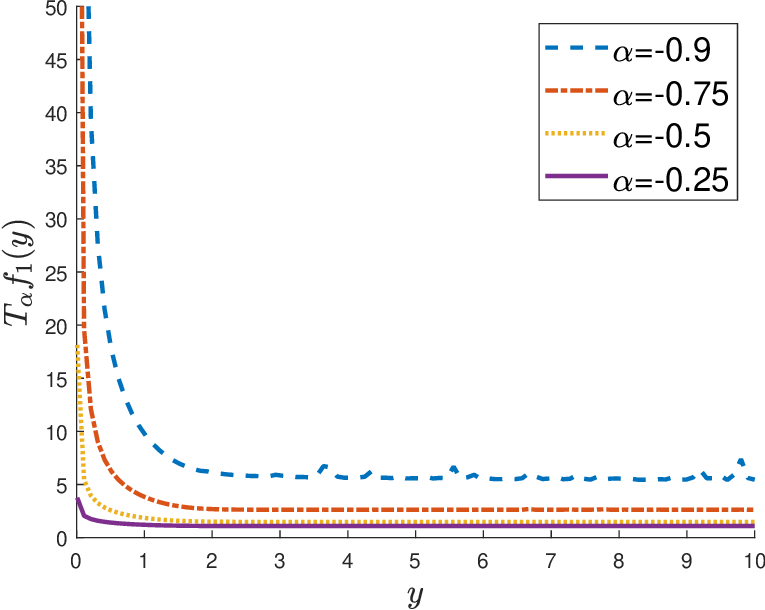}
		\caption{}
	\end{subfigure}
	\begin{subfigure}{0.32\textwidth}
		\includegraphics[width=\textwidth]{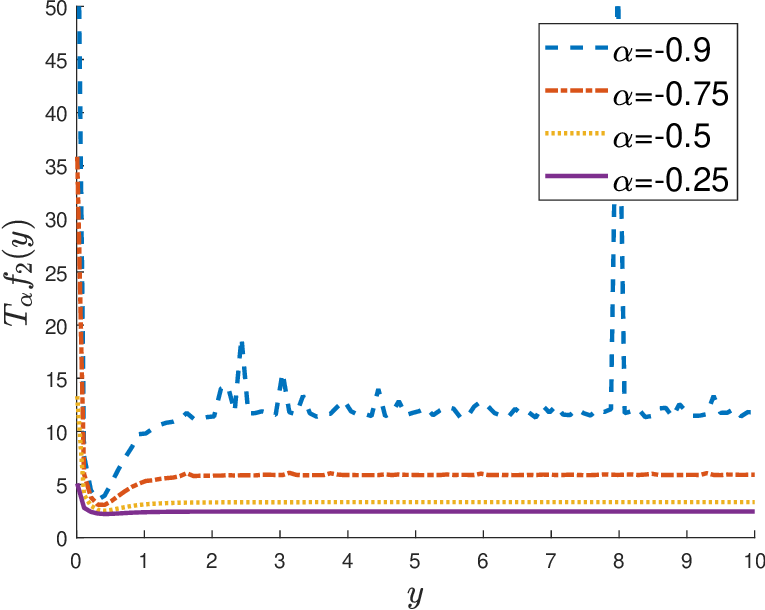}
		\caption{}
	\end{subfigure}
	\begin{subfigure}{0.32\textwidth}
		\includegraphics[width=\textwidth]{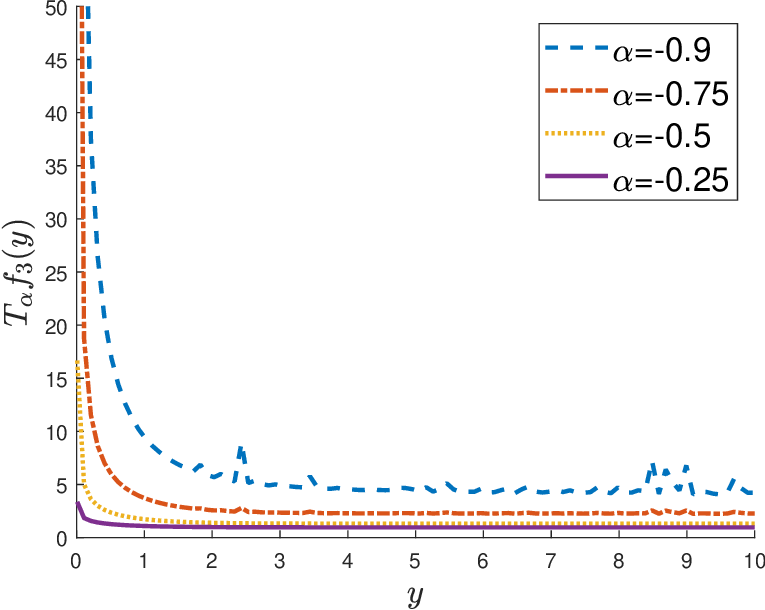}
		\caption{}
	\end{subfigure}
	\caption{\textbf{$\alpha\in(-1,0)$.} Plots of the function $T_\alpha f$ for all three examples $f_1, f_2$ and $f_3$ with $\alpha\in\{-0.9,-0.75,-0.5,-0.25\}$.}
	\label{fig:Tfneg}
\end{figure}

Figure \ref{fig:fneg} and \ref{fig:f2neg} show the inversion results of all three examples using band-limited interpolation for $\alpha=-0.5$ as well as for example $f_2$ and $\alpha\in\{-0.9,-0.75,-0.25\}$, respectively. All results were computed with $N=100, R=10$. For the case $\alpha=-0.9$, the reconstruction of $f_2$ shows larger deviations. This was to be expected looking at the irregularities of the transform $T_{-0.9}f_2$ in Figure \ref{fig:Tfneg} (b). 
For $f_1$ and $f_2$ the results for values of $\alpha$ close to $-1$ are better as the graphs of their transforms $T_\alpha f$ do not show such severe spikes. 
\begin{figure}[h]
	\centering
	\begin{subfigure}{0.32\textwidth}
		\includegraphics[width=\textwidth]{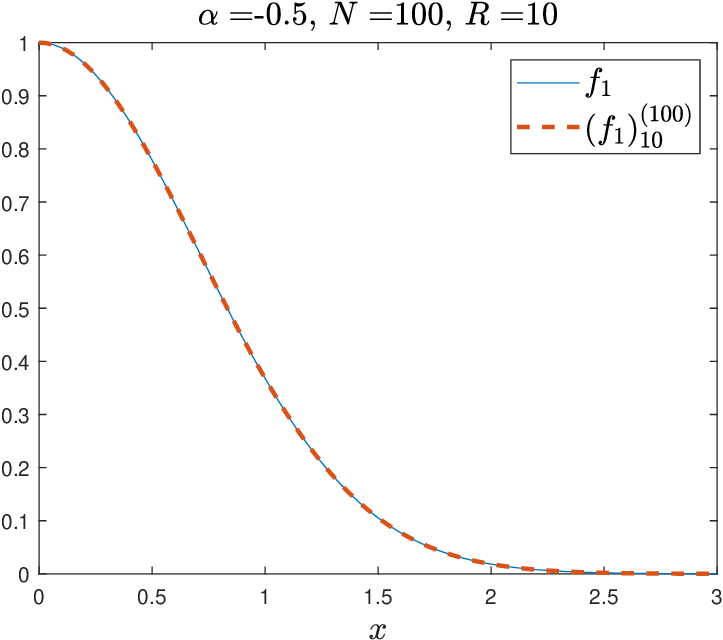}
		\caption{}
	\end{subfigure}
	\begin{subfigure}{0.32\textwidth}
		\includegraphics[width=\textwidth]{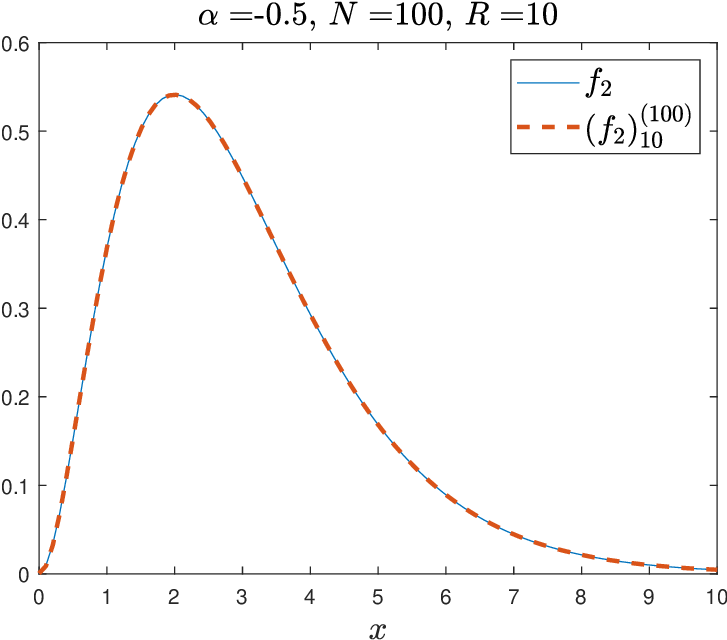}
		\caption{}
	\end{subfigure}
	\begin{subfigure}{0.32\textwidth}
		\includegraphics[width=\textwidth]{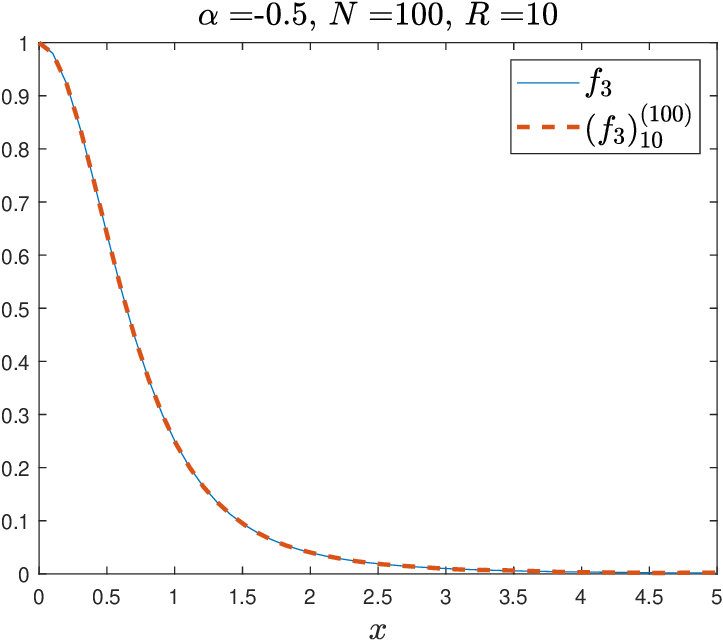}
		\caption{}
	\end{subfigure}
	\caption{\textbf{Band-limited interpolation, Inversion, $\alpha=-0.5, N=100, R=10$.} Solid blue line shows the functions $f_1, f_2$ and $f_3$.  The dashed red line shows the inversion results.}
	\label{fig:fneg}
\end{figure}
\begin{figure}[h]
	\centering
	\begin{subfigure}{0.32\textwidth}
		\includegraphics[width=\textwidth]{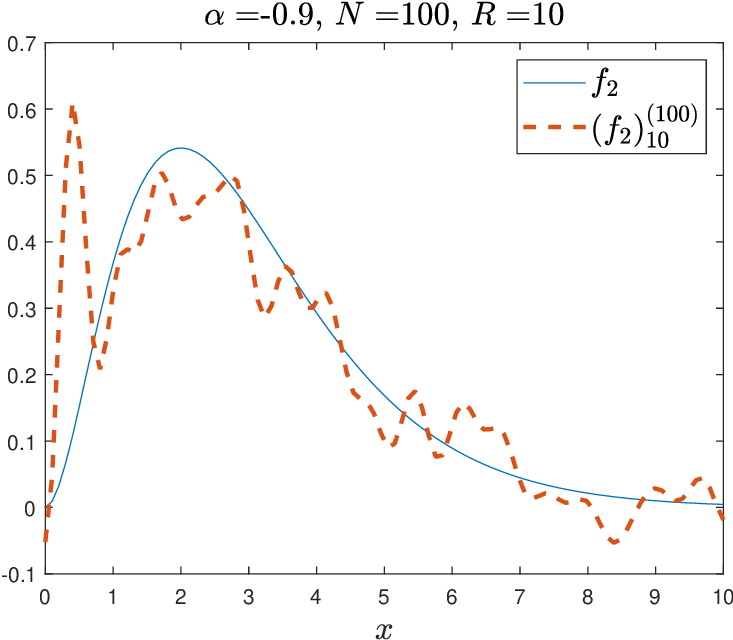}
		\caption{}
	\end{subfigure}
	\begin{subfigure}{0.32\textwidth}
		\includegraphics[width=\textwidth]{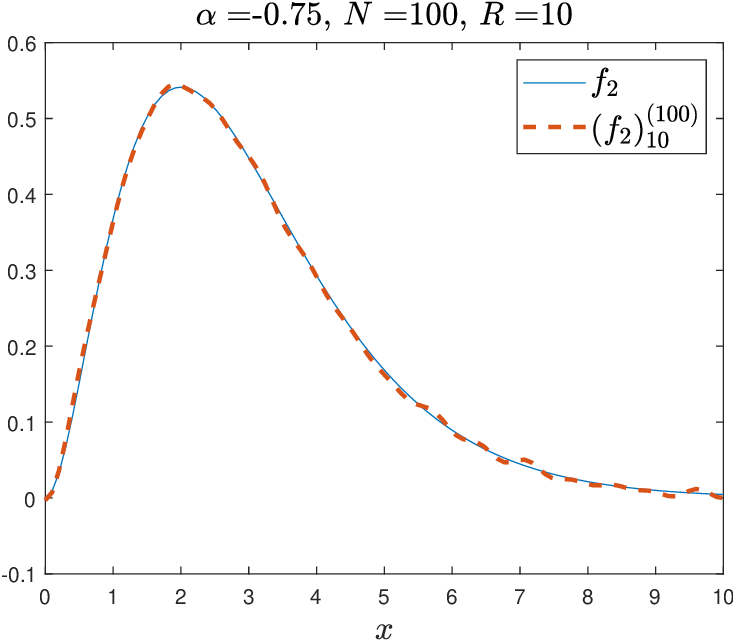}
		\caption{}
	\end{subfigure}
	\begin{subfigure}{0.32\textwidth}
		\includegraphics[width=\textwidth]{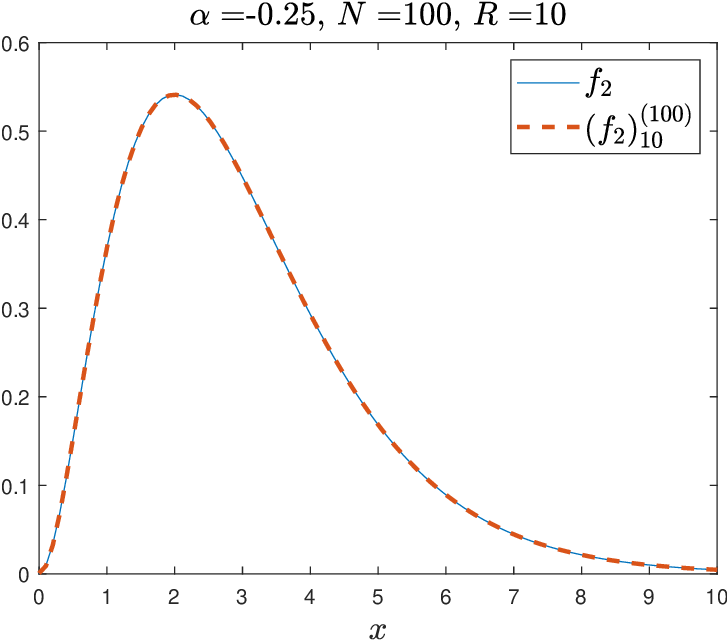}
		\caption{}
	\end{subfigure}
	\caption{\textbf{Band-limited interpolation, Inversion, $\alpha\in\{-0.9,-0.75,-0.25\}, N=100, R=10$.} Solid blue line shows the function $f_2$. The dashed red line shows the inversion results.}
	\label{fig:f2neg}
\end{figure}

\subsection{Spherical $\alpha$-cosine transform on $S^1$}
\label{numspherecos}

We consider the following examples of even probability density functions on the unit circle $S^1$: 
\begin{ex}
	\label{SphereEx}
	\begin{enumerate}[(a)]
		\item $f_1(x)=\vert\sin\left(x-h\right)\vert/4$ with $h\in[-\pi,\pi]$.
		\item $f_2(x)=I^{-1}e^{\cos(4(x-h)}$ with $h\in[-\pi,\pi]$, where $I$ is a normalizing constant. 
		\item $f_3(x;\mu,\kappa)=M\left(1/2,1,\kappa\right)^{-1}e^{\kappa\cos(x-\mu)^2}$ (\emph{two-dimensional Watson distribution}), where the normalization factor $M(1/2,d/2,\kappa)$ is the $d$-dimensional Kummer function. Samples of this distribution concentrate around $\pm\mu\in[-\pi,\pi]$ with concentration parameter $\kappa>0$.
	\end{enumerate}
\end{ex}
Their two-dimensional spherical $\alpha$-cosine transform reads $K_{\alpha,S}f(y)=\int_{-\pi}^\pi\vert\cos(y-x)\vert^\alpha f(x)dx$ for $\alpha>-1$. Applying Corollary \ref{spherecoscor}, we get an estimate $f_N$ for the density function $f$ given by
\begin{align*}
	f_N(x)=\frac{1}{2\pi}\left(1+\sum\limits_{n=-N,~n\neq0}^N\frac{\reallywidehat{K_{\alpha,S} f}(2n)}{\tilde{c}_n}e^{i2nx}\right).
\end{align*}

Figure \ref{fig:Rf} shows the spherical $\alpha$-cosine transform of all three densities $f_1,f_2,f_3$ from Example \ref{SphereEx} with various values of $\alpha>-1$. The results of the reconstruction of all three functions from their Fourier coefficients for $\alpha=1.5$, using Corollary \ref{spherecoscor}, are given in Figure \ref{fig:fRfa15}. The results inversion results for other values of $\alpha>-1$ show an identical picture to Figure \ref{fig:fRfa15}, and thus are omitted here.
\begin{figure}[h]
	\centering
	\begin{subfigure}{0.32\textwidth}
		\includegraphics[width=\textwidth]{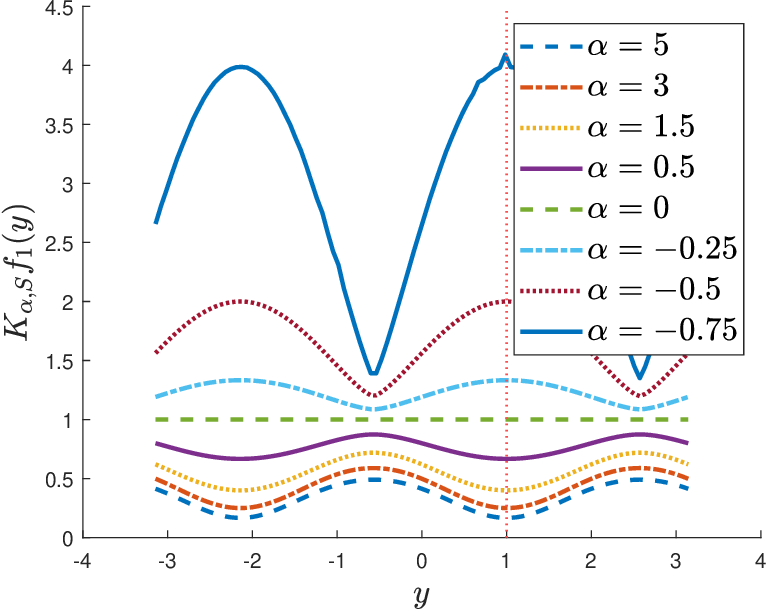}
		\caption{}
	\end{subfigure}
	\begin{subfigure}{0.32\textwidth}
		\includegraphics[width=\textwidth]{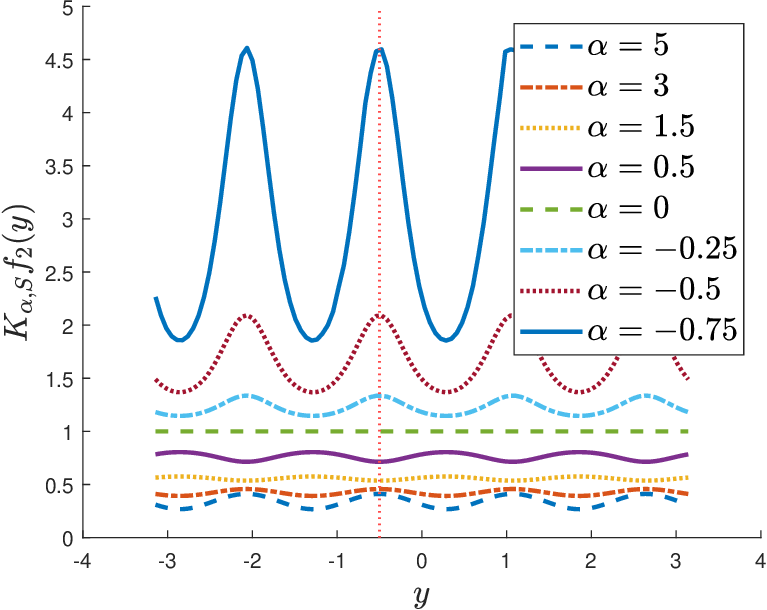}
		\caption{}
	\end{subfigure}
	\begin{subfigure}{0.32\textwidth}
		\includegraphics[width=\textwidth]{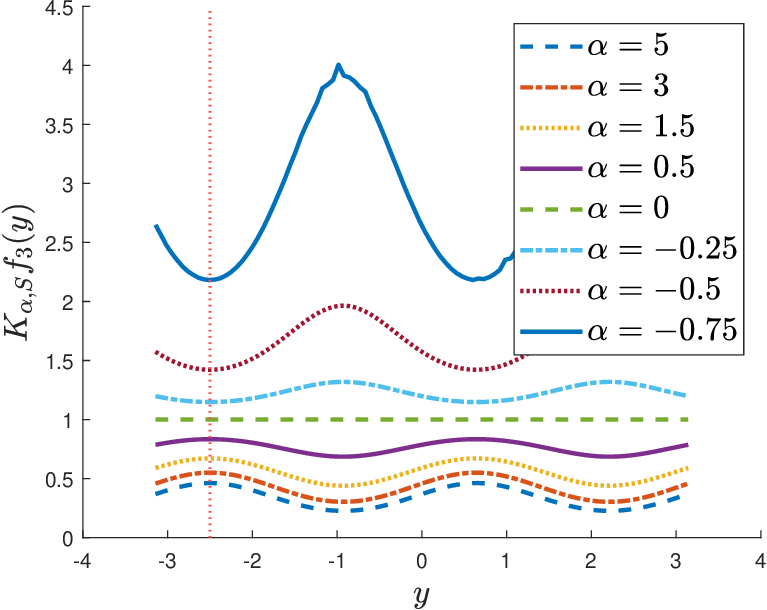}
		\caption{}
	\end{subfigure}
	\caption{\textbf{Spherical $\alpha$-cosine transform on $S^1$.} Plots of $R_\alpha f$ for the examples $f_1$ (with $h=1$), $f_2$ ($h=-0.5$) and $f_3$ (with $\mu=-2.5$, $\kappa=1$) and $\alpha\in\{5,3,1.5,0.5,0,-0.25,-0.5,-0.75\}$. The vertical red line illustrates the horizontal shift of the examples from their respective version which is even about 0.}
	\label{fig:Rf}
\end{figure}
\begin{figure}[h]
	\centering
	\begin{subfigure}{0.32\textwidth}
		\includegraphics[width=\textwidth]{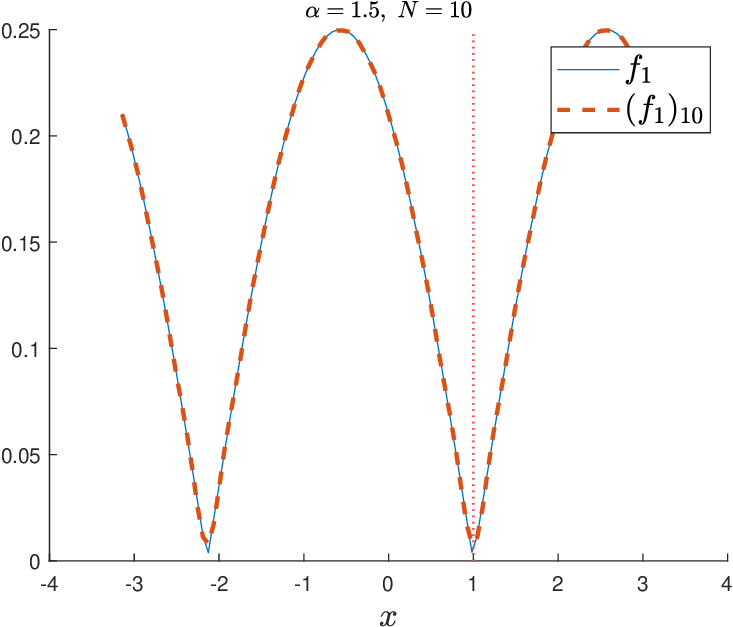}
		\caption{}
	\end{subfigure}
	\begin{subfigure}{0.32\textwidth}
		\includegraphics[width=\textwidth]{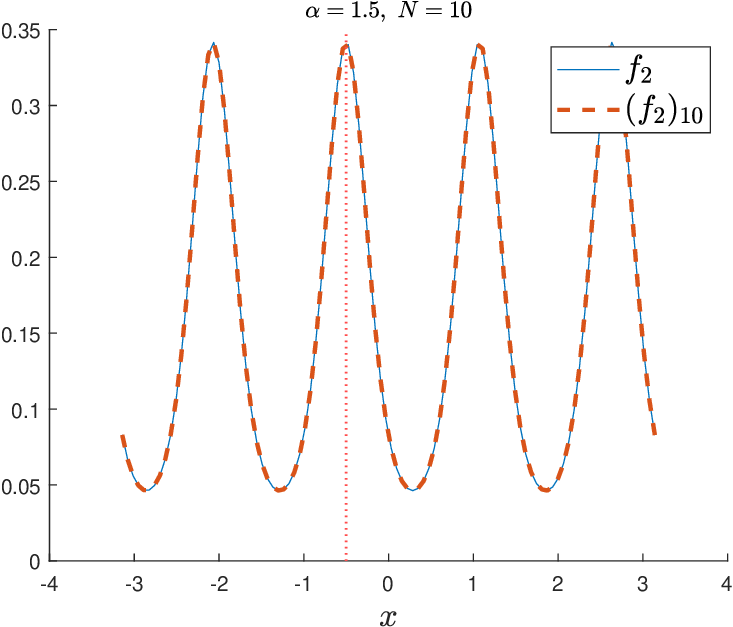}
		\caption{}
	\end{subfigure}
	\begin{subfigure}{0.32\textwidth}
		\includegraphics[width=\textwidth]{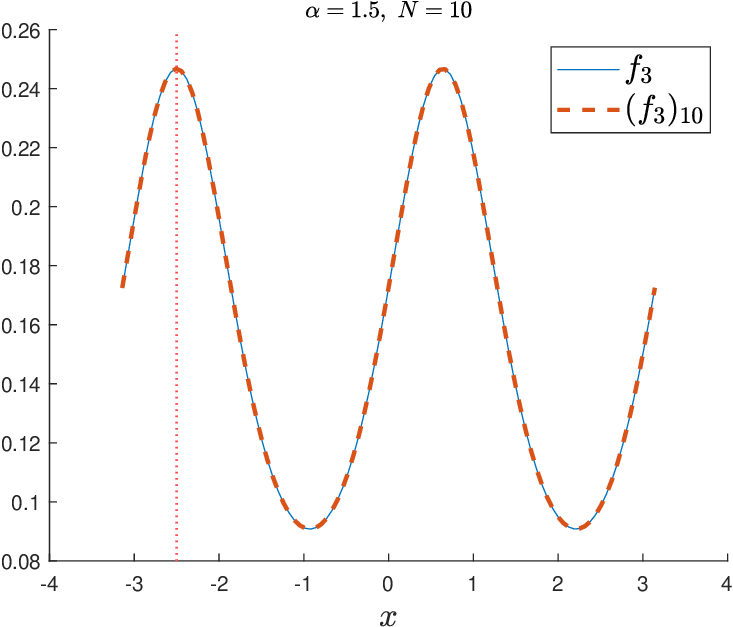}
		\caption{}
	\end{subfigure}
	\caption{\textbf{Spherical $1.5$-cosine transform on $S^1$, Inversion, $N=10$.} Inversion results for all three examples $f_1$, $f_2$ (with $h=1$) and $f_3$ (with $\mu=-2.5$, $\kappa=1$). The solid blue line shows the actual function $f$. The red dashed line shows the result of the inversion.}
	\label{fig:fRfa15}
\end{figure}

\section{Discussion}

Corollary \ref{final_cor} gives an inversion operator of $T_\alpha$ that is injective on the space of square integrable functions on the real positives with respect to a weighted Lebesgue measure. 
The direct approach 
provides a direct computational formula for the function $f$ from its $\alpha$-sine transform $T_\alpha f$, and generally works on a larger class of transformations than the presented $\alpha$-sine transform, but is lacking in terms of efficiency. The practical implementation of the underlying theory quickly proves to be quite cumbersome. 
The numerical instability of the function $\mu_+$ as well as the Fourier transformation $\mathcal{F}_+$ on the multiplicative group $(\mathbb{R}_+,\cdot)$ require tremendous computational effort.
Most importantly, the direct approach is only applicable for $\alpha>1$. 

The choice of the cut-off parameter $\varepsilon$ significantly impacts the results of the inversion. An insufficiently small choice leads to large deviations around $0$. Very small $\varepsilon$ reduce these deviations but at the price of very long computation times. Satisfying results could only be achieved when accepting several hours of computation (on a modern Intel i5 8500 6-core CPU).
In our examples, $T_\alpha f$ needed to be sampled and an interpolation was used as an approximate. A direct application of Corollary \ref{final_cor} to the integral form of $T_\alpha f$ led to no results as numerical integration failed. The step size of the sample had to be chosen as small as $10^{-6}$. Larger step sizes increased computation times enormously or ended in diverging numerical integration. See Table \ref{tab:comp} for a comparison of the direct approach to the Fourier approximation approach in terms of computation time and accuracy. 

The Fourier approximation approach is based on the series representation of $T_\alpha f$ given in Theorem \ref{prop2}, which holds for all $\alpha>-1$. The specific Fourier expansion of the integral kernels $\vert\sin(x)\vert^\alpha$ and $\vert\cos(x)\vert^\alpha$, respectively, allow for the construction of a fast and efficient approximative inversion method. 
In the case $\alpha=2$, the transform $T_2 f$ is simply a linear transformation of the Fourier transform $\mathcal{F}f$, thus inversion is achieved by a simple application of the inverse Fourier transform. 
In any other case $\alpha>-1$, the representation of $T_\alpha f$ as a series of the Fourier transform $\mathcal{F}f$, and the resulting system of linear equation from Proposition \ref{propmatrix} enable us to compute an approximate of $\mathcal{F}f$ efficiently. Applying the inverse Fourier transform allows us to estimate the function $f$.

The presented Fourier approximation method is fast and delivers accurate results for the inversion in a matter of seconds. The solution of the system of linear equations involved is easily computed since the matrix $\bm{C}$ is triangular and non-singular by construction. We have seen that the parameters $N$ and $R$ can be chosen in a practical manner. Our computations were performed with only $N=100$ equidistant samples of $Tf$ on the interval $[0,10]$. This is a huge advantage over the direct approach, where only a much smaller step size of $10^{-6}$ led to acceptable results. 

Table \ref{tab:comp} compares the numerical performance of all methods for the function $f_2$ from Example \ref{ex}(b). Linear interpolation performed slightly faster than band-limited interpolation but is less accurate, and both perform tremendously faster and are more accurate than the direct approach. Most of the error in the direct approach stems from the fluctuation of the inversion near $0$, e.g. on the interval $[0.5,R]$ the $L^2$-distance is much smaller. Still the long computation times are an immense drawback. But most importantly, the advantage of the Fourier approximation approach lies in its applicability for all $\alpha>-1$. 
\vspace{0.5cm}
\begin{table}[h]
	\renewcommand{\arraystretch}{1.5}
	\begin{center}
		\begin{tabular}{ |c|c||p{2.5cm}|p{2.5cm}| p{2.5cm}| p{2.5cm}|  }
			\hline
			\multicolumn{2}{|c||}{}&\multicolumn{2}{|c|}{$\alpha=2$}&\multicolumn{2}{|c|}{$\alpha=1.5$}\\
			\hline
			\multicolumn{2}{|c||}{\textbf{Method}}&Computation time (in sec.)&$L^2$-distance to $f_2$ on $[0,R]$&Computation time (in sec.)&$L^2$-distance to $f_2$ on $[0,R]$\\
			\hline\hline
			\multicolumn{2}{|c||}{Direct approach $(\ast)$}&$386.8818$&$4.7530\times10^{-2}$ $(\ast^2)$&$1.0711\times 10^5$&$0.5048$ $(\ast^3)$\\
			\hline
			\multirow{2}{2cm}{Fourier approximation}&band-limited&$1.2031$&$5.4255\times 10^{-4}$&$1.5785$&$5.4383\times 10^{-4}$\\
			\cline{2-6}
			&linear interp.&$0.0469$&$7.3239\times 10^{-3}$&$0.4423$&$7.3239\times 10^{-3}$\\
			\hline
		\end{tabular}
		\caption{
			Comparison between the direct approach and the Fourier approximation method (with band-limited or linear interpolation). Computations were performed for $f_2$ (cf. Example \ref{ex}(b)) with $R=10$, $N=100$. The computation time was measured for the computation of the approximate $\hat{f}_2$ at $100$ equidistantly spaced points on $[0,R]$.\\[1ex]
			$(\ast)$ For $\alpha=2$ we used $\varepsilon=0.025$ and the analytical form of $T_2f_2$ given in Example \ref{ex} (b). For $\alpha=1.5$ we set $\varepsilon=0.1$ and used a discrete sample of $T_{1.5}f_2$ with a sample step size of $10^{-6}$.\\[1ex]
			$(\ast^2)$ Computation was performed on $[0.05,R]$ only as computation times increase tremendously the closer the left integration boundary is to $0$. The error becomes larger, too.\\[1ex]
			$(\ast^3)$ Due to high computation times, we used the built-in trapezoid discretization for numerical integration.}
		\label{tab:comp}   
	\end{center}
\end{table}

When dealing with noise inflicted input data, the convolution property of the Fourier transformation enables us to compute a smoothed solution to our inversion problem. Multiplication of the interpolated approximate of the Fourier transform of $f$ with a reconstruction kernel corresponds to the convolution of $f$ with the corresponding mollifier. For our synthetically generated, noise corrupted data, satisfactory smoothed solutions are still achieved when dealing with noise that is considerably higher than the examples shown in Section \ref{noisy}. 

The closer $\alpha$ gets to $-1$, the more difficult it is to compute satisfactory inversion results, as we have seen in the case $\alpha=-0.9$ for example $f_2$. The deviation of the inversion from its target is deeply connected to the numerical integrability of the integral kernel for $\alpha$ close to $-1$. Signs of the numerical instability can be seen immediately from the spikes in the transform $T_\alpha f$. A viable solution to this can be achieved by smoothing $T_\alpha f$ as we have done in our sample with synthetically noise inflicted data before. 

Lastly, we also derived a series representation for the two-dimensional spherical $\alpha$-sine and cosine transforms with the identical coefficients as for the transforms on $\mathbb{R}_+$. This allows us establish an inversion algorithm which computes the Fourier coefficients of an even function on the unit circle from the Fourier coefficients of its two-dimensional spherical transform.
For $\alpha>-1$, $\alpha\neq0,2,4,\dots$ successful inversion results for probability density functions on the unit circle $S^1$ are achieved.



\appendix




 \bibliographystyle{model4-names}

\bibliography{bibliography}

\end{document}